\documentclass{amsart}
\usepackage{hyperref}
\usepackage{amscd, amssymb , amsmath,epsfig,subfig} %, fullpage}
\usepackage{dsfont}
\usepackage{mathrsfs}
\usepackage{epsfig,tikz-cd}
\usepackage{color}
\definecolor{brightlavender}{rgb}{0.75, 0.58, 0.89}
\definecolor{carnelian}{rgb}{0.7, 0.11, 0.11}
\definecolor{orange}{rgb}{1,0.5,0}%

\usepackage[all]{xy}
\usepackage{verbatim}
\usepackage{stmaryrd}
\usepackage{euscript}

\DeclareMathAlphabet{\matheusm}{U}{eus}{m}{n}
\newcommand{\co}{\colon}

\usepackage{tikz}
\newcommand{\any}{{\textendash}}
\newcommand{\skeq}{{\dot=}}
\newcommand{\ie}{i.e.\ }

\newcommand{\Proj}{{\operatorname{\mathsf{Proj}}}}
\newcommand{\tto}{\xrightarrow}
\newcommand{\chr}{{\mathsf{c}}}
\newcommand{\Rep}{{\mathsf{Rep}}}

\newcommand{\ptr}{\operatorname{ptr}}

\newcommand{\card}{\operatorname{card}}
\newcommand{\et}{\quad\text{and}\quad}

\newtheorem{Df}{Definition}[section]
\newtheorem{definition}[Df]{Definition}
\newtheorem{theorem}[Df]{Theorem}

\newtheorem{prop}[Df]{Proposition}
\newtheorem{proposition}[Df]{Proposition}
\newtheorem{lemma}[Df]{Lemma}

\newtheorem{remark}[Df]{Remark}

\newtheorem{corollary}[Df]{Corollary}

\newcommand{\rcoev}{\stackrel{\longrightarrow}{\operatorname{coev}}}
\newcommand{\rev}{\stackrel{\longrightarrow}{\operatorname{ev}}\!\!}
\newcommand{\lev}{\stackrel{\longleftarrow}{\operatorname{ev}}\!\!}
\newcommand{\lcoev}{\stackrel{\longleftarrow}{\operatorname{coev}}}
\newcommand{\bp}[1]{{\left(#1\right)}}

\newcommand{\bs}[1]{{\left\{#1\right\}}}
\newcommand{\qn}[1]{{\left\{#1\right\}}}
\newcommand{\ba}[1]{{\left\langle#1\right\rangle}}
\newcommand{\abs}[1]{{\left|#1\right|}}
\newcommand{\cat}{\mathscr{C}}

\newcommand{\catd}{\mathscr{D}}
\newcommand{\C}{\ensuremath{\mathbb{C}}}
\newcommand{\Z}{\ensuremath{\mathbb{Z}}}
\newcommand{\R}{\ensuremath{\mathbb{R}}}
\newcommand{\N}{\ensuremath{\mathbb{N}}}
\newcommand{\g}{\ensuremath{\mathfrak{g}}}

\newcommand{\End}{\operatorname{End}}
\newcommand{\Hom}{\operatorname{Hom}}

\newcommand{\unit}{\ensuremath{\mathds{1}}}
\newcommand{\Id}{\operatorname{Id}}

\newcommand{\qd}{\operatorname{\mathsf{d}}}

\newcommand{\Gr}{\ensuremath{\mathsf{G}}}
\newcommand{\F}{\ensuremath{\mathsf{F}}}
\newcommand{\kk}{{\mathbb K}}
\newcommand{\ve}{\varepsilon}
\newcommand{\wt}{\widetilde}
\newcommand{\wa}{\overrightarrow}
\newcommand{\wb}{\overline}
\newcommand{\ms}[1]{\mbox{\tiny\ensuremath{#1}}}
\newcommand{\vp}{\varphi}

\newcommand{\ideal}{\mathcal{I}}
\newcommand{\Jideal}{\mathcal{J}}

\newcommand{\mt}{{\operatorname{\mathsf{t}}}}
\newcommand{\mb}{{\operatorname{\mathsf{b}}}}
\newcommand{\TV}{{\operatorname{\mathsf{TV}}}}

\newcommand{\XX}{\ensuremath{\mathsf{X}}}
\newcommand{\Rib}{\operatorname{Rib}}
\newcommand{\Ladm}{{\mathcal{L}_{\ideal}}}
\newcommand{\Skein}{{\matheusm{S}}}
\newcommand{\Emb}{\mathrm{Emb}}
\newcommand{\Vect}{\mathrm{Vec}}
\newcommand{\Int}{\operatorname{Int}}
\newcommand{\Graph}{{T}}% the color will be removed. This is just to have a consistent notation for ribbon graphs.
\newcommand{\Groupoid}{\operatorname{Groupoid}}
\newcommand{\Group}{\operatorname{Group}}
\newcommand{\Set}{\operatorname{Set}}
\newcommand{\cob}{\operatorname{\textbf{Cob}}}
\newcommand{\cobnc}{\operatorname{\textbf{Cob}^{\!\mathsf{nc}}}}

\newcommand{\Sp}{\mathbb{S}}

\newcommand{\FF}{\mathcal{F}}
\newcommand{\handle}{{\mathsf M}}
\newcommand{\cyl}{{\mathsf C}}
\newcommand{\pathtoFig}{Fig/}
\usepackage[calc]{picture}\usepackage{settobox}
\newlength{\posx}
\newlength{\posy}
\newlength{\Imagew}
\newlength{\Imageh}
\newlength{\Textw}
\newlength{\Texth}
\newsavebox{\Image}
\newsavebox{\Text}

\newcommand{\epsh}[2]{
  \savebox{\Image}{\epsfig{figure=\pathtoFig#1,height=#2}}
  \settoboxwidth{\Imagew}{\Image}
  \settoboxheight{\Imageh}{\Image}
  \begin{array}{c} \hspace{-1.3mm}
    \raisebox{-4pt}{\usebox{\Image}}
    \hspace{-1.9mm}\end{array}
}

\newcommand{\putne}[3]{
  \savebox{\Text}{#3}
  \settoboxwidth{\Textw}{\Text}
  \settoboxheight{\Texth}{\Text}
  \setlength\posx{-\Imagew+0.4pt+0.01\Imagew*\real{#1}}
  \setlength\posy{-0.5\Imageh+2.5pt+0.01\Imageh*\real{#2}}
  \put(\posx,\posy){#3}
}
\newcommand{\pute}[3]{
  \savebox{\Text}{#3}
  \settoboxwidth{\Textw}{\Text}
  \settoboxheight{\Texth}{\Text}
  \setlength\posx{-\Imagew+0.4pt+0.01\Imagew*\real{#1}}
  \setlength\posy{-0.5\Imageh+2.5pt+0.01\Imageh*\real{#2}-0.5\Texth}
  \put(\posx,\posy){#3}
}
\newcommand{\putnw}[3]{
  \savebox{\Text}{#3}
  \settoboxwidth{\Textw}{\Text}
  \settoboxheight{\Texth}{\Text}
  \setlength\posx{-\Imagew+0.4pt+0.01\Imagew*\real{#1}-\Textw}
  \setlength\posy{-0.5\Imageh+2.5pt+0.01\Imageh*\real{#2}}
  \put(\posx,\posy){#3}
}
\newcommand{\putw}[3]{
  \savebox{\Text}{#3}
  \settoboxwidth{\Textw}{\Text}
  \settoboxheight{\Texth}{\Text}
  \setlength\posx{-\Imagew+0.4pt+0.01\Imagew*\real{#1}-\Textw}
  \setlength\posy{-0.5\Imageh+2.5pt+0.01\Imageh*\real{#2}-0.5\Texth}
  \put(\posx,\posy){#3}
}
\newcommand{\putc}[3]{
  \savebox{\Text}{#3}
  \settoboxwidth{\Textw}{\Text}
  \settoboxheight{\Texth}{\Text}
  \setlength\posx{-\Imagew+0.4pt+0.01\Imagew*\real{#1}-0.5\Textw}
  \setlength\posy{-0.5\Imageh+2.5pt+0.01\Imageh*\real{#2}-0.5\Texth}
  \put(\posx,\posy){#3}
}
\newcommand{\puts}[3]{
  \savebox{\Text}{#3}
  \settoboxwidth{\Textw}{\Text}
  \settoboxheight{\Texth}{\Text}
  \setlength\posx{-\Imagew+0.4pt+0.01\Imagew*\real{#1}-0.5\Textw}
  \setlength\posy{-0.5\Imageh+2.5pt+0.01\Imageh*\real{#2}-\Texth}
  \put(\posx,\posy){#3}
}
\newcommand{\putn}[3]{
  \savebox{\Text}{#3}
  \settoboxwidth{\Textw}{\Text}
  \settoboxheight{\Texth}{\Text}
  \setlength\posx{-\Imagew+0.4pt+0.01\Imagew*\real{#1}-0.5\Textw}
  \setlength\posy{-0.5\Imageh+2.5pt+0.01\Imageh*\real{#2}}
  \put(\posx,\posy){#3}
}
\newcommand{\putsw}[3]{
  \savebox{\Text}{#3}
  \settoboxwidth{\Textw}{\Text}
  \settoboxheight{\Texth}{\Text}
  \setlength\posx{-\Imagew+0.4pt+0.01\Imagew*\real{#1}-\Textw}
  \setlength\posy{-0.5\Imageh+2.5pt+0.01\Imageh*\real{#2}-\Texth}
  \put(\posx,\posy){#3}
}
\newcommand{\putse}[3]{
  \savebox{\Text}{#3}
  \settoboxwidth{\Textw}{\Text}
  \settoboxheight{\Texth}{\Text}
  \setlength\posx{-\Imagew+0.4pt+0.01\Imagew*\real{#1}}
  \setlength\posy{-0.5\Imageh+2.5pt+0.01\Imageh*\real{#2}-\Texth}
  \put(\posx,\posy){#3}
}

\newtheorem*{theorem*}{Theorem}
\begin{document}
\title{Graded spherical skein 2+1-$\Gr$-HQFT and modified Turaev-Viro invariants.}
\begin{abstract} For $\Gr$ a group, we present a $\Gr$-graded version
  of chromatic maps and skein modules and use them to define a
  2+1-$\Gr$-HQFT out of a $\Gr$-chromatic category.  The construction
  applies to the representations of unrestricted quantum groups at root
  of unity and recovers the modified Turaev-Viro 3-dimensional
  invariants.
\end{abstract}

\author[F. Costantino]{Francesco Costantino}
\address{Institut de Math\'ematiques de Toulouse\\
118 route de Narbonne\\
 Toulouse F-31062}
\email{francesco.costantino@math.univ-toulouse.fr}

\author[N. Geer]{Nathan Geer}
\address{Mathematics \& Statistics\\
  Utah State University \\
  Logan, Utah 84322, USA}
\email{nathan.geer@gmail.com}

\author[B. Ha\"{\i}oun]{Benjamin Ha\"{\i}oun}
\address{Institut de Math\'ematiques de Toulouse\\
118 route de Narbonne\\
Toulouse F-31062}
\email{benjamin.haioun@ens-lyon.fr}

\author[B. Patureau-Mirand]{Bertrand Patureau-Mirand} \address{Univ
  Bretagne Sud, CNRS UMR 6205, LMBA, F-56000 Vannes, France}
\email{bertrand.patureau@univ-ubs.fr}

\maketitle
\setcounter{tocdepth}{3}
\tableofcontents

%**************************************************\\
%To cite: \cite{MW24b}\\
%**************************************************

% Do we need spherical ?  is a left trace enough (with a left chromatic
% map) ?  More complicated to deal with index 1-2 cancellation.
\section*{Introduction}
Topological quantum field theories (TQFTs) were axiomatized by Atiyah
\cite{Ati88} as symmetric monoidal functors from the category of
$(n+1)$–dimensional cobordisms to the category of vector spaces. In
dimension three, a fundamental example is the Turaev–Viro TQFT, which
is constructed from semisimple 
%\BPm{added} 
spherical tensor categories
\cite{TV92,BW99}. This theory provides a state-sum model for
3-manifold invariants and has been extensively studied in various
contexts.  Over the past two decades, substantial progress has been
made in extending the Turaev–Viro type theories to non-semisimple
spherical type categories, see for example: \cite{AG20, CGPT20, GP13,
  GPT11}.

A related line of research concerns skein theoretic methods in quantum
topology. Skein modules provide algebraic models for spaces generated
by embedded graphs in a manifold, subject to local relations. One of
the most influential examples is the Kauffman skein algebra of a
surface, introduced independently by Przytycki \cite{Prz91,Prz99} and
Turaev \cite{Tur91}. Its definition uses relations derived from the
Jones polynomial, or equivalently the Kauffman bracket. Although this
construction is combinatorial in nature, it encodes deep geometric and
algebraic information and has been extensively studied; see for
instance \cite{BW11,BFK99, DKS25,FKL19,GJS23}.

More recently, skein-theoretic techniques have been adapted to the
non-semisimple framework and lead to the construction of TQFTs, see
\cite{GKP22,CGHP23,CGP23,CGPV23,BH26,Hai25,MSWY26}. Admissible skein modules have also been re-interpreted as the images of so-called ansular functors in \cite{MW24b}. One of the main
tools of this theory is modified traces on ideals in pivotal
categories. In \cite{CGPV23} admissible skein modules associated with
such ideals were introduced and used to construct ``non-compact''
$(2+1)$-dimensional TQFTs (i.e. TQFTs defined on the category of cobordisms such that, for each connected component, the incoming boundary is nonempty).  The purpose of this article is to
generalize those results to a $\Gr$-graded setting where $\Gr$ is a
group, not necessarily abelian. 

Before discussing this in more detail, let us recall that another natural extension of TQFTs arises from homotopical
refinements. Turaev introduced the notion of a homotopy quantum field
theory (HQFT), in which manifolds and cobordisms are equipped with
maps to a fixed target space, see \cite{Tur00, Tur10,TV12,TV14, TV17, SV23}. % \NGm{I added this reference.  If ok please remove this comment.}
 When the target is an Eilenberg–MacLane
space $K(\Gr,1)$, this structure amounts to decorating manifolds with
(non-commutative) 
%\BPm{added} 
$\Gr$-valued cohomology classes, leading to
the notion of a $\Gr$-HQFT. Such theories are closely related to
$\Gr$-graded tensor categories and provide a natural framework for
quantum invariants with homotopical data.

\subsection*{Results}
The main objective of this paper is to develop a graded version of the
skein-theoretic constructions associated to non-semisimple
$\Gr$-graded categories
% quantum field theories\BPm{sounds strange to  me}
 and to use it to produce $(2+1)$-dimensional $\Gr$-HQFTs (see Definition \ref{def:GHQFT}). Our
results can be organized into three main themes.

First, we formulate a $\Gr$-graded analogue of admissible skein
modules. This is a generalization of the main result of \cite{DST24}
to non-semisimple and non-finite cases, using rather different
techniques. For a $\Gr$-graded pivotal category equipped with an
ideal, we define skein modules generated by ribbon graphs whose
colorings respect the grading data. The resulting graded skein modules
carry natural actions of mapping class groups compatible with the
grading.  It should be noted that these graded admissible skein
modules first appeared in \cite{CGP23} and are generalizations of the
admissible skein modules introduced in \cite{CGPV23}.  However,
referees suggested that admissible skein modules should be presented
with their corresponding topological applications, so we reorganized
our presentation of these results and we will not publish \cite{CGP23}
but include its results in this paper and \cite{CGPV23} with their
corresponding topological
applications. %\cite{CGP23}\BH{That's the same ref, do you mean to put a ref where this is used for topological applications?}.
In particular, we include the $\Gr$-graded results here in full detail
since that work will not be published elsewhere.  Roughly speaking
$\Gr$-graded skein modules are defined for ``decorated surfaces'',
i.e. $3$-uples $(\Sigma,\rho,Y)$ where $Y\subset \Sigma$ is a non
empty set of base points and $\rho:\pi_1(\Sigma,Y)\to \Gr$ is a
morphism of groupoids (see Definition \ref{def:Greps}).
%\BP{Maybe replace the following with:\\ This data is equivalent to a locally flat $\Gr$-bundle $E\tto\pi \Sigma$ equipped with a trivialisation $\pi^{-1}(Y)\tto\sim Y\times\Gr$.  We describe in Theorem \ref{P:SkeinRepCobe} how gauge equivalences and changes of the trivialisation act on these skein modules.}\\
%\BP{**********}\\
This data is equivalent to a locally flat $\Gr$-bundle $E\tto\pi \Sigma$ equipped with a trivialisation $\pi^{-1}(Y)\tto\sim Y\times\Gr$.  We describe in Theorem \ref{P:SkeinRepCobe} the non trivial action of  gauge equivalences and changes of the trivialisation on these skein modules:

\begin{theorem*}[Partial restatement of Theorem \ref{P:SkeinRepCobe}]
For any equivalence between $\widetilde{\Sigma}=(\Sigma,\rho,Y)$ and $\widetilde{\Sigma}'=(\Sigma,\rho',Y')$ (as in Definition \ref{def:equivalence}) there exists a canonical isomorphism between $\Skein(\widetilde{\Sigma})$ and $\Skein(\widetilde{\Sigma}')$.%\FCm{Reformulated here and used environment. Do you agree about the "canonically" ? Commented the previous version here}\BPm{No, I don't think it is canonical.  It depends of the choice of the extension in $\Rep_\Gr(\Sigma,Y\cup Y')$ which might not be unique.  For example, it is not if $\Sigma$ is a torus or a sphere.  Without canonical, the Theorem does not say much more than ``they have the same dimension''...}
\end{theorem*}
%By the above theorem we observe that the group of self-equivalences acts on the skein modules in a non-trivial way. 
%\noindent\BP{**********}\\
%The main result of this part is Proposition \ref{P:SkeinRepCobe} which shows that the graded admissible skein modules of decorated surfaces of Definition \ref{D:sk-dec-surface} only depend on the conjugacy class of the $\Gr$ decoration $\rho\in\Hom(\pi_1(\Sigma),\Gr)$.

Second, we introduce graded versions of the algebraic structures
required for the skein-theoretic TQFT construction. In particular, we define a notion of chromatic maps in each grading component and use them to formulate the concept of a \emph{$\Gr$-chromatic category}. This is a $\Gr$-finite spherical category equipped with a non-degenerate modified trace on the ideal of projective objects
together with a chromatic map in every degree $g\in \Gr$, see Definitions \ref{D:chr} and \ref{def:newdef}. 
A central structural result is the following:
\begin{theorem*}[Proposition \ref{chrom-cat1}]
If $\cat$ is a $\Gr$-finite spherical category with modified trace $\mt$, then $\cat$ is a $\Gr$-graded chromatic category if and only if $\cat_1$ is a chromatic category.
\end{theorem*}
In particular if $\cat$ is a tensor category in the sense of \cite{EGNO15} and $\Gr$-finite spherical then it is a $\Gr$-graded chromatic category. 
%A central structural theorem  shows that the graded condition is controlled entirely by the degree-one component: a $\Gr$-finite spherical category endowed with a modified trace is $\Gr$-chromatic precisely when its degree-one component is a chromatic category, see Corollary \ref{chrom-cat1}.

Third, and most importantly, we use these structures to construct
$(2+1)$-dimensional homotopy quantum field theories:
\begin{theorem*}[Theorem \ref{T:main}]
  For any $\Gr$-chromatic category $\cat$ the $\Gr$-graded skein
  modules $\Skein$ are the vector spaces of a finite dimensional $\Gr$-HQFT defined on the category of non-compact $\Gr$-decorated
  cobordisms. Furthermore the linear maps associated to cobordisms are
  obtained by composing explicit linear maps associated to handle
  decompositions.
\end{theorem*} 
In particular the above theorem produces invariants of closed $3$-manifolds equipped with flat $\Gr$-bundles by removing a $3$-ball in each component (see Corollary \ref{cor:invariant}). 
%\FCm{Added this and commented below}
%More precisely, from any $\Gr$-chromatic category we obtain a symmetric monoidal functor from a category of non-compact $\Gr$-decorated cobordisms to vector spaces, see Theorem \ref{T:main}. 
The construction relies on a partial %\BPm{added partial because we don't explicitely give a complete set of relations} 
graded version of the
generators and relations presentation of the cobordism category
established by Juh\'{a}sz \cite{Juh18}. %\FCm{Commented here} %This allows us to extend the skein-theoretic operators arising from modified traces and chromatic maps to a particular cobordism category and thereby produce a $(2+1)$-dimensional $\Gr$-HQFT whose state spaces are described in terms of graded admissible skein modules.

Finally, we show that this framework applies to important families of
tensor categories arising in quantum algebra. In particular, when
$\cat$ is the category of representations of an unrestricted quantum
group at a root of unity equipped with its natural grading, the
resulting theory gives a simpler extension of %\BPm{changed extends with gives a simpler extension of}
the modified Turaev-Viro invariants of
$3$-manifolds of \cite{GPT11} and \cite{GP13} to a full $\Gr$-graded
TQFT (see Theorem \ref{T:TV=Skein}). 
%\FCm{Modified here and commented later: do you agree with "extend "? If yes, maybe we should use an  environment \theorem to evidence this ?}\BPm{Yes, in some sense it is not false, but there was already a full $\Gr$-graded TQFT in  \cite{GP13} for a more complicated category of cobordisms where objects are surfaces with holes}
%reproduces the modified Turaev-Viro invariants of
%$3$-manifolds, see Theorem \ref{T:TV=Skein}. 
This provides an HQFT interpretation of these
invariants and illustrates the effectiveness of the graded
skein-theoretic approach developed here.

\section{Preliminaries}
\subsection{Pivotal categories}\label{SS:LinearCat}
  Let $\cat$ be a strict monoidal category with tensor product $\otimes$ and unit object
$\unit$.  The notation $V\in \cat$
means that $V$ is an object of $\cat$.
 
The category $\cat$ is a \emph{pivotal category} if it has duality morphisms
\begin{align*} \lcoev_{V} &: 
\unit \rightarrow V\otimes V^{*} , & \lev_{V} & :  V^*\otimes V\rightarrow
\unit , \\
 \rcoev_V & : \unit \rightarrow V^{*}\otimes V, &  \rev_V & :  V\otimes V^*\rightarrow \unit
\end{align*}
which satisfy compatibility
conditions (see for example \cite{BW99, GKP13, Mal95}). In particular, the left dual and right dual $f^*: W^*\to V^*$ of a morphism $f\colon V \to W$ in $\cat$ coincide:
\begin{align*}
f^*&= (\lev_W \otimes  \Id_{V^*})(\Id_{W^*}  \otimes f \otimes \Id_{V^*})(\Id_{W^*}\otimes \lcoev_V)\\&= (\Id_{V^*} \otimes \rev_W)(\Id_{V^*} \otimes f \otimes \Id_{W^*})(\rcoev_V \otimes \Id_{W^*}).
\end{align*}
The category $\cat$ is \emph{ribbon} if it is pivotal, braided and has a twist satisfying compatibility
conditions (see for example \cite{GP18}).

\subsection{$\kk$--categories}\label{SS:KCategories}
Let $\kk$ be a field.  A \emph{$\kk$--additive category} is an
additive\footnote{Additivity is not strictly needed to define
  admissible skein modules but one can easily show that the admissible
  skein module of a category is equal to the admissible skein module
  of its additive completion.} category $\cat$ such that its hom--sets
are $\kk$--vector spaces and the composition of morphisms is
$\kk$-bilinear.    % , and the canonical $\kk$--algebra map
% $\kk \to \End_\cat(\unit), k \mapsto k \, \Id_\unit$ is an
% isomorphism.
% \NG{this is what was written in the Adm Skein paper, but it seems
%   wrong because we need monoidal to have a $\unit$.} 
A \emph{monoidal $\kk$--additive category} is a monoidal category
$\cat$ such that $\cat$ is a $\kk$--additive category 
where  the tensor product of morphisms is $\kk$-bilinear.  
  A {$\kk$--additive category} is
\begin{enumerate}
\item \emph{locally finite} if its hom--sets are finite dimensional,
\item \emph{essentially small} if it is equivalent to a subcategory
  whose collection of objects and morphisms form a set. 
\end{enumerate}
\begin{definition} A \emph{monoidal $\kk$-category} is an essentially small, locally
finite, monoidal $\kk$-additive category such that
$\End_\cat(\unit)=\kk\Id_\unit$.  A \emph{pivotal $\kk$-category} is a
monoidal $\kk$-category which is pivotal.
\end{definition}

An object $V$ in a $\kk$--additive category is \emph{null} if $\Id_V=0$.
We say a category is \emph{null}, when all its objects are null.  An object
$W$ of a $\kk$-additive category $\cat$ is a \emph{generator of
  $\cat$} if the identity of any object of $\cat$ factors through a
finite direct sum $W^{\oplus n}$ of copies of $W$.

\begin{definition}\label{def:sharp}
  We say a pivotal $\kk$-category $(\cat,\otimes,\unit)$ is \emph{sharp} if
 %N: its unit is simple: 
 any non-zero morphism to $\unit$ is an
  epimorphism and any non-zero epimorphism from $\unit$ is an
  isomorphism.

  % \begin{enumerate}
  % \item $\unit$ is absolutely simple: $\End_\cat(\unit)=\kk$;
  % \item $\unit$ is simple: any non-zero morphism to $\unit$ is an
  %   epimorphism and any non-zero epimorphism from $\unit$ is an
  %   isomorphism.
  % \end{enumerate}
\end{definition}
In a sharp pivotal $\kk$-category, the tensor product of two non-null
objects is non-null.  A tensor category in the sense of \cite{EGNO15} is sharp.

%\FC{added} We shall say that an object is \emph{null} if its identity morphism is $0$. Remark that a null object is both initial and terminal and any two null objects are isomorphic via a unique isomorphism. The ideal generated by a null object is formed by the null objects and is clearly contained in any ideal. 

\subsection{$\Gr$-graded pivotal $\kk$-category}%$\kk$-linear 
Let $\cat$ be a sharp pivotal $\kk$-category and $\Gr$ be a group.  The
category $\cat$ is \emph{$\Gr$-graded} if % it is sharp and
%\NGm{I moved sharp here from the line above.  If ok please remove this comment.}\BPm{I put it back because sharp sound strange in the def of graded} \NGm{the way it is worded, I had to read it very carefully to make sure the definition requires sharp.  I was a little confused when first reading this.} 
there exists a family
$\bs{\cat_g}_{g\in\Gr}$ of non-null full subcategories whose objects
are called \emph{homogeneous} such that
\begin{enumerate}
\item each object of $\cat$ is a direct sum of homegeneous objects,
\item if $g_1,g_2\in\Gr$, $V_1\in\cat_{g_1}$ and $V_2\in\cat_{g_2}$ then
  \begin{enumerate}
  \item if an object of $\cat$ is isomorphic to $V_1$, then it is an
    object of $\cat_{g_1}$,
  \item $V_1\otimes V_2\in\cat_{g_1g_2}$,
  \item if $g_1\neq g_2$ then $\Hom_\cat(V_1, V_2)=\bs0$.
  \end{enumerate}
\end{enumerate}
It follows from the definition that $\unit\in\cat_1$. Also, for
$V\in\cat_g$ we have $V^*\in\cat_{g^{-1}}$.  If $V$ is non-null, $g$
is uniquely determined by $V$ and is called its \emph{degree}.

\subsection{M-traces on ideals in pivotal categories}\label{SS:trace}
Let $\cat$ be a pivotal $\kk$--category.  Here we recall the definition
of an m-trace on an ideal in $\cat$, for more details see
\cite{GKP13,GPV13}.  By an \emph{ideal} of $\cat$ we mean a full
subcategory, $\ideal$, of~$\cat$ which is
\begin{description}
\item[Closed under tensor product]
 If $V\in \ideal$ and
  $W\in \cat$, then $V\otimes W$ and $W \otimes V $ are objects of
  $\ideal$.
\item[Closed under retractions] If $V\in \ideal$, $W\in \cat$
  and there are morphisms $f:W\to V$, $g:V\to W$ such that
  $g f=\Id_W$, then $W\in \ideal$ (we say $W$ is a \emph{retract} of $V$).
\end{description}

As an example, the full subcategory $\Proj$ of projective objects of $\cat$ is an ideal.  
\begin{remark}\label{rk:ideal-gen}
  Since the intersection of ideals is an ideal, we can define the \emph{ideal
  generated} by a collection of objects as the smallest ideal that
  contains these objects.
\end{remark}
Given an object $V$ of $\cat$ the \emph{ideal generated by $V$} is set
of all objects $U$ which are a retract of $W\otimes V\otimes X$ for
some $W,X\in \cat$.  
%\NGm{added this sentence. Bertrand, if ok please  remove this comment.}\BPm{Then I move the remark about ideal  generated before}

%\NGm{I moved this lemma from above because we had not defined an ideal yet.  I also reworded it a little, left the old commented.  Bertrand,  if ok please remove this comment.}
\begin{lemma}\label{L:V-appear}
 %N:  Let $\cat$ be a sharp $\kk$-linear pivotal category,
  Let $\cat$ be a sharp pivotal $\kk$-category and
 %
  % N: let $V$ be a non null object and let $P$ be a projective object
  % of $\cat$, then the ideal generated by $V$ contains $\Proj$ and
  let $V$ be a non-null object then the ideal generated by $V$
  contains the ideal $\Proj$ of projective objects of $\cat$.
  Moreover, if $P$ is a projective object of $\cat$, there exist
  morphisms $e_{V,P}\in\End_\cat(V^*\otimes P)$ and
  $e'_{V,P}\in\End_\cat(V\otimes P)$ such that
  $\ptr^L_{V^*}(e_{V,P})=\ptr^L_{V}(e'_{V,P})=\Id_P$ (see Figure
  \ref{fig:V-appear0} for a graphical interpretation).
\end{lemma}
\begin{figure}
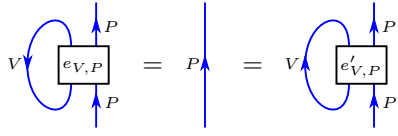

\[\epsh{fig2l.pdf}{11ex}
\putc{70}{50}{$\ms{e_{V,P}}$}
\pute{93}{81}{$\ms{P}$}
\pute{94}{20}{$\ms{P}$}
\putw{0}{51}{$\ms{V}$}
\quad=\quad
\epsh{fig2i.pdf}{11ex}
\putw{0}{51}{$\ms{P}$}
\quad=\quad
\epsh{fig2u.pdf}{11ex}
\putc{70}{52}{$\ms{e'_{V,P}}$}
\pute{93}{81}{$\ms{P}$}
\pute{94}{20}{$\ms{P}$}
\putw{0}{51}{$\ms{V}$}
\]
\caption{The morphisms $e_{V,P}$ and $e'_{V,P}$.}
\label{fig:V-appear0}
\end{figure}
\begin{proof}
  Let $\ideal$ be the ideal generated by $V$. Since $V$ is not null
  then $\rev_V:V\otimes V^*\to \unit$ is an epimorphism and so for
  each projective object $P$, the morphism
  $\rev_V\otimes \Id_P : V\otimes V^*\otimes P\to P$ is an
  epimorphism, therefore it splits. This shows that $P$ is a retract
  of $V\otimes W$ for some $W\in \cat$ (explicitly, $W=V^*\otimes P$)
  and so $\Proj\subset \ideal$. 
  
  Now let $f:P\to V\otimes W$ and
  $g:V\otimes W\to P$ be such that $\Id_P=g\circ f$.  Then let
  $e_{V,P}=(\Id_{V^*}\otimes g)\circ (\rcoev_{V^*}\otimes \Id_P)\circ
  (\rev_{V}\otimes \Id_W)\circ (\Id_{V^*}\otimes f)$.  By applying the
  ``snake'' (or ``zig-zag'') relation twice on dualities, one gets that
  $\ptr^L_{V^*}(e_{V,P})=g\circ f=\Id_P$.  The proof for
  $e'_{V,P}$ is similar.% Now let $f_i:P\to V\otimes W$ and
  % $g_i:V\otimes W\to P$ be such that $\Id_P=\sum g_i\circ f_i$.\BPm{I
  %   think one $f$ and one $g$ might be enough...} Then let
  % $e_{V,P}=\sum_i (\Id_{V^*}\otimes g_i)\circ (\rcoev_{V^*}\otimes
  % \Id_P)\circ (\rev_{V}\otimes \Id_W)\circ (\Id_{V^*}\otimes f_i)$.
  % By applying the snake relation twice on dualities, one gets that
  % $\ptr^L_{V^*}(e_{V,P})=\sum_{i} g_i\circ f_i=\Id_P$.  The proof for
  % $e'_{V,P}$ is similar.
\end{proof}
%\NG{need to modify  the above to be consistent with sharp below.  Also need to figure  out how what terms we use involving: sharp, $\kk$--category and  $\kk$-linear.  We can do this once we have things a little more  figured out, Bertrand wanted to wait. }\BP{I made a try}

A \emph{right (resp.\ left) partial trace} of $f\in \End_\cat(U\otimes W)$ is the morphism
\begin{align*}
  \ptr_R^W(f)&=(\Id_U \otimes \rev_W)(f\otimes \Id_{W^*})(\Id_U \otimes \lcoev_W)\in \End_\cat(U), \\
 \text{ resp.\ } \ptr_L^U(f)&=
( \lev_U\otimes \Id_W )( \Id_{U^*} \otimes f)( \rcoev_U\otimes \Id_W)\in \End_\cat(W).
\end{align*}
A \emph{right m-trace} on an ideal $\ideal$ is a family of linear functions
$\{\mt_V:\End_\cat(V)\rightarrow \kk \}_{V\in \ideal}$
such that:
\begin{description}
\item[Cyclicity] 
 If $U,V\in \ideal$ and $f:V\rightarrow U,$ $g:U\rightarrow V$ are any morphisms in $\cat$ then
$
\mt_V(g f)=\mt_U(f  g).$
\item[Right partial trace property]  If $U\in \ideal$, $W\in \cat$, $f\in \End_\cat(U\otimes W)$ then
$
\mt_{U\otimes W}\bp{f}
=\mt_U(\ptr_R^W(f)).
$
\end{description}
A \emph{left m-trace} on $\ideal$ is a family of linear functions $\{\mt_V:\End_\cat(V)\rightarrow \kk \}_{V\in \ideal}$ which is cyclic and  satisfies
\begin{description}
 \item[Left partial trace property] 
  If $U\in \ideal$, $W\in \cat$, $g\in \End_\cat(W\otimes U)$ then
$
\mt_{W\otimes U}\bp{g}
=\mt_U(\ptr_L^W(g)).
$
\end{description}
A \emph{m-trace} on an ideal is a right m-trace which is also a left m-trace.

\begin{definition}\label{D:spherical}%\BPm{added}\NGm{ok}
  A m-trace $\mt$ on $\ideal$ is \emph{non-degenerate} if for any
  $V\in\ideal$, the pairing
  $\Hom_{\cat}(\unit,V)\otimes_\kk\Hom_{\cat}(V,\unit)\to\kk$ defined
  by $x\otimes y\mapsto \mt(x\circ y)$ is non-degenerate.  A
  \emph{spherical category} (over $\kk$) is a sharp pivotal
  $\kk$-category which has a non-degenerate m-trace on its ideal
  $\Proj$ of projective objects.
\end{definition}
Remark that in a spherical category, any non-zero
m-trace on $\Proj$ is necessarily non degenerate (see \cite[Corollary
4.3.44]{GP25}).
\subsection{Invariants of colored ribbon graphs}
Let $n=2$ or $3$ and $\cat$ be a pivotal $\kk$--category.  If $n=3$,
we assume that $\cat$ is ribbon.  Next, we recall the definition of
$\cat$--colored ribbon graphs, for details see \cite{Tur94}.

A morphism
$f:V_1\otimes{\cdots}\otimes V_n \rightarrow W_1\otimes{\cdots}\otimes
W_m$ in $\cat$ can be represented by a box and arrows, which are
called {\it coupons}, see graph on the right:
$$ 
\epsh{fig32}{9ex}\putw{0}{18}{$\ms{V_1}$}\pute{100}{18}{$\ms{V_n}$}\putw{0}{82}{$\ms{W_1}$}\pute{100}{82}{$\ms{W_n}$}\putc{53}{10}{{\footnotesize $\cdots$}}\putc{53}{90}{{\footnotesize $\cdots$}}\putc{50}{50}{{\footnotesize $f$}}\hspace{10pt}.
$$
All coupons have top and bottom sides which in our pictures will be
the horizontal sides of the coupons.  By a ribbon graph in an oriented
manifold $\Sigma$, we mean an oriented compact surface embedded in
$\Sigma$ which is decomposed into elementary pieces: bands, annuli,
and coupons (see \cite{Tur94}) and is a thickening of an oriented
graph.  In particular, the vertices of the graph lying in
$\Int \Sigma=\Sigma\setminus\partial \Sigma$ are thickened to coupons.
A $\cat$--colored ribbon graph is a ribbon graph whose (thickened)
edges are colored by objects of $\cat$ and whose coupons are colored
by compatible morphisms of~$\cat$.  The colors of the bottom edges
of a coupon determine the source object of the morphism and its top
edges determine the target object of the morphism. 
%N: Thanks: \BPm{I developped  the compatibility of colorings} 
 The intersection of a
$\cat$--colored ribbon graph in $\Sigma$ with $\partial \Sigma$ is
required to be empty or to consist only of vertices of valency~1.
When $\Sigma$ is a surface the ribbon graph is just a tubular
neighborhood of the graph.

%N: Thanks: \BPm{here we only defined the pivotal functor $\F$ for graph in  $\R\times[0,1]$.  We also need the ribbon functor $\F$ for graph in  $\R^2\times[0,1]$. Maybe use $\F:\Rib^n_\cat\to\cat$ where $n=2,3$.}
Let $\Rib^n_\cat$ be the category of $\cat$--colored ribbon graphs in
$\R^{n-1}\times[0,1]$ and let $$\F:\Rib^n_\cat\to \cat$$ be the
Reshetikhin-Turaev functor associated with $\cat$ via the Penrose
graphical calculus, see for example \cite{GPV13,Tur94,GP25}.

Let $\Ladm$ be the class of all $\cat$--colored ribbon graphs in
the sphere $S^n$ obtained as the braid closure of a (1,1)-ribbon graph
$\Graph_V$ whose open edge is colored with an object $V\in\ideal$.  Given
an m-trace $\mt$ on $\ideal$ we can renormalize $\F$ to an invariant
\begin{equation}\label{E:DefF'}
\F':\Ladm\to \kk \text{ given by } \F'(L)=\mt_V(\F(\Graph_V))
\end{equation}
where $\Graph_V$ is any (1,1)-ribbon graph whose closure is
$L\in \Ladm$.  The properties of the m-trace imply that $\F'$ is a
isotopy invariant of $L$, see \cite{GPV13}.

% A $\cat$--colored ribbon graph in $\R^2$ (resp.\
% $S^{2}=\R^2\cup\{\infty\}$) is called \emph{planar} (resp.\
% \emph{spherical}).  If $n=2$, let $\Ladm$ be the class of all spherical
% $\cat$--colored ribbon graphs obtained as the braid closure of a
% (1,1)-ribbon graph $T_V$ whose open edge is colored with an object
% $V\in\ideal$.  Given an m-trace $\mt$ on $\ideal$ we can renormalize $\F$
% to an invariant
% \begin{equation}\label{E:DefF'}
% F':\Ladm\to \kk \text{ given by } F'(L)=\mt_V(F(T_V))
% \end{equation}
% where $T_V$ is any (1,1)-ribbon graph whose closure is $L\in \Ladm$.  The properties of the m-trace imply that $F'$ is a isotopy invariant of $L$, see \cite{GPV13}.  

% If $n=3$ and $\cat$ is a ribbon category then let $\Ladm$ be the
% class of all $\cat$--colored ribbon graphs in $\R^3$ obtained as the
% braid closure of a (1,1)-ribbon graph $T_V$ whose open edge is colored
% with an object $V\in\ideal$.  Then the assignment $F'$ given in
% Equation \eqref{E:DefF'}, where $F$ is the R-T functor, is an isotopy
% invariant of $L$ in $\R^3$.
\section{$\Gr$-decorated cobordisms}
We introduce a cobordism category tailored to the study of the
character variety of 3-manifolds, equivalently the moduli space of
flat $\Gr$-bundles over 3-manifolds.
\subsection{$\Gr$-representations}
If $X$ is a nonempty locally path connected topological space %\NGm{In the topological paper we say "nonempty locally path connected topological space" do we need this here?}\BPm{I add.  It might not be necessary but we are clearly not interested in the other cases} 
and 
%N: $Y\subset X$, 
$Y\subset X$ is a subset, 
the \emph{fundamental groupoid} $\pi_1(X,Y)$ is the groupoid (\ie the small category in which all morphisms are isomorphisms) 
whose set of objects is 
$Y$ and morphisms are the homotopy class of paths from one point of
$Y$ to another (where homotopies fix extremities).  Following
Whitehead \cite{Whi78}, such a path $\gamma$ is a morphism from
$\gamma(1)$ to $\gamma(0)$. This makes the composition compatible with
the usual concatenation of paths (\ie $\gamma\circ\gamma'$ is the
path following first $\gamma$ then $\gamma'$).  We write $\gamma:p\leadsto q$ if $\gamma$ is a morphism from $q=\gamma(1)$ to $p=\gamma(0)$.
%N:  added
When $Y$  is a single point $*$ of $X$ then $\pi_1(X,Y)=\pi_1(X,\{*\})$ is the usual fundamental group.

\begin{definition}\label{def:Greps}
The set of \emph{$\Gr$-representations} of a connected topological space $X$ with set of basepoints $Y$ is
%N: $$\Rep_\Gr(X,Y)=\Hom(\pi_1(X,Y),\Gr)\et \Rep_\Gr(X,\emptyset)=\Hom(\pi_1(X,\bs*),\Gr)/\Gr$$ 
$$\Rep_\Gr(X,Y)=\Hom_{\Groupoid}(\pi_1(X,Y),\Gr)$$
if $Y$ is non-empty and 
$$
 \Rep_\Gr(X,\emptyset)=\Hom_{\Group}(\pi_1(X,\bs*),\Gr)/\Gr$$ 
 if $Y$ is empty, 
where $*$ is
any base point of $X$ and the group $\Gr$ acts by conjugation on
$\Hom_{\Group}(\pi_1(X,\bs*),\Gr)$. If $X$ is not connected we define $$ \Rep_\Gr(X,Y) = \prod_{X_0\in\pi_0(X)} \Rep_\Gr(X_0,Y\cap X_0)\ . $$ By convention,
$\Rep_\Gr(\emptyset,\emptyset)$ is a singleton.
\end{definition}

If $X\subset X'$ and $Y\subset Y'$, there is a restriction map
$\Rep_\Gr(X',Y')\to\Rep_\Gr(X,Y)$.
Remark that a representation always extends when one increases the set of base
  points: If $Y\subset Y'$, the restriction map
  $\Rep_\Gr(X,Y')\to\Rep_\Gr(X,Y)$ is surjective.
Also there is a left action of 
%N:  $\Gr^Y\curvearrowright \Rep_\Gr(X,Y)$ given for $f\in\Gr^Y$ and $\rho\in\Rep_\Gr(X,Y)$ of $\Gr^Y$  by $$(f.\rho)(\gamma)=f(\gamma(0))\rho(\gamma)f(\gamma(1))^{-1}$$
$\Gr^Y=\Hom_{\Set}(Y,\Gr)$ on $ \Rep_\Gr(X,Y)$ 
given by
$$(\vp.\rho)(\gamma)=\vp(\gamma(0))\rho(\gamma)\vp(\gamma(1))^{-1}$$
for $\vp\in\Gr^Y$ and
$\rho\in\Rep_\Gr(X,Y)$.
\begin{lemma}\label{L:Rep}
  Let $X$ be a path connected topological space, and $Y\subset X$ with
  $\card(Y)=n\in\N^*$.  Let $\Gamma$ be a rooted
  directed tree mapped in $X$ with vertices $Y$ and root $y_0\in Y$.
  Then
  %N:  $\Rep_\Gr(X,Y)\simeq \Rep_\Gr(X,\bs{y_0})\times \Gr^{\Gamma_1}$ 
 there is an isomorphism of sets 
 $$\Rep_\Gr(X,Y)\to \Rep_\Gr(X,\bs{y_0})\times \Gr^{\Gamma_1}$$ 
  where $\Gamma_1$ is the set of $n-1$ oriented edges of $\Gamma$.
\end{lemma}
\begin{proof}
  The graph $\Gamma$ gives for every pair of base points $y,y'\in Y$ a
  prefered path $\gamma_{y,y'}:y\leadsto y'$ given by a concatenation
  of edges of $\Gamma_1$.  Then any path $\sigma_0\in\pi_1(X,Y)$ from
  $y_1$ to $y_2$ can be written
  $\sigma_0=\gamma_{y_1,y_0}(\gamma_{y_0,y_1}\sigma_0\gamma_{y_2,y_0})\gamma_{y_0,y_2}$
  with $\gamma_{y_0,y_1}\sigma\gamma_{y_2,y_0}\in\pi_1(X,y_0)$ and 
  any homotopy $\bs{\sigma_t}_{t\in[0,1]}$ between $\sigma_0$ and
  $\sigma_1$ is conjugated to a relation
  $\gamma_{y_0,y_1}\sigma_t\gamma_{y_2,y_0}$ in $\pi_1(X,y_0)$.
\end{proof}
\subsection{The category of $\Gr$-decorated cobordisms}\label{def:Gdecorated}
A \emph{$\Gr$-decorated surface} is a 3-tuple $(\Sigma, Y, \rho)$ where
$\Sigma$ is a closed surface (\ie a smooth compact oriented surface
without boundary), $Y$ is a finite set of base points of $\Sigma$ such
that there is at least one base point in each component of $\Sigma$,
and $\rho\in\Rep_\Gr(\Sigma,Y)$ is a $\Gr$-representation.
 
A \emph{$\Gr$-decorated 3-manifold} $\wt M$ %\FCm{Do we accept closed $3$-manifolds ?}\BPm{Yes, with no base points}  
is a (possibly closed) compact oriented 3-manifold $M$
equipped with a set $Y\subset\partial M$ of base points, and a
representation $\rho\in\Rep_\Gr(M,Y)$ such that the restriction of
$\rho$ to its boundary gives a $\Gr$-decorated surface
$(\partial M,Y,\rho\vert_{\partial M})$.

A \emph{$\Gr$-decorated cobordism}
$\wt M=(M,Y,\rho):(\Sigma_-,Y_-,\rho_-)\to(\Sigma_+,Y_+,\rho_+)$ is
the datum of a $\Gr$-decorated 3-manifold $(M,Y,\rho)$, two $\Gr$-decorated
surfaces $(\Sigma_-,Y_-,\rho_-)$, $(\Sigma_+,Y_+,\rho_+)$, an
orientation preserving diffeomorphism
$f:\Sigma_+\sqcup(-\Sigma_-)\overset\sim\to\partial M$ such that
$Y_\pm=Y\cap f(\Sigma_\pm)$ and $\rho_\pm$ is the restriction of
$\rho\circ f$ to $(\Sigma_\pm,Y_\pm)$. A \emph{diffeomorphism of $\Gr$-decorated
cobordisms} $\psi:(M,Y,\rho)\to (M',Y',\rho')$ is an orientation
preserving diffeomorphism $\psi:M\to M'$ which induces the identity on the boundary, and such that $\psi^*(\rho')=\rho$.

The composition of two $\Gr$-decorated cobordisms is given by gluing them along their common $\Gr$-decorated surface, equipping the result with the unique representation given by Van Kampen Theorem, and erasing the base points contained in the common surface.

\begin{definition}
  The \emph{category of $\Gr$-decorated (2+1)-cobordisms} $\wt\cob$ is
  the category whose objects are $\Gr$-decorated surfaces and
  morphisms are diffeomorphism classes of $\Gr$-decorated
  cobordisms. It is a symmetric monoidal category whose tensor product
  is the disjoint union.
\end{definition}

\begin{definition}\label{def:GHQFT}
A 3-dimensional $\Gr$-HQFT is a symmetric monoidal functor $$\mathcal{Z}: \wt\cob \to \operatorname{Vect} \ .$$
\end{definition}

\section{(Graded) admissible skein modules}
In this section, we recall and adapt material from the unpublished manuscript \cite{CGP23}
concerning (graded) admissible skein modules. These modules generalize classical skein modules by imposing additional algebraic constraints. We present a self-contained treatment of their definition and basic properties.

\subsection{Definition of admissible skein modules}
Let $\cat$ be a % n essentially small
pivotal $\kk$--category with an ideal $\ideal$.
Let $M$ be a
oriented manifold of dimension $n=2$ or $3$. %two or three. 
In this paper all manifolds will be oriented and if we consider the skein module of a three dimensional manifold we require $\cat$ to be ribbon.

An \emph{$\ideal$--admissible ribbon graph} in $M$ is a
$\cat$--colored ribbon graph in $M$ where each connected component of
$M$ contains at least one edge colored with an object in $\ideal$.

A \emph{box} in $M$ is an orientation preserving embedding
$f: \R^{n-1}\times [0,1] \to M$.
\begin{definition}\label{D:ISkeinRelation}
  Given $\Graph_1,\,\Graph_2,\ldots,\Graph_m$ ribbon $\cat$-graphs in
  $M$ and $a_1,\ldots,a_m\in\kk$ the linear combination
  $a_1\Graph_1+a_2\Graph_2+\cdots a_m\Graph_m$ is an
  \emph{$\ideal$-skein relation} (in $M$) if there is a box
  $f:\R^{n-1}\times [0,1] \to M$ such that
  \begin{enumerate}
  \item $f^{-1}(\Graph_i)$ represents a morphism in $\Rib^n_\cat$ for all
    ${i}={1},\dots,{m}$,
  \item the $\Graph_i$'s coincide outside the box:  
  $$\Graph_i \cap \big(M\setminus f(\R^{n-1}\times [0,1])\big)=\Graph_j \cap \big(M\setminus  f(\R^{n-1}\times [0,1])\big)$$ for all  ${i},j={1},\dots,{m}$,
\item $\Graph_i \cap \big(M_0\setminus f(\R^{n-1}\times [0,1])\big)$
  has an $\ideal$--colored edge where $M_0$ is the connected component
  of $M$ which contains the box, and
  % $\Graph_i$ has an $\ideal$--colored edge not completely
  % contained in the box for all $\for{i}{1}{m}$,
  \item $\sum_{i=1}^m a_i\F(f^{-1}(\Graph_i)) =0$ in $\cat$.
\end{enumerate}
We say two formal linear combinations are \emph{$\ideal$-skein
  equivalent} if their difference is a sum of $\ideal$-skein relation.
When $\ideal=\cat$ we say \emph{skein relation} for a $\cat$-skein
relation and skein equivalent for
$\cat$--skein equivalent.
\end{definition}

Now we give the main definition of this subsection. 
\begin{definition}\label{D:AdmisSkeinMod}
  Define the \emph{admissible skein module} $\Skein_{\ideal}(M)$ as
  the $\kk$--span of closed %\BPm{added closed} 
  $\ideal$--admissible ribbon graphs in $M$ modulo
  the span of $\ideal$--skein relations.   Define a \emph{skein} as a
  ribbon $\cat$-graph representing an element of the admissible skein
  module.
\end{definition}
 Since  $\cat$ is essentially small then $\Skein_{\ideal}(M)$ is always a set; indeed up to skein equivalence we can replace any color with a fixed representative of its  isomorphism class in $\cat$.  
Remark that since the skeins have no boundary vertices they are in fact embedded in the interior of $M$. 
Note that isotopic skein are $\ideal$--skein equivalent.  
\subsection{Some properties of admissible skein modules} Some of the properties discussed in this subsection are not used elsewhere in this paper, but are of independent interest and are used in related works. Several of these results first appeared in the unpublished manuscript \cite{CGP23}, and we record them here for completeness and future reference.

Let $\Emb^{surj}_n$ be the category whose objects are oriented
$n$-dimensional manifolds and morphisms are isotopy classes of
orientation preserving 
embeddings $M\hookrightarrow N$ such that the inducted map $\pi_0(M) \to \pi_0(N)$ is surjective.

\begin{prop}\label{P:MappingClassActs}
Let $n\in \{2,3\}$, assuming that $\cat$ is ribbon if $n=3$. Then  $$\Skein_{\ideal}:(\Emb^{surj}_n, \sqcup) \to (\Vect,\otimes)$$ is a symmetric monoidal functor.
%, where $\Emb^{surj}_n$ is equipped with disjoint union and $\Vect$ with tensor product symmetric monoidal structure.
 In particular, if $n=2$ they provide representations of the mapping class group of surfaces. 
\end{prop}
\begin{proof}
If $f:M\to N$ is an embedding which is surjective on connected components  and $\Graph\in \Skein_{\ideal}(M)$ 
%N: is admissible then $f(\Graph)\in \Skein_{\ideal}(N)$ is also admissible. 
then $f(\Graph)\in \Skein_{\ideal}(N)$ is $\ideal$-admissible since $\Graph$ is  $\ideal$-admissible.  
%\NGm{Added $\ideal$-admissible here. } 
%
Furthermore, the image of a skein relation under $f$ is a skein relation.  
\end{proof} 

\begin{prop}\label{P:FunctorIndSkein}
Let $\cat$ and $\catd$ be % essentially small 
pivotal $\kk$--categories.   Let $G: \cat\to\catd$ be a strict pivotal functor (see \cite[Section 1.7.5]{TV17}) and let $\ideal$ be an ideal of $\cat$.  Let $\Jideal$ be the ideal of $\catd$ generated by the objects of $G(\ideal)$.  The functor $G$ induces a natural transformation $G_*:\Skein_\ideal\to \Skein_\Jideal$.
\end{prop}
\begin{proof}
It is clear that the image of an $\ideal$--admissible graph is a $\Jideal$-admissible graph.
Furthermore, a $\ideal$-skein relation is transformed via the application of $G$ into a $\Jideal$-skein relation; indeed $\F_\catd\circ G_*=G\circ \F_\cat:\Rib^n_\cat\to \catd$ and if the complement of a box contains a $\ideal$-colored edge then after application of $G_*$ the complement of the box contains a $\Jideal$-colored edge.
\end{proof}

\begin{prop}\label{prop:ddd}
  $\Skein_\ideal(M)$ is generated by ribbon graphs where each edge is
  colored by an object of $\ideal$.
\end{prop}
\begin{proof}
  By adding identity coupons it is easy to see $\Skein_\ideal(M)$ is
  generated by ribbon graphs where each edge is joining two different
  coupons.  Then we can induct on the number of edges whose color
  does not belongs to $\ideal$.
  
      Using a skein relation, we can
  replace such an edge colored by $V$ which is close to an
   edge colored by $U\in \ideal$ with an edge colored by $V\otimes U\in\ideal$
   changing the coupons as in the following figure:
    $$\epsh{fig4a}{16ex}\put(1,2){\ms{U}}\put(-25,2){\ms{V}}
   \put(-19,22){\ms{f}}\put(-19,-19){\ms{g}}
   \qquad\longrightarrow\quad
   \epsh{fig4b}{16ex}\put(-11,2){\ms{V\!\!\otimes\! U}}
   \put(-25,22){\ms{f\!\otimes\!\Id_U}}\put(-25,-20){\ms{g\!\otimes\!\Id_U}} \hspace{15pt}. $$
      Remark that up to skein relations in the neighborhood of the
   coupons adjacent to the $V$-colored edge, we can always assume that
   its position relatively to the coupons is as in the above figure.
 \end{proof}
%\subsection{Algebra and module structures of $\ideal$-skein modules}

\begin{prop}\label{prop:boundaryaction}
  Let $n=2$ or 3 and assume that $\cat$ is ribbon if $n=3$.  Let
  $\ideal$ and $\Jideal$ be ideals of $\cat$.  If $\Sigma$ is an
  $(n-1)$-manifold then $\Skein_{\ideal}(\Sigma\times I)$ is an
  associative algebra, which is unital if $\ideal= \cat$.
  Furthermore, if $M$ is a $n$-manifold and
  $i:\Sigma\hookrightarrow \partial M$ is an orientation reversing
  (resp.\ preserving) %\BPm{exchanged reversing/preserving}\NGm{OK} 
  embedding, then $\Skein_{\ideal}(M)$ is a left
  (resp.\ right) module over $\Skein_{\Jideal}(\Sigma \times
  I)$. Finally, when $n=3$ we have that $\Skein_{\ideal}(\Sigma)$ is
  isomorphic to $\Skein_{\ideal}(\Sigma\times I)$.
\end{prop}

\begin{proof}
We give the algebra operation of $\Skein_{\ideal}(\Sigma\times I)$.  Let  $\Graph,\Graph'\subset \Sigma\times [0,1]$ are ribbon graphs.  We can identify $\Graph$ and  $\Graph'$ as graphs in  $\Sigma\times [0,1]$  and  $\Sigma\times [1,2]$, respectively.  Then   $\Graph\cdot \Graph'$ is the ribbon graph $\Graph\cup \Graph'$ in $\Sigma\times [0,2]\simeq\Sigma\times [0,1]$ obtained by glueing $\Sigma\times [0,1]$  and  $\Sigma\times [1,2]$. 

We claim that this operation gives a well defined associative product.
Indeed, if $\Graph$ and $\Graph'$ are $\ideal$-admissible graphs in
$\Sigma\times I$ then so is $\Graph\cdot \Graph'$.
Furthermore,  
%\BPm{modified}\NGm{OK}
 if $\Graph'$ and $\Graph''$ are related by a
$\ideal$-skein relation in a box $B$ then 
$\Graph\cdot \Graph'$ and $\Graph\cdot \Graph''$ are $\ideal$-skein
equivalent with the same box $B$ which does not intersect $\Graph$.
% Furthermore if $\Graph'$ and $\Graph''$ are related by a
% $\ideal$-skein relation on a box $B$ then we will show that
% $\Graph\cdot \Graph'$ and $\Graph\cdot \Graph''$ are $\ideal$-skein
% equivalent: up to isotopy of $\Graph$ we can suppose that both
% $\Graph\cdot \Graph'$ and $\Graph\cdot \Graph''$ intersect $B$
% transversally and that $\Graph\cap B$ lies entirely at the left of
% both $\Graph'\cap B$ and $\Graph''\cap B$.  Then
% $\F(\Graph\cdot \Graph'\cap B)=\F(\Graph\cap B)\otimes \F(\Graph'\cap
% B)$ and
% $\F(\Graph\cdot \Graph''\cap B)=\F(\Graph\cap B)\otimes
% \F(\Graph''\cap B)$.  But since $\F(\Graph'\cap B)=\F(\Graph''\cap B)$
% we conclude that
% $\F(\Graph\cdot \Graph'\cap B)=\F(\Graph\cdot \Graph''\cap B)$.
%
The operation is clearly associative.
% If $\unit\notin \ideal$ then the  empty graph is not a $\ideal$--admissible graph so that $\Skein_{\ideal}(\Sigma)$ is not unital.
If $\ideal=\cat$ then $\Skein_{\ideal}(\Sigma\times I)$ has a unit given by the graph with one edge colored by $\unit$ and two coupons colored by $\Id_\unit$.

Suppose $i:\Sigma\hookrightarrow \partial M$ is an orientation
reversing 
%\BPm{changed}\NGm{OK}
 embedding.  % Let $N\simeq \partial M\times [0,1]$ be a
% closed tubular neighborhood of $\partial
% M$.
Then $i$ induces a map $$j:\Sigma\times[0,1]\sqcup M\to\Sigma\times[0,1]\cup_{i\times\bs1} M\simeq M$$ which at the level of skein module induces
$$\Skein_{\Jideal}(\Sigma\times[0,1])\otimes\Skein_{\ideal}(M)\to\Skein_{\ideal}(M),$$
% Then $i$ induces the map
% $$j:\Sigma \times [1/3,2/3] \hookrightarrow \partial M\times [1/3,2/3] \subset N\subset M$$
% which gives an embedding of $\Sigma \times [1/3,2/3]\simeq \Sigma \times [0,1]$ into the
% interior of $M$.  \NGm{I changed here as well.  If ok please remove.}
given by % We will show the assignment
% $\cdot : \Skein_{\Jideal}(\Sigma\times [0,1])\times \Skein_{\ideal}(M)\to
% \Skein_{\ideal}(M)$ given by
$[\alpha]\cdot [\beta]=[j(\alpha)\sqcup j(\beta)]$ % is well defined 
where
$\alpha$ and $\beta$ are representatives of
$[\alpha]\in \Skein_{\Jideal}(\Sigma\times [0,1])$ and
$[\beta]\in \Skein_{\ideal}(M)$, respectively.%  (here we assume
% $\beta\subset M\setminus N$ % $\beta\subset int(M)$
% which can be done by pushing it inside $M$ with an isotopy).
Notice $j(\alpha)\sqcup j(\beta)$ is an $\ideal$-admissible graph in
$M$.  If $\beta$ and $\beta'$ are related by a $\ideal$-skein relation
then $j(\alpha)\sqcup j(\beta)$ and $j(\alpha)\sqcup j(\beta')$ are
$\ideal$-skein equivalent.  Also, if $\alpha$ and $\alpha'$ are
related by a $\Jideal$-skein relation given by a box
$B\subset \Sigma\times [0,1]$ then $j(\alpha)\sqcup j(\beta)$ and
$j(\alpha')\sqcup j(\beta)$ are $\ideal$-related because $j(\beta)$
has an $\ideal$-colored edge outside $j(B)$.  Therefore we have a well
defined action as claimed.  If the orientation induced by $M$ on
$i(\Sigma)$ is positive %\BPm{changed}\NGm{Ok, thanks!} %negative
then the right action is proven similarly with
$j:\Sigma\times[0,1]\sqcup M\to\Sigma\times[0,1]\cup_{i\times\bs0}
M\simeq M$. % and $i(\alpha)$ lies below $\beta$.
% $j(\alpha)\sqcup j(\beta)$ is an $\ideal$-admissible graph in $M$.  If
% $\beta$ and $\beta'$ are related by a $\ideal$-skein relation given in
% a box $B$ then up to isotopy we can assume $B\subset int(M)$ implying
% that $j(\alpha)\sqcup \beta$ and $j(\alpha)\sqcup \beta'$ are
% $\ideal$--skein related in $B$.  Also, if $\alpha$ and $\alpha'$ are
% related by a $\Jideal$-skein relation given by a box $B\subset \Sigma\times [0,1]$
% then $j(\alpha)\sqcup \beta$ and $j(\alpha')\sqcup \beta$ are
% $\Jideal$-related in a box $B'$ obtained by thickening slightly $j(B)$
% inside $M$.  Since $[\beta]\in \Skein_{\ideal}(M)$ then up to isotopy
% this relation is actually a $\ideal$-skein relation.  Therefore we
% have a well defined action as claimed.  If the orientation induced by
% $M$ on $i(\Sigma)$ is negative then the right action is proven
% similarly where $[\beta\cdot \alpha]=[\beta\sqcup
% j(\alpha)]$. % and $i(\alpha)$ lies below $\beta$.

For the final statement, let $n=3$ then $\Sigma$ is a surface which is
naturally embedded in the 3-manifold $\Sigma\times [0,1]$ by the map
$p: \Sigma \hookrightarrow \Sigma\times \{1/2\} \subset \Sigma\times
[0,1]$.  Let $\Graph$ be a skein in $\Skein_{\ideal}(\Sigma)$.  The
assignment $\Graph\mapsto p(\Graph)$ gives a skein in
$\Skein_{\ideal}(\Sigma\times [0,1])$ and a map from
$\Skein_{\ideal}(\Sigma)$ to $\Skein_{\ideal}(\Sigma\times [0,1])$.
The inverse of this map is given as follows.  Let $\Graph'$ be a skein
in $\Skein_{\ideal}(\Sigma\times I)$.  There exist a regular
projection of the ribbon graph $ \Graph'$ onto $\Sigma \times \{1/2\}$
where all the intersections are transverse points.  Each of these
crossing points can be replaced by a coupon decorated by the braiding
of $\cat$ to obtain a skein in $\Skein_{\ideal}(\Sigma)$.
\end{proof}

%\NG{The above proposition maps a surface to an algebra.  Is this interesting?  }

Next we recall an  interpretation of the skeins of the disk and the sphere in terms of m-traces, which is given in \cite[Theorem 2.4.]{CGPV23}. 
\begin{theorem}\label{T:DiskRmt}Let $\cat$ be a % n essentially small
  pivotal $\kk$--category and let $\ideal$ be an  ideal of $\cat$.  Then
$$\Skein_\ideal(D^2)^*\cong \{\text{right m-traces on } \ideal\}\cong \{\text{left m-traces on } \ideal\} \;\; $$ $$\text{ and } \;\;
\Skein_\ideal(S^2)^*\cong \{\text{m-traces on } \ideal\}$$
where $D^2$ and $S^2$ are the 2-disk and 2-sphere, respectively. 
\end{theorem}

The following result is a consequence of the last theorem; it first appeared in the unpublished manuscript \cite{CGP23}, we prove it here so it can appear in print.  %N \BH{Please add a sentence on where this already appears in the literature.}

\begin{corollary}\label{T:dim3SphereSkein}
  Let $\cat$ be a % n essentially small
  ribbon $\kk$--category and let
  $\ideal$ be an ideal of $\cat$.  Then
  $$\Skein_\ideal(B^3)^*\cong \Skein_\ideal(S^3)^*
  \cong \{\text{m-traces on } \ideal\}$$
  where $B^3$ and $S^3$ are the 3-ball and 3-sphere, respectively.
\end{corollary}
\begin{proof}
  We will show that $\Skein_\ideal(D^2)$ is isomorphic to
  $\Skein_\ideal(B^3)$, then by \cite[Remark 9]{GPV13} since any right
  m-trace is a m-trace in a ribbon category then the corollary will
  follow from Theorem \ref{T:DiskRmt}.

  The inclusion of $D^2$ into $B^3$ induces a morphism
  $\Skein_\ideal(D^2)\to\Skein_\ideal(B^3)$.  It has an inverse map
  given by sending any ribbon graph in $B^3$ to a 
%N:   standard projection  in $D^2$ at the price of   replacing  crossings with coupons
  projection in general position  in $D^2$ where we have replaced each crossing with a coupon
%  \NGm{changed "standard projection" to "projection in general position"  This is the term I found on the internet.  If ok please remove this comment. }
  %
colored by
  a braiding.  Any two such projections are related by %N standard
  Reidemeister type moves which are skein equivalences in $D^2$.
  Finally, a skein relation in $B^3$ can be isotoped to project onto  a
  skein relation in $D^2$.
\end{proof}

\subsection{Definition of graded admissible skein modules}
As above, let $n=2$ or $3$, $\Gr$ be
a group and $\cat$ be a sharp
$\Gr$-graded pivotal $\kk$--category. If $n=3$, we assume that $\cat$ is ribbon and $\Gr$ is commutative.

Define the discrete category $\catd_\Gr$ of $\Gr$ to be the category
whose objects are the elements of $\Gr$, % \ie $\{g, g\in \Gr\}$,
and the only morphisms are the
identities.  A strict monoidal structure on $\catd_\Gr$ is given by
$g\otimes h=gh$ and the pivotal structure by $g^*=g^{-1}$ and the fact
that all (co)-evaluation morphisms are the identity of $\unit$.  If
$\Gr$ is abelian we endow $\catd_\Gr$ with the symmetric braiding
$g\otimes h\to h\otimes g$ given by the identity map of $gh=hg$
turning $\catd_\Gr$ into a ribbon category.

%N: This category is not linear, but we can nevertheless talk about its skein module $\Skein_{\catd_\Gr}(M)$ over an $n$-manifold $M$. It is now a only a set, of $\catd_\Gr$-colored graphs modulo isotopy and the skein relations where two graphs are skein equivalent if they agree outside a ball and evaluate to the same morphism inside the ball. Since there is at most one morphism between any tensor product of objects of $\catd_\Gr$, two $\catd_\Gr$-colored graphs are skein equivalent as soon as they agree outside a ball. 
This category is not $\kk$--additive, but we can nevertheless talk about its skein module associated to  an $n$-manifold $M$ as follows: let $\Skein_{\catd_\Gr}(M)$ be the set of $\catd_\Gr$-colored graphs modulo  skein relations.  Here two graphs are skein equivalent if they agree outside a ball and evaluate to the same morphism inside the ball.  Note calling $\Skein_{\catd_\Gr}(M)$ a module is an abuse of terminology since it is only a set.  
%\NGm{Changed this paragraph, left the original commented.}
%

\begin{lemma}\label{L:H1ab}
  If $\Gr$ is abelian then sending a $\catd_\Gr$-colored
  skein to the homology class of the associated $1$-cycle with
  coefficients in $\Gr$ induces a bijection $$\Skein_{\catd_\Gr}(M)\simeq H_1(M;\Gr)\ .$$
\end{lemma}
When $\Gr$ is not abelian, the above lemma motivates us to call the skein module $\Skein_{\catd_\Gr}(M)$ the ``non-commutative $H_1$'' of
$M$ with coefficients in $\Gr$.  Remark that in
contrast with the commutative case this $H_1$ with non-commutative
coefficients is not a group.
\begin{definition}
 We denote $H_1(M;\Gr):=\Skein_{\catd_\Gr}(M)$.
\end{definition}

% A \textit{$\Gr$-grading} on $\cat$ is an equivalence of linear categories
% $\cat \cong \bigoplus_{g \in \Gr} \cat_g$ where 
% $\{ \cat_g \mid g \in \Gr \}$ is a  family of full subcategories of $\cat$
% satisfying the following conditions: 1) $\unit \in \cat_1$, 2)  if $V\in\cat_g$,  then  $V^{*}\in\cat_{g^{-1}}$, 3) if $V\in\cat_g$, $V'\in\cat_{g'}$ then $V\otimes
%     V'\in\cat_{gg'}$ and 4)  if $V\in\cat_g$, $V'\in\cat_{g'}$ and $\Hom_\cat(V,V')\neq \{0\}$, then
%     $g=g'$.

% Let $\cat$ be a sharp $\Gr$-graded pivotal $\kk$-category. An object
% $V\in\cat_g$ is called homogeneous with degree $g$.  The degree of a
% non-zero homogeneous object is unique.
\begin{definition}\label{def:homogeneous}
  % A $\cat$-colored graph $\Graph$ is \emph{homogeneous} if each edge $e$ of $\Graph$ is colored by an homogeneous object $c(e)\in \cat_{g_e}$ for some $g_e\in \Gr$. In this case there is an associated $\catd_\Gr$-colored graph $h_\Graph \in H_1(M;\Gr)$ having the same underlying graph, the color of each edge $e$ is $g_e$ and morphisms decorating the coupong of $\Graph$ are replaced by $1$ on each coupon. 
%  \FCm{Is this ok for you ? I have a doubt because what if one morphism was zero in the coupons ? It seems to me that this is no contradiction as 0 has every degree but just want to check with you}\BPm{I put back the previous definition, because yours does not guaranty the cycle condition at vertices which is necessary for the existence of $h_T$}
   A $\cat$-colored graph $\Graph$ is \emph{homogeneous} if there exists a  $\catd_\Gr$-colored graph $h_\Graph \in H_1(M;\Gr)$ with the same underlying ribbon graph such that each edge $e$ of $\Graph$ is colored with a homogeneous %\FCm{I don't thiknk we can talk about homogeneous here: I would drop this here, but I did not to check with you guys}\BPm{why ? it is really an homogeneous object of $\cat$}%\FCm{Because the colors of $h\Graph$ must be in $\catd_\Gr$ not $\cat$.}\BPm{I think it is ok: $\catd_\Gr\simeq\Gr$.  We should discuss how to clarify...} 
   object in $\cat$ whose degree $g\in\Gr$ %\BPm{added $g\in\Gr$} 
   is given by the color of $e$ in $h_\Graph$.  The graph $h_\Graph$ is called the \emph{degree} of $\Graph$.
\end{definition}
% \begin{definition}\label{def:homogeneous}
%   A $\cat$-colored graph $\Graph$ is \emph{homogeneous} if each edge
%   $e$ of $\Graph$ is colored with an non-zero homogeneous
%   object % of$_{g(e)}$ for some $g(e)\in \Gr$
%   and every coupon $c$ of $\Graph$ is filled with a morphism of
%   $\cat_{g(c)}$ for some $g(c)\in \Gr$.
% \end{definition}

\begin{proposition}\label{P:graduation}
Let $M$ be an $n$-manifold.  We have the following decomposition of $\kk$-vectors spaces:
 $$\Skein_{\ideal}(M)=  \bigoplus_{h\in H_1(M;\Gr)} \Skein^h_{\ideal}(M)$$
where $\Skein^h_{\ideal}(M)$ is the $\kk$-span of homogeneous $\ideal$--admissible graphs $M$ of degree $h$.
%, induced by the inclusion of each $\Skein^h_{\ideal}(M)$ into $\Skein_{\ideal}(M)$.
\end{proposition}
\begin{proof}
Let us first show surjectivity of the map $\omega:\bigoplus_{h\in H_1(M;\Gr)} \Skein^h_{\ideal}(M)\to
  \Skein_{\ideal}(M)$, \ie that each $\ideal$--admissible graph is $\ideal$-skein equivalent to a linear combination of homogeneous $\ideal$--admissible graphs.
 If $V\in \cat$  then the definition of a $\Gr$-grading 
  implies there exist finitely many homogeneous objects $V_i$
  such that $V=\oplus V_i$. Furthermore, if $V\in \ideal$ then each
  $V_i$ is in $\ideal$ because $V_i$ is a retract of $V$.  Let
  $p_i:V\to V_i$ and $j_i:V_i\to V$ be projections and inclusions such
  that $\Id_V=\sum_{i} j_i\circ p_i$.
 
  %N: If $\Graph$ contains an edge colored by $V$ 
   If a skein contains an edge colored by $V\in \ideal$ 
  then up to applying an
  $\ideal$-skein relation we can replace a small sub-arc of the edge
  by the formal sum of arcs each containing two coupons decorated
  respectively with $j_i$ and $p_i$ because
  $\Id_V=\sum_{i} j_i\circ p_i.$ Then isotoping each of the coupons
  decorated by $j_i$ and $p_i$ we can ``absorb'' them in the coupons
  at the endpoint of the edge via a $\ideal$-skein relation.
  Repeating this argument for all the non-homogeneous edges we
  conclude that any skein is $\ideal$-skein equivalent to a linear
  combination of graphs whose edges are colored by homogeneous objects.
  For such a  graph $\Graph$,
  %N: Added, to make it clear it is not the original skein:
  colored with homogeneous objects,  
   if a coupon is colored by a non-zero
  morphism, this means that its source and target objects of $\cat$
  are homogeneous objects of the same degree.  Hence, if $\Graph$ is
  not $0$ in the admissible skein module, changing each object with
  its degree in $\Gr$ and each morphism with the identity gives a
  $\catd_\Gr$-coloring which means that $\Graph$ is homogeneous.

Let us turn to injectivity of
  $\omega$, %N: :\bigoplus_{h\in H_1(M;\Gr)} \Skein^h_{\ideal}(M)\to  \Skein_{\ideal}(M)$, 
  %N i.e. 
  \ie we will show 
  that any skein relation in
  $\Skein_{\ideal}(M)$ is a sum of skein relations where each skein
  relation involves graphs of the same homogeneous degree.  Let
  $\sum_{i=1}^k c_i \Graph_i=0\in \Skein_{\ideal}(M)$ be a skein
  relation associated to a box $B$. As above we can split the $\Graph_i$ as linear
  combination of graphs which are homogeneous outside $B$.  By this
  process, the identity $\sum_ic_i\F(\Graph_i\cap B)=0$ which
  holds in some hom-space $\Hom_\cat(V_-,V_+)$ splits into many
  identities corresponding to summands $\Hom_\cat(V_-,V_+)$.  Hence
  the skein relation $\sum_{i=1}^k c_i \Graph_i=0$ is a sum of skein
  relations where all involved graphs are homogeneous outside
  $B$. 
  %\BPm{added}
   This reduces to the case of a skein relation where
  all $\Graph_i$ are homogeneous outside $B$.  Now if $\Graph$ and
  $\Graph'$ are any two homogeneous graph which coincide outside the
  box $B$, we have $h_\Graph=h_{\Graph'}$ because they are both
  equivalent to the $\catd_G$-graph where the box $B$ is replaced with
  a coupon colored with the identity morphism.  Hence applying the surjectivity argument above inside the box
  gives a linear combination of homogeneous graphs which have all the
  same degree.
\end{proof}

\subsection{$\Gr$-representations as a skein module}
\label{SS:HomoGSkeinMod}
We now want to compare $\Gr$-re\-pre\-sen\-ta\-tions with $\catd_\Gr$-colored
graphs used to define homogeneous graphs above.
We will see that admissible skein modules of surfaces are actually graded by
% a very familiar object, the set of representations of the fundamental group into $\Gr$.
the set of representations of the fundamental group into $\Gr$. 

\begin{definition}
Let $\Sigma$ be an oriented connected surface and $Y\subset \Sigma$ a finite set of basepoints, let $\gamma$ be a path immersed in $\Sigma$ with endpoints in $Y$ and let $\Graph$ be a $\catd_\Gr$-colored graph in $\Sigma$ which does not intersect $Y$. Assume that $\gamma$ is transverse to $\Graph$ in the sense that $\gamma$ only intersects $\Graph$ along edges transversally. 
The \emph{intersection of $\gamma$ and $\Graph$} is the ordered product 
$$[\gamma\cap \Graph]_\Gr:=\prod_{i=1}^n[\gamma\cap e_i]_\Gr \in \Gr $$
where $\{x_i\}_{i=1,\dots,n}$ is the ordered set of intersection
points of $\gamma$ and $\Graph$ (ordered by the path $\gamma$, see
Figure \ref{fig:intersection}), $e_i$ is a portion of the edge of
$\Graph$ in a neighborhood of $x_i$ 
%\BPm{I changed again because it was looking strange to me to have $\gamma\cap$ an intersection point}\NGm{Ok} 
such that at each intersection
$x_i$, we set $[\gamma\cap e_i]_\Gr=g^{\ve}$ where $g\in\Gr$ is the
color of the edge of $\Graph$ at the intersection point, and $\ve$ is
the sign of the intersection.
\end{definition}
This intersection only depends on the homotopy class of $\gamma$ in $\pi_1(\Sigma,Y)$ because the only coupons in $\catd_\Gr$ are identities. It only depends on the class of $\Graph\in H_1(\Sigma\smallsetminus Y;\Gr)$ because skein relations happen in a ball, which can be avoided by % \BPm{changed from by an isotopy of $\gamma$.}
moving $\gamma$ using an isotopy.

%\BH{There was no proof of this proposition in the old version, only that the map is well defined. This statement is FALSE without the assumption that the surface is closed. Consider a disk with a single basepoint. Then $H_1(D^2\smallsetminus \ast;\Gr) = G$ generated by circles around the points, whereas $\Rep_\Gr(D^2,\ast) = Hom(\pi_1(D^2,\ast),\Gr) = \ast$ is trivial. In the proof below we rely on a good Poincaré duality.} \NG{I think this comment is now addressed.  If so Ben please remove this comment and yours.  Thanks!!}
\begin{proposition} \label{P:H1GisRepG}
Suppose $\Sigma$ is closed and $Y\subset \Sigma$ a finite set of basepoints with at least one basepoint per connected component. The intersection pairing induces a bijection %an isomorphism \NGm{Should we say bijection here, or say what kind of isomorphism it is?}
$$
\begin{array}{rcl}
[-\cap -]_\Gr : H_1(\Sigma\smallsetminus Y;\Gr) &\overset\sim\longrightarrow& \Rep_\Gr(\Sigma,Y) \\ \Graph\quad&\mapsto&[\any\cap \Graph]_\Gr
\end{array}\ . 
$$
% \NGm{I changed $[T\cap\any]_\Gr$ to $[\any\cap T]_\Gr$ in the definition of the isomorphism, so it is consistent with above and below.  IF this is ok please remove this comment.}
\end{proposition}
\begin{proof}
  One can construct its inverse as follows. Let
  $\rho\in \pi_1(\Sigma,Y)$ be a $\Gr$-represen\-ta\-tion. Choose an
  ideal triangulation of $\Sigma$ with 0-cells $Y$, orient its edges
  arbitrarily and take $\Graph$ to be the 1-skeleton dual to the
  triangulation; orient the edges of $\Graph$ so that  the
  edges of the triangulation intersect $\Graph$ positively with respect to the
  orientation of $\Sigma$. Ideal triangulations exist except for
  $\Sigma = S^2$ with at most two basepoints, which can be treated
  separately. An edge of $\Graph$ intersects % \BPm{changed from corresponds
    % to}
  an edge $\gamma$ of the ideal triangulation which we view as
  an element of $\pi_1(\Sigma,Y)$. We color this edge by
  $\rho(\gamma)\in\Gr$.  We thicken the vertices of the graph $\Graph$ into
  coupons colored by an identity morphism (all possible ways of doing
  this are skein-equivalent). % \BPm{added}
  Let us denote the skein
  obtained by $\Graph_\rho \in H_1(\Sigma\smallsetminus Y;\Gr)$.

Clearly, for every edge $\gamma$ of the ideal triangulation, we have $[\gamma\cap \Graph_\rho]_\Gr = \rho(\gamma)$ as there is only one intersection point by construction. Such edges generate $\pi_1(\Sigma, Y)$, hence $[-\cap \Graph_\rho]_\Gr = \rho$ as elements of $\Rep_\Gr(\Sigma,Y)$. Reciprocally, $\Sigma\smallsetminus Y$ retracts on $\Graph$ hence every element $\Graph'$ of $H_1(\Sigma\smallsetminus Y;\Gr)$ is skein equivalent to one with underlying graph $\Graph$. For every edge of $\Graph$, there is an edge $\gamma$ of the ideal triangulation that intersects it (and only it) once. Hence $\Graph_{[-\cap \Graph']_\Gr}$ has the same colors as $\Graph'$.
\end{proof}
\begin{remark}
%N It is also true that $H_1(\Sigma;\Gr) \simeq \Rep_\Gr(\Sigma,\emptyset)$, though we will not need it.
One can also show that there is a bijection  $H_1(\Sigma;\Gr) \simeq \Rep_\Gr(\Sigma,\emptyset)$, however this fact will not be used in this paper.  
% \NGm{I edited this remark. Is this just a bijection?  Left old, if ok please remove this comment.  }
%
\end{remark}
It will be useful to generalize the intersection pairing defined above to $\cat$-graphs.

\begin{definition}

Let $\gamma$ be an immersed path with endpoints in $Y$ transverse to a  homogeneous $\cat$-graph $\Graph$ which does not intersect  $Y$ in an oriented surface $\Sigma$. We define the $\cat$-intersection % and $\Gr$-intersection 
of $\gamma$ and $\Graph$, % respectively,
as the ordered (tensor) products
\begin{equation}
  \label{eq:inter}
  [\gamma\cap \Graph]_\cat=\bigotimes_{i=1}^n[\gamma\cap e_i]_\cat
  % \quad \text{  and  }\quad
  % [\gamma\cap \Graph]_\Gr=\prod_{i=1}^n[\gamma\cap e_i]_\Gr
\end{equation} where $\gamma$ intersects $\Graph$ only at the edges
$e_1,\ldots e_n$ in this order and when $\gamma$ intersects a piece of % \NGm{Why is the word piece used here?  Can we remove this or make it more clear what is happening?} 
an oriented edge $e$ of $\Graph$ which is colored by an object $V$ of degree $g$ of
$\cat$, we define $[\gamma\cap e]_\cat=V^{\ve}$ % and $[\gamma\cap e]_\Gr=g^{\ve}$ 
where $\ve$ is the sign of the intersection $\gamma\cap e$ and by convention, $V^{+1}=V$ and
$V^{-1}=V^*$. See Figure \ref{fig:intersection}.
\end{definition}
Clearly, $[\gamma\cap \Graph]_\cat$ is an object of $\cat_{[\gamma\cap \Graph]_\Gr}$ and $[\gamma\cap \Graph]_\Gr = [\gamma\cap h_\Graph]_\Gr$ where $h_\Graph$ is the degree of $\Graph$.

% $%{\begin{array}{l}
%       V_1\in\cat_{g_1}, V_2\in\cat_{g_2}, V_3\in\cat_{g_3},$\\%
%       $[ \gamma\cap \Graph ]_{\cat}=V_1\otimes V_2^*\otimes V_3,$\\%
%       $[ \gamma\cap \Graph ]_\Gr=g_1g_2^{-1}g_3$
%       % \end{array}

\begin{figure}
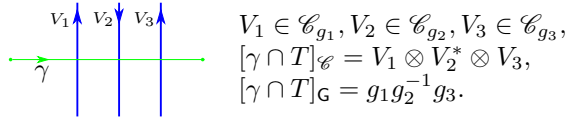

  \begin{minipage}{0.4\linewidth}
    \hfill  $
    \epsh{fig33f.pdf}{10ex}
    \putw{32}{86}{$\ms{V_1}$}
    \putw{54}{87}{$\ms{V_2}$}
    \putw{75}{87}{$\ms{V_3}$} \puts{17}{45}{${\gamma}$}
  $
\end{minipage}
\quad
  \begin{minipage}{0.4\linewidth}
  $V_1\in\cat_{g_1}, V_2\in\cat_{g_2}, V_3\in\cat_{g_3},$\\
      $[\gamma\cap \Graph]_\cat=V_1\otimes V_2^*\otimes V_3,$\\
      $[\gamma\cap \Graph]_\Gr=g_1g_2^{-1}g_3$.
\end{minipage}
\caption{Exemple of intersection of a path $\gamma$ with a skein $\Graph$.}
  \label{fig:intersection}
\end{figure}

\subsection{Finite dimensionality}
In this subsection $\cat$ is a % n additive
             sharp
pivotal $\Gr$-graded $\kk$--category (the case of ungraded categories
is included by setting $\Gr$ to be the trivial group).  Let $\Sigma$ be a
surface.

\begin{definition}
  We say that an ideal $\ideal$ of $\cat$ is \emph{graded finite} if for each $g\in \Gr$
  the subcategory $\ideal\cap \cat_g$ admits a generator $V_g$. % , that is any object of $\ideal\cap \cat_g$ is retract of $V_g^{\oplus n}$ for some $n$.
\end{definition}
Recall that by Proposition \ref{P:graduation} we have
$\Skein_\ideal(\Sigma)=\bigoplus_{h\in H_1(\Sigma;\Gr)}
\Skein_\ideal^h(\Sigma)$.
\begin{proposition}\label{prop:finitewithboundary}
Let $\Sigma$ be a compact connected 
  surface with non-empty boundary then for every graded finite ideal
  $\ideal\subset \cat$ and every $h\in H_1(\Sigma;\Gr)$ it holds
  $\dim_{\kk}\Skein^h_\ideal(\Sigma)<\infty$. 
\end{proposition}
\begin{proof}
  Let $Y$ be a set of base points lying in the boundary of $\Sigma$.
  Remark that $\Skein_\ideal(\Sigma\setminus Y)=\Skein_\ideal(\Sigma)$ and
  $H_1(\Sigma\setminus Y;\Gr)=H_1(\Sigma;\Gr)$.  There exists a set of
  disjoint simple arcs $c_1,\ldots c_k$ properly embedded in
  $\Sigma$, with end points in $Y$, and such that
  $\Sigma \setminus (c_1\sqcup \ldots \sqcup c_k)$ is a disk.  If
  $\Graph\in \Skein_\ideal(\Sigma)$ is homogeneous of degree
  $h\in H_1(\Sigma, \Gr)$, then we can assume up to isotopy that it
  intersects each $c_i$ transversally and by Proposition
  \ref{prop:ddd} that all its edges are $\ideal$-colored. Then
  $\Graph$ is skein equivalent (by fusing all the edges intersecting
  each $c_i$) to a graph intersecting each $c_i$ at most once; let
  $h_i=[c_i\cap\Graph]_\Gr = [c_i\cap h]_\Gr$ be the degree of such an intersection and $V_{h_i}$ a generator of $\ideal\cap \cat_{h_i}$. Then up
  to applying one skein relation for each $c_i$ we can replace
  $\Graph$ with a linear combination of graphs intersecting $c_i$ via
  a single edge colored by $V_{h_i}$.  We can express each of these
  graphs via linear combination of graphs each containing a single
  coupon decorated by a morphism in
  $\Hom(\unit,V_{h_1}\otimes \cdots \otimes V_{h_k})$ in the disk
  $\Sigma\setminus (c_1\sqcup\cdots \sqcup c_k)$.  Therefore % if
  % $h\in H_1(\Sigma, \Gr)$ is the homology class of $\Graph$ 
  we just
  showed that
  $\dim_\kk \Skein_\ideal^h(\Sigma)\leq
  \dim_{\kk}\Hom(\unit,V_{h_1}\otimes \cdots \otimes
  V_{h_k})<\infty$.% \BP{we did not prove that $h_i$ is independant of
    % $T$ (depends only of $h$).}
\end{proof}

\begin{proposition}\label{P:G_I}
  Let $\ideal$ be an ideal in $\cat$, then the set
  $\Gr_\ideal=\{g\in\Gr:\ideal\subset \ideal_{\cat_g}\}$ is a subgroup
  of $\Gr$ where $\ideal_{\cat_g}$ is the ideal generated by all
  objects of $\cat_g$.
\end{proposition}
\begin{proof}
  First $1\in\Gr_\ideal$ since the ideal generated by $\unit\in\cat_1$ is $\cat$.%\BPm{added}  
  Let $g,g'\in\Gr_\ideal$ and $V\in\ideal$, %\FCm{Is it clear that $G_\ideal$ is non empty ?} 
  then
  $V\in\ideal_{\cat_{g'}}$ is a retract of $U\otimes W'$ for some
  $U\in \cat$ and $W'\in\cat_{g'}$ implying $V$ is also a retract of
  $(V\otimes V^*\otimes U)\otimes W'$ (using $\lcoev_V$ and $\lev_V$).
  So we can assume $U\in\ideal$ but then $U\in\ideal_{\cat_g}$ is retract of
  $U'\otimes W$ for some $U'\in \cat$ and $W\in\cat_g$ implying $V$ is
  a retract of $U'\otimes (W\otimes W')\in\ideal_{\cat_{gg'}}$. Thus, $gg'\in\Gr_\ideal$.  A
  similar argument for the left tensor product implies that
  $g'g\in\Gr_\ideal$.  Finally using the equivalence of category given
  by the dual, if $\ideal\subset\ideal_{\cat_g}$, then
  $\ideal=\ideal^*\subset(\ideal_{\cat_g})^*=\ideal_{\cat_{g^{-1}}}$
  thus $(\Gr_\ideal)^{-1}=\Gr_\ideal$.
\end{proof}

\begin{theorem}\label{T:finitewithoutboundary}
  Let $\ideal$ be a graded finite ideal of $\cat$.  Assume $\Gr$ is
  abelian or $\Gr=\Gr_\ideal$.  Let $\Sigma$ be a compact oriented
  surface without boundary then for every $h\in H_1(\Sigma;\Gr)$ we have $\dim_{\kk}\Skein^h_\ideal(\Sigma)<\infty$.
\end{theorem}

\begin{proof} Choose basepoints $Y\subset \Sigma$, one
%\BPm{added ``one'', else $\pi_\Gr$ is not an iso in the commutative case}\NGm{Ok} 
in every connected
  components. Any graph in $\Sigma$ can be isotoped to be disjoint
  from $Y$, hence we have surjective maps
  $${\pi_\ideal: \Skein_\ideal(\Sigma\smallsetminus Y) \to
  \Skein_\ideal(\Sigma)} \;\; \text{ and } \;\;
  \pi_\Gr: H_1(\Sigma\smallsetminus Y;\Gr) \to H_1(\Sigma;\Gr).$$ The
  map $\pi_\ideal$ sends homogeneous skeins of degree $\tilde h$ to
  homogeneous skeins of degree $\pi_\Gr(\tilde h)$. Hence we have a
  surjective map
$$\bigoplus_{\tilde h\in \pi_\Gr^{-1}(h)}\Skein_\ideal^{\tilde h}(\Sigma\smallsetminus Y) \twoheadrightarrow \Skein_\ideal^h(\Sigma). $$
By Proposition \ref{prop:finitewithboundary}, each summand 
%N: at the left 
in this sum 
is finite dimensional, but there might be infinitely many summands. If $\Gr$ is abelian, the map $\pi_\Gr$ is an isomorphism so there is only one summand; and the statement is proven. Otherwise, choose an arbitrary preferred lift $\tilde h \in \pi_\Gr^{-1}(h)$, \ie isotope the $\catd_\Gr$-graph $h$ to be disjoint from $Y$ in some way. We will prove that the map from a single summand $\Skein_\ideal^{\tilde h}(\Sigma\smallsetminus Y) \to \Skein_\ideal(\Sigma)$ is surjective, which will prove the statement. 

Consider any other $\tilde h' \in \pi_\Gr^{-1}(h)$ and let $\Graph'$ % $= \pi_G(\tilde\Graph')$ \NGm{Do we need  $= \pi_G(\tilde\Graph')$ here?}
be in the image of $\Skein_\ideal^{\tilde h'}(\Sigma\smallsetminus Y) \to \Skein_\ideal(\Sigma)$. Then $\tilde h'$ is a $\catd_\Gr$-graph % another of isotoping $h$ to be disjoint from $Y$ is 
related to $\tilde h$ by sliding finitely many edges of $h$ across a point of $Y$. % \NGm{Can we edit this sentence?}
  If an an edge colored by
  $g\in \Gr$ crosses a point $v\in Y$, then assuming there is a
  $W_l\in \ideal \cap \cat_{l}$ colored edge of $\Graph'$ near $v$, we apply the
  skein equivalence given in the following figure to $\Graph'$ where
  the coupons in the graph represents $W_l$ as a retract of
  $W_g\otimes W_{g^{-1}l}$ which exist since $l\in \Gr_{\ideal} = \Gr$ by our assumption:
% \end{minipage}
% \begin{minipage}{0.4\linewidth}
  \begin{equation}\label{eq:boundary}
  \epsh{fig7a}{14ex}\put(-20,5){\ms v}\put(0,5){\ms{W_l}}\quad\longrightarrow\epsh{fig7b}{14ex}\longrightarrow\quad\epsh{fig7c}{14ex}\putw{0}{50}{$\ms{W_g}$}\pute{100}{50}{$\ms{W_{g^{-1}l}}$}
\end{equation}
% \end{minipage}
The new skein is disjoint from $Y$ and has degree $\tilde h$ in $\Sigma\smallsetminus Y$, showing that $\Graph'$ is in the image of $\Skein_\ideal^{\tilde h}(\Sigma\smallsetminus Y) \to \Skein_\ideal(\Sigma)$. This concludes the proof.
%After all these modifications on $\Graph'$ we fuse the resulting graph as
%we did for $\Graph$ in the complement of $v$ then the resulting graph
%is $\Graph''$: indeed by these modifications, one ``follows'' the degree graph via skein equivalences on the level of $\Graph$.
   \end{proof}
%\BH{I found this proof hard to understand. I would put forward the fact that we have to choose a representative for $h$, for each representative it's clear it's finite, but there may be infinitely many representatives.}

%\begin{remark}
%One can prove in a similar way that if $\cat_\Gr$\BPm{why $\cat_G$ insted of $\cat$ ?}\FCm{I would say it should be $\cat$} is ribbon and $\ideal$ is a graded finite ideal of $\cat$ then similar statement to Theorem \ref{T:finitewithoutboundary} holds for compact 3-manifolds.  
%\end{remark}
%\BPm{To check that we can replace the remark by the theorem.  I don't see how to avoid the assumption $G=G_\ideal$...}
\begin{theorem}\label{T:finitewithoutboundaryribboncase}
  Let $\Gr$ be an abelian group and $\ideal$ a graded finite ideal of a sharp ribbon $\Gr$-graded
  category $\cat$.  Let $M$ be a compact oriented 3-manifold without boundary. Then for every $h\in H_1(M;\Gr)$ we have $\dim_{\kk}\Skein^h_\ideal(M)<\infty$.
\end{theorem}
\begin{proof}
%Consider a $t$ triangulation of $M$ whose vertices we denote $Y$, and denote $PD(t)$ the Poincaré-dual cellularization and $T$ the 1-skeleton of $PD(t)$. Since any ribbon graph can be made disjoint from $Y$, and any isotopy can be made to avoid $Y$, the map $\Skein_\ideal(M\smallsetminus Y)\tilde \to \Skein_\ideal(M)$ is an isomorphism. Since $T$ is the 1-skeleton of cellularization of $M$, any skein $\Graph$ in $M$ can be represented by a skein with underlying graph $T$. Now, $\Graph$ has fixed degree $h$, and for each edge $e$ of $T$ there is an edge $e'$ of $t$ intersecting it, and only it, once. Hence the color of the edge $e$ has fixed degree $[e'\cap h]_\Gr$. By the arguments above, $\Graph$ is skein equivalent to a graph whose edges are colored by the generators $V_{[e'\cap h]_\Gr}$, and the only choice is in the morphisms coloring the vertices, which is finite-dimensional. 
%
  Let $S$ be a special spine of the 3-manifold $M$, that is $S$ is a
  2-polyedra with: 1) $M\setminus S\simeq B^3$, 2) its 1-skeleton sing$(S)$ has no
  circular components (isomorphic to $S^1$), 3) $S\setminus$sing$(S)$
  is a disjoint union of discs $B^2$ called the regions of the spine.
  Then any $1$-cycle $z\in Z_1(M,\Gr)$ which does not intersect
  sing$(S)$ is Poincar\'e dual to a unique $2$-cocycle
  $z^*\in Z^2(S,\Gr)$. Let us fix a framed dot in each region of $S$ and
  fix a graph $\Graph$ made of a coupon in box $B$ disjoint from $S$,
  a braid in $M\setminus (S\cup B)\simeq S^2\times]0,1[$ and the dots
  on the regions.  Let us fix a 1-cycle representing $h$ and call
  $h^*$ the dual 2-cocycle on $S$. We color each edge of $\Graph$ with
  a generator $V_g$ of $\ideal_g$ where $g$ is given by the value of
  $h^*$ on the region intersected by the edge.  Filling the coupon of
  $\Graph$ by a morphism $f$, we obtain a skein
  $\Graph_f\in\Skein^h(M)$.  We claim that the map $f\mapsto \Graph_f$
  is surjective which will prove the finite dimensionality of
  $\Skein^h(M)$.  Indeed, given a skein in $\Skein^h(M)$, we can first
  assume that each region intersects an edge of $\Graph'$ colored by a
  color in $\ideal$. Then, if $z^*$ is a 2-cocyle of $S$ Poincar\'e dual
  to the degree of $\Graph'$, we have that $z^*-h^*$ is a $\Gr$-linear
  combination of coboundary supported by edges of sing$(S)$.  Using
  relations like \eqref{eq:boundary} where now $v$ represents an edge
  of sing$(S)$ we can modify $z^*$ so that $z^*=h^*$.  Finally, fusing
  the edges intersecting the same region, we can see that $\Graph'$ is
  equivalent to a skein which intersects $S$ as $\Graph$ and with the
  same color $V_g$ (since $V_g$ is a generator of
  $\cat_g\cap \ideal$).  By an isotopy in $M\setminus S$, we can now
  deform $\Graph'$ so that it become the same braid that $\Graph$
  outside the box $B$.  Finally a skein relation in $B$ reduce
  $\Graph'$ to some $\Graph_f$.
\end{proof}
%\section{Admissible skein modules for $\Gr$-chromatic categories}

\section{\texorpdfstring{Representation of the groupoid $\wt\cob_e$}{Representation of the groupoid Cob e}}
\subsection{The admissible skein module of a $\Gr$-decorated surface}
We now turn to the definition of a vector space
associated to a $ \Gr$-decorated surface.
       %        equipped with a $\Gr$-representation.
These vector spaces can be thought of as a generalization of the
construction of \cite{DST24} to non-semisimple and non-finite
settings.  They will be the vector space that the $\Gr$-HQFT of
Section \ref{S:GHQFT} assigns to $\Gr$-decorated surfaces.
% \begin{definition}\label{D:sk-dec-surface}
%   Let $\cat$ be a $\Gr$-graded pivotal $\kk$--category and $\Sigma$ be
%   a closed surface with a finite set $Y\subseteq \Sigma$ of
%   basepoints. Then by the above result the admissible skein module of
%   the punctured surface $\Skein_\ideal(\Sigma\smallsetminus Y)$ is
%   graded by $\Rep_\Gr(\Sigma,Y)$. For $\rho \in \Rep_\Gr(\Sigma,Y)$,
%   we will denote
%   $\Skein_\ideal^\rho(\Sigma\smallsetminus Y):=
%   \Skein_\ideal^h(\Sigma\smallsetminus Y)$ the graded component of
%   grading the unique $h\in H_1(\Sigma\smallsetminus Y)$ with
%   $[h\cap -]_\Gr = \rho(-)$.

%   We call the triple $\wt\Sigma=(\Sigma, Y,\rho)$ a
%   \emph{$\Gr$-decorated surface}. We define the \emph{admissible skein
%     module of the decorated surface}
%   $$\Skein(\wt\Sigma)=\Skein^{\rho}_{\ideal}(\Sigma\setminus Y)/\bs{\text{transparency relations}}$$
% to be the quotient of the admissible skein module of the punctured surface by the {\em transparency relation} where an edge with coloring of trivial degree can pass through any base point.
% %\BPm{(this ensure that the image of the  gauge cylinder are well defined, and that restriction cylinder are  compatible iso).}
% \end{definition}
\begin{definition}\label{D:sk-dec-surface}%\BPm{modified}
  Let $\cat$ be a sharp
  %\BPm{added sharp... maybe we don't need it for Proposition \ref{P:H1GisRepG}?} \NGm{I think we do not need it here because it is in the definition of graded?}\BPm{not anymore} 
  $\Gr$-graded pivotal $\kk$--category and $\Proj$ its ideal of projective objects.
  %\BPm{added $\ideal=\Proj$}  
  Let
  $\wt\Sigma=(\Sigma, Y,\rho)$ be a $\Gr$-decorated surface.  By
  Proposition \ref{P:H1GisRepG}, there exists a unique
  $h\in H_1(\Sigma\smallsetminus Y)$ with $[h\cap -]_\Gr =
  \rho(-)$. Recall the admissible skein module of the punctured
  surface $\Skein_\Proj(\Sigma\smallsetminus Y)$ is graded by
  $H_1(\Sigma\smallsetminus Y)$.  We define the \emph{admissible skein
    module of the $\Gr$-decorated surface}
  $$\Skein(\wt\Sigma)=\Skein_\Proj^h(\Sigma\setminus Y)/\bs{\text{transparency relations}}$$
  to be the quotient of the admissible skein module of the punctured
  surface by the {\em transparency relations} given by an edge with
  coloring of trivial degree can pass through any base point:
 %\NGm{ what are the transparency relations?} \NGm{make this more clear that this is the definition.}
%\BPm{(this ensure that the image of the  gauge cylinder are well defined, and that restriction cylinder are  compatible iso).}
\[\epsh{fig33y.pdf}{12ex}
\putne{50}{91}{$\ms{V\!\in\!\cat_1}$}
\putw{0}{50}{$\ms{p\,\mathsf x}$}\qquad=% \skeq
\qquad
  \epsh{fig33x.pdf}{12ex}
\pute{85}{90}{$\ms{V\!\in\!\cat_1}$}
\pute{97}{52}{$\ms{\mathsf x\,p}$}\qquad\in\Skein(\wt\Sigma)
  \]
  here the base point is denoted by $p$ and a cross.  %\NGm{Added.  Remove if ok.}
\end{definition}
This whole section aims at proving Theorem \ref{P:SkeinRepCobe} which in particular states that the isomorphism class of %\BPm{added} 
$\Skein(\wt\Sigma)$ only depends on the image of $\rho$ in $\Rep_\Gr(\Sigma,\emptyset)$. 
Roughly this theorem is a preparatory step to the construction of the TQFT in that it allows to treat in detail a subgroupoid of the category of cobordisms consisting in freely changing the basepoints and conjugating the morphisms.  %\FCm{slightly modified here}

\subsection{\texorpdfstring{The groupoid $\wt\cob_e$}{The groupoid Cob e}}
As examples of $\Gr$-decorated cobordisms, we have:
\begin{enumerate}
\item The trivial cylinder
  $C_{\Sigma,Y,\rho}:(\Sigma,Y,\rho)\to(\Sigma,Y,\rho)$ is the
  cylinder $\Sigma\times[0,1]$ equipped with the representation
  $\rho\circ\pi$ where $\pi:\Sigma\times[0,1]\to\Sigma$ is the
  projection.
\item A gauge cylinder for a $\Gr$-decorated surface $(\Sigma,Y,\rho)$ and
  for $\vp\in \Gr^Y$ is the $\Gr$-decorated cobordism
  $C_{\Sigma,Y,\rho}^\vp:(\Sigma,Y,\rho)\to(\Sigma,Y,\vp.\rho)$ where for
  $*\in Y$, the image of the path $*\times[0,1]$ from $*\times0$ to
  $*\times 1$ by the representation
  $\pi_1(\Sigma\times[0,1],Y\times\bs{0,1})\to\Gr$ is
  $\vp(*)^{-1}\in\Gr$.  % \BPm{To such a lift is associated the map
    % between skein modules which pass an edge colored in degree $\vp(*)$
    % through each base point $*\in Y$.}
  If $\vp_1,\vp_2\in \Gr^Y$, we have
  $C_{\Sigma,Y,\rho}^{\vp_2}\circ C_{\Sigma,Y,\rho}^{\vp_1}\simeq
  C_{\Sigma,Y,\rho}^{\vp_2\vp_1}$.  The support of $C_{\Sigma,Y,\rho}^{\vp}$
  is the support of $\vp$ \ie the set $\bs{y\in Y:\vp(y)\neq 1}$.
% \item A path cylinder
%   $C_{\Sigma,\rho,\gamma}:(\Sigma,*,\rho)\to(\Sigma,*',\rho' )$ where
%   $\gamma$ is a path embedded in $\Sigma$ from $*$ to $*'$ and
%   $\rho'=\rho(\gamma^{-1}\any\gamma)$.  \BP{To such a lift is associated
%   the map between skein modules which push the edges intersecting
%   $\gamma$ along $\gamma$.}   \BP{Maybe we don't need them (just use restriction cylinders).}
\item A restriction cylinder is obtained from a trivial cylinder by
  forgetting some base points in the source or target surfaces: if
  $Y_1,Y_2$ are finite subsets of $\Sigma$ with possible intersection
  and $\rho\in\Rep_\Gr(\Sigma,Y_1\cup Y_2)$ then
  $R_{\Sigma,Y_1,Y_2,\rho}:(\Sigma,Y_1,\rho_{|(\Sigma,Y_1)})\to(\Sigma,Y_2,\rho_{|(\Sigma,Y_2)})$
  is the cylinder $\Sigma\times[0,1]$ equipped with the representation
  $\rho\circ\pi$ where $\pi:\Sigma\times[0,1]\to\Sigma$ is the
  projection. %\BPm{added the notation $R_{\Sigma,Y_1,Y_2,\rho}$}
\end{enumerate}
\begin{definition}\label{def:equivalence} %\BPm{added the end of the section}
  A $\Gr$-decorated cobordism $\wt M:\wt\Sigma\to\wt\Sigma'$ between two
  $\Gr$-decorated surfaces $\wt\Sigma=(\Sigma,Y,\rho)$ and
  $\wt\Sigma'=(\Sigma,Y',\rho')$ with the same underlying surface is
  called an {\em equivalence of $\Gr$-decorated surfaces} if the underlying
  cobordism is diffeomorphic to the trivial cobordism $\Sigma\times[0,1]$ via a diffeomorphism which is the identity on the boundary. %\FCm{Detailed here}\BPm{ok}  
  Then the
  $\Gr$-decorated surfaces $\wt\Sigma$ and $\wt\Sigma'$ are called
  equivalent.
\end{definition}
Equivalences can be thought as the ``kernel'' of the forgetful functor
from $\wt\cob$ to the usual category $\cob$ of cobordisms.  Clearly,
gauge and restriction cylinders are equivalences of $\Gr$-decorated
surfaces.
\begin{remark} It is important to stress that mapping cylinders are not equivalences of
  $\Gr$-decorated surfaces: If $f:(\Sigma,Y,\rho)\to (\Sigma,Y,\rho)$
  is a diffeomorphism different from the identity, the trivial
  cylinder with parametrization of its boundary given by $\Id\sqcup f$
  is not an equivalence according to the above definition.
\end{remark}
\begin{definition}
  Let $\wt\cob_e$ be the subcategory of $\wt\cob$ with the same
  objects and whose maps are equivalences of $\Gr$-decorated surfaces.
\end{definition}
The following proposition shows that $\wt\cob_e$ is a groupoid
generated by gauge and restriction cylinders.
\begin{proposition}\label{P:eq-surf}
  Two $\Gr$-decorated surfaces $\wt\Sigma=(\Sigma,Y,\rho)$ and
  $\wt\Sigma'=(\Sigma,Y',\rho')$ are equivalent if and only if the
  restrictions of $\rho$ and $\rho'$ in $\Rep_\Gr(\Sigma,\emptyset)$
  are equal.  Furthermore, equivalences of $\Gr$-decorated surfaces are
  generated by gauge and restriction cylinders.
\end{proposition}
\begin{proof}
  The equivalence of $\Gr$-decorated surfaces is clearly an equivalence
  relation and is compatible with the disjoint union. Thus it is
  enough to consider the case of a connected surface.  % {\bf Let
  % $Y''=Y\cup Y'$ and choose extensions $\wb\rho,\wb\rho'$ of $\rho$
  % and $\rho'$ in $\Rep_\Gr(\Sigma,Y'')$. } \FCm{ I don't understand: suppose that $Y=Y'$ but $\rho\neq \rho'$. Then there is no extension.... Therefore I propose another proof here below: the previous is commented. }\BPm{François, you misunderstood: the meaning was: $\wb\rho$ extends $\rho$ and $\wb\rho'$ extends $\rho'$.  But your proof is simpler. Good job.}

  Let $y_0\in \Sigma\setminus (Y\cup Y')$, $Y_0=Y\sqcup \{y_0\}$,
  $Y'_0=Y'\sqcup \{y_0\}$, $\sigma=\{y_0\}\times[0,1]:(y_0,0)\leadsto (y_0,1)$ and choose
  extensions $\rho_0,\rho_0'$ respectively of $\rho,\rho'$ to
  $(\Sigma,Y_0)$ and $(\Sigma,Y'_0)$.

  If there exists an equivalence
  $(\Sigma\times[0,1],Y\times\bs{0}\sqcup Y'\times \bs{1},\rho'')$,
  then, by pre and post-composing with restriction cylinders
  $R_{\Sigma,Y_0,Y,\rho_0}$ and $R_{\Sigma,Y',Y'_0,\rho'_0}$, there is
  also one
  $(\Sigma\times[0,1],Y_0\times\bs{0}\sqcup Y'_0\times
  \bs{1},\rho'')$. Then for any loop $\gamma\in\pi_1(\Sigma,y_0)$ we
  have
  %\BPm{changed $\wb\rho $ with $\rho_0$}
  $\rho_0(\gamma)=\rho''(\gamma\times\bs0)=\rho''(\sigma(\gamma\times\bs1)\sigma^{-1})=
  \rho''(\sigma)\rho'_0(\gamma)\rho''(\sigma)^{-1}$. Thus the
  restrictions of $\rho_0$ and $\rho'_0$ coincide in
  $\Rep_\Gr(\Sigma,\emptyset)$ hence also the restrictions of $\rho$ and $\rho'$.%\BPm{added Hence also}
  
Reciprocally, suppose that there is $g\in \Gr$ such that %\BPm{changed $\wb\rho $ with $\rho_0$}
$\rho'_0(\gamma)=g\rho_0(\gamma)g^{-1},\ \forall \gamma\in \pi_1(\Sigma,y_0)$. Then an equivalence between $(\Sigma,Y_0,\rho_0)$ and $(\Sigma,Y'_0,\rho'_0)$ is $R_{\Sigma,\{y_0\},Y'_0,\rho}\circ C_{\Sigma,\{y_0\},\rho}^\vp \circ R_{\Sigma,Y_0,\{y_0\},\rho_0}$ where $\vp(y_0)=g$.%\BPm{changed $g^{-1}$ with $g$.  Please double check François}% \BPm{you mean there exists a $\vp$ like this and you don't think it needs a justification ?}  \FCm{Remark that I am pre and post composing with restriction morphisms so that the gauge cylinder in the middle has only the marked point $y_0$. Is there anything to justify here then ?}\BPm{Ok, I missed the point.}
\end{proof}
\subsection{\texorpdfstring{Presentation of $\wt\cob_e$}{Presentation of Cob e}}
If $\vp:Y\to\Gr$, $Y\subset Y'\subset\Sigma$ we will keep calling $\vp$ the
extension of $\vp$ to $Y'$ defined on $y\in Y'\setminus Y$ by
$\vp(y)=1_\Gr$.  We use the simplified notation
$J_\vp=C_{\Sigma,Y',\rho}^{\vp}$.
\begin{proposition}\label{P:JR}
  Any equivalence $\wt e=(\Sigma\times[0,1],Y,\rho_0)$ between two
  $\Gr$-decorated surfaces $\wt\Sigma_1=(\Sigma,Y_1,\rho_1)$ and
  $\wt\Sigma_2=(\Sigma,Y_2,\rho_2)$ can be written in a unique way
  $$\wt e=J_\vp\circ R(Y_1,Y_2,\rho)$$%\BPm{added the equality to avoid confusion in the order} 
    as the composition of a gauge cylinder $J_\vp$ with support in
  ${Y_1\cap Y_2}$ and a restriction cylinder
  $R(Y_1,Y_2,\rho)$.%\FCm{Reversed the order here}
  % associated to $f\in \Gr^{Y_1\cap Y_2}$.
  % Furthermore, 
  % $$R(Y_1,Y_2,\rho)\circ J_f=J_f\circ R(Y_1,Y_2,f\cdot\rho).$$
\end{proposition}
\begin{proof}
  It is enough to consider the case of a connected surface.  We define
  $\rho\in\Rep_\Gr(\Sigma,Y_1\cup Y_2)$ as follows~:\\ For any
  $\gamma:y\leadsto y'\in \pi_1(\Sigma,Y_1\cup Y_2)$,
  \begin{enumerate}
  \item if $y,y'\in Y_1$, then
    $\rho(\gamma)=\rho_0(\gamma_{11})=\rho_1(\gamma)$ where
    $\gamma_{11}(t)=(\gamma(t),0)$ ;
  \item if $y\in Y_1$, $y'\in Y_2\setminus Y_1$ then
    $\rho(\gamma)=\rho_0(\gamma_{12})$ where
    $\gamma_{12}(t)=(\gamma(t),t)$ ;
  \item if $y\in Y_2\setminus Y_1$, $y'\in Y_1$ then
    $\rho(\gamma)=\rho_0(\gamma_{21})$ where
    $\gamma_{21}(t)=(\gamma(t),1-t)$ ;
  \item if $y,y'\in Y_2\setminus Y_1$ then
    $\rho(\gamma)=\rho_0(\gamma_{22})=\rho_2(\gamma)$ where
    $\gamma_{22}(t)=(\gamma(t),1)$.
  \end{enumerate}
  Then we claim that $\rho$ is a well defined representation in
  $\Rep_\Gr(\Sigma,Y_1\cup Y_2)$.  Indeed remark that if
  $\gamma''=\gamma\circ\gamma'\in\pi_1(\Sigma,Y_1\cup Y_2)$, then
  there exist $i,j,k\in\bs{1,2}$ such that
  $\rho(\gamma)=\rho_0(\gamma_{ij})$,
  $\rho(\gamma')=\rho_0(\gamma'_{jk})$ and
  $\rho(\gamma'')=\rho_0(\gamma''_{ik})$ where
  $\gamma''_{ik}\simeq\gamma_{ij}\circ\gamma'_{jk}$. (In words: each oriented path $\gamma$  in $\Sigma$ with endpoints in $Y_i$ and $Y_j$ can be mapped canonically to one in $\Sigma\times [0,1]$ of the form $\gamma_{ij}$. Then $\rho$ is the pull back of $\rho_0$ through this map.) %\FCm{Added this. Do you agree?}\BPm{yes}  
   Hence
  $\rho(\gamma'')=\rho(\gamma)\rho(\gamma')$ and $\rho$ is a
  representation.

  Now, for $y\in Y_1\cap Y_2$, let $\gamma_y:(y,0)\leadsto(y,1)$ be
  the path defined by $\gamma_y(t)=(y,t)$, and let
  $\vp:Y_1\cap Y_2\to\Gr$ be given by $\vp(y)=\rho_0(\gamma_y)^{-1}$.  Recall
  we still call $\vp: Y_2\to\Gr$ the extension by $1_\Gr$ of
  $\vp:Y_1\cap Y_2\to\Gr$.  We now show that
  $$\wt e\simeq J_\vp\circ R(Y_1,Y_2,\rho).$$
  If $y,y'\in Y_2$ and $\gamma:y\leadsto y'$ is a path in $\Sigma$, then
  $\rho(\gamma)=\rho_0(\gamma_{ij})$ where $(i,j)$ depends on if
  $y,y'$ belong to $Y_1$.  In particular,
  $\gamma_{11}=\gamma_y\circ(\gamma\times\bs1)\circ\gamma_{y'}^{-1}$,
  $\gamma_{12}=\gamma_y\circ(\gamma\times\bs1)$,
  $\gamma_{21}=(\gamma\times\bs1)\circ\gamma_{y'}^{-1}$ and
  $\gamma_{22}=\gamma\times\bs1$.  In any case,
  $\rho(\gamma)=\vp(y)^{-1}\rho_2(\gamma)\vp(y')=(\vp^{-1}.\rho_2)(\gamma)$ and
  $$R(Y_1,Y_2,\rho):(\Sigma,Y_1,\rho_1)\to(\Sigma,Y_2,\vp^{-1}.\rho_2).$$
  Finally let $y_0\in Y_1\times\bs0$. The $\Gr$-decorated cobordisms $\wt e$
  and $J_\vp\circ R(Y_1,Y_2,\rho)$ have the same source and target, they
  coincide with $\rho_1$ on
  $\pi_1(\Sigma\times[0,1],y_0)\simeq\pi_1(\Sigma,y_0)$, and they
  coincide on a tree made by pathes from $y_0$ to any other point of
  $Y$.  Hence, by Lemma \ref{L:Rep}, they are the same representation.

  This also prove the unicity since for $y\in Y_1\cap Y_2$, $\vp(y)$ has
  to be the image of $\gamma_y^{-1}$ by the representation $\rho_0$
  and $R(Y_1,Y_2,\rho)=J_{\vp^{-1}}\circ\wt e$.
% the representation $\rho\in\Rep_\Gr(\Sigma,Y_1\cup Y_2)$ is
  % uniquely determined by the representation of 
\end{proof}
%\BPm{I changed the order of the next Lemma and the Prop}
\begin{lemma}\label{L:unique-eq}
  For fixed $\Gr$-decorated surfaces $\wt\Sigma=(\Sigma,Y_0,\rho_0)$,
  $\wt\Sigma_1=(\Sigma',Y_1,\rho_1)$,
  $\wt\Sigma_2=(\Sigma',Y_2,\rho_2)$, let
  $\wt C_1=(M,\wb Y_1,\wb \rho_1):\wt\Sigma\to\wt\Sigma_1$ and
  ${\wt C_2=(M,\wb Y_2,\wb \rho_2):\wt\Sigma\to\wt\Sigma_2}$ be two $\Gr$-decorated
  cobordisms where $\wb \rho_1$, $\wb \rho_2$ have the same restriction
  $\wb \rho_0\in\Rep_\Gr(M,Y_0)$.
  %ith
  %${(\rho_1)_{|(M,Y_0)}=(\rho_2)_{|(M,Y_0)}}$.
  % the same underlying manifold.
  % If ${(\rho_1)_{|(M,Y_0)}=(\rho_2)_{|(M,Y_0)}}$ t
  Then there exists a unique
  equivalence of surfaces ${\wt e:\wt\Sigma_1\to\wt\Sigma_2}$ such that
  $\wt C_2=\wt e\circ\wt C_1$.
\end{lemma}
\begin{proof}%\BPm{I changed the proof}
  First the general case can easily be deduced from the case where $M$
  is connected which we will assume in this proof.  Then if
  $\Sigma=\emptyset$, we fix $\dot M=M\setminus B^3$ obtained by
  removing a small 3-ball in the interior of $M$, we put a base point
  $y_0$ in the 2-sphere and choose extensions $\dot\rho_1,\dot\rho_2$
  of $\wb\rho_1,\wb\rho_2$ ($\dot\rho_i\in\Rep_\Gr(\dot M,\wb Y_i\cup\bs{y_0})$
  such that
  $(\dot\rho_1)_{|(\dot M,\bs{y_0})}=(\dot\rho_2)_{|(\dot M,\bs{y_0})}$.  This
  shows that it is enough to consider the case in which
  $Y_0\neq\emptyset$.  Then we fix a base point $y_0\in Y_0$.  

  Let $Y_3\subset\Sigma'$ be a subset with exactly one point per
  connected component and choose any extension
  $\rho_3\in\Rep_\Gr(\Sigma',Y_3)$ of the restriction of
  $(\wb\rho_0)_{|\pi_1(\Sigma',\emptyset)\to\Gr}$.  % to $\pi_1(\Sigma',\emptyset)$.
  By Proposition \ref{P:eq-surf}, the $\Gr$-decorated surface $\wt \Sigma_1$
  and $\wt \Sigma_2$ are equivalent to
  $\wt\Sigma_3=(\Sigma',Y_3,\rho_3)$ with equivalences
  $\wt e_i:\wt\Sigma_i\to\wt\Sigma_3$ for $i=1,2$. Hence it is enough
  to show that there exists a unique equivalence
  $\wt e':\wt\Sigma_3\to\wt\Sigma_3$ such that
  $\wt e_2\circ\wt C_2=\wt e'\circ\wt e_1\circ\wt C_1$ and the desired
  equivalence will then be $\wt e=(\wt e_2)^{-1}\circ\wt e'\circ\wt e_1$.  Thus
  it suffices to consider only the case where
  $\wt\Sigma_1=\wt\Sigma_2=\wt\Sigma_3$ which we assume now.

  For any $y\in Y_3$, let $\gamma_y:y_0\leadsto y$ be a path in $M$
  and define $\vp:Y_3\to\Gr$ by
  $\vp(y)=\rho_2(\gamma_y)^{-1}\rho_1(\gamma_y)$.  By Proposition
  \ref{P:JR}, if the equivalence $\wt e:\wt\Sigma_3\to\wt\Sigma_3$
  exists, then it is a gauge cylinder,
  $\wt e=\wt C^\psi_{\Sigma',Y_3}$.  Then, the path $\gamma_y$ in
  $\wt C_2$ corresponds in the composition $\wt e\circ \wt C_1$ (i.e. it is homotopic)
  %\FCm{Added this}
  , to the
  composition of $\gamma_y$ in $\wt C_1$ with $y\times[0,1]$ in
  $\wt e$. This forces $\rho_2(\gamma_y)=\rho_1(\gamma_y)\psi(y)^{-1}$
  and $\psi=\vp$ is unique.  Conversely, let
  $\wt e=\wt C^\vp_{\Sigma',Y_3}$, then the representations of
  $\wt e\circ \wt C_1$ and of $\wt C_2$ coincide on $\pi_1(M,Y_0)$ and
  on every path in $\bs{\gamma_y}_{y\in Y_3}$, so by Lemma
  \ref{L:Rep}, they are equal and $\wt e\circ \wt C_1=\wt C_2$.  This
  shows the existence and unicity of $\wt e$.
\end{proof}
% \BP{Maybe remove ?
%   \begin{remark}
%   Lemma \ref{L:unique-eq} has an obvious counterpart for $\Gr$-decorated
%   corbordisms with the same target, 
%   $\wt C_1=(M,Y_1,\rho_1):\wt\Sigma_1\to\wt\Sigma \et\wt
%   C_2=(M,Y_2,\rho_2):\wt\Sigma_2\to\wt\Sigma$ related by
%   $\wt C_2=\wt C_1\circ\wt e$.
% \end{remark}
% \begin{corollary}
%   Let $\wt C_1=(M,Y_1,\rho_1)$, $\wt C_2=(M,Y_2,\rho_2)$ be two
%   $\Gr$-decorated cobordisms with the same underlying cobordism, such that
%   $\rho_1$ and $\rho_2$ restrict to the same element in
%   $\Rep_\Gr(M,\emptyset)$ then there exists equivalences
%   $\wt e_s,\wt e_t$ such that
%   $\wt C_2=\wt e_t\circ\wt C_1\circ\wt e_s$.
% \end{corollary}
% If $g\in\Gr$, let $g^{\delta_y}$ be a $\Gr$-valued map with support
% $\bs y$ such that $g^{\delta_y}(y)=g$ (we use the same notation for
% maps whose sources are different sets of base points).}
\begin{lemma}\label{L:rel_eq}
  The category $\wt\cob_e$ is the groupoid generated by gauge and restriction
  cylinders with relations:
  \begin{enumerate}
  \item $J_{\vp_1}\circ J_{\vp_2}=J_{\vp_1\vp_2}$, for any $\vp_1,\vp_2\in \Gr^Y$,
       %        where $J_\vp$ is a gauge
  % cylinder associated to a $\Gr$-valued map $\vp$,
  \item
    $R(Y_1,Y_2,\rho)\circ J_{\vp}=J_{\vp_{|Y_2}}\circ
    R(Y_1,Y_2,\vp^{-1}\cdot\rho)$ for any $\vp\in \Gr^{Y_1}$ if
    $Y_2\subset Y_1$,
       %        \item $R(Y_1,Y_2,\rho)\circ J_{\vp_{|Y_1}}=J_{\vp_{|Y_2}}\circ R(Y_1,Y_2,\vp^{-1}\cdot\rho)$ for any $\vp:Y_1\cup Y_2\to\Gr$,
  \item $R(Y_2,Y_3,\rho_{|Y_2})\circ R(Y_1,Y_2,\rho)=R(Y_1,Y_3,\rho)$ if $Y_3\subset Y_2\subset Y_1$,
  \item $R(Y_2,Y_1,\rho)= R(Y_1,Y_2,\rho)^{-1}$ if $Y_2\subset Y_1$,
  \item $R(Y_1\cup Y_2,Y_2,\rho)\circ R(Y_1,Y_1\cup Y_2,\rho)=R(Y_1,Y_2,\rho)$.
  % \item
  %   $R(Y_2,Y_2\cup\bs y,\rho')\circ R(Y_1,Y_2,\rho)=R(Y_1,Y_2\cup\bs
  %   y,\rho'')$ if $Y_2\subset Y_1$ and $y\notin Y_1$, where $\rho''$
  %   is the unique representation that coincide with $\rho$ on
  %   $Y_1$ and with $\rho'$ on $Y_2\cup\bs y$,
  %   %is uniquely determined by the source and target,
  % \item
  %   $R(Y,Y\cup\bs y,\rho')\circ R(Y\cup\bs
  %   y,Y,\rho)=J_{g^{\delta_y}}$ if
  %   $y\notin Y$, where $g=\rho'(\gamma)\rho(\gamma)^{-1}$ for any path
  %   $\gamma:y\leadsto y'\in Y\setminus\bs y$.\BP{not needed}
  % % \item $R(Y_1,Y_2,\rho)\circ J_\vp=J_{\vp_{|Y_1}}\circ R(Y_1,Y_2,\vp\cdot\rho)$,
  % %\item $R(Y_2,Y_3,\rho')\circ R(Y_1,Y_2,\rho)=R(Y_1,Y_3,\rho)$ if $Y_3\subset Y_2\subset Y_1$,
  % % \item $R_1R_2=R_3J_1$,
  \end{enumerate}
  In relations (1) and (2), the maps are assumed to be composable and
  both sides of the equalities have the same source and target.
\end{lemma}
\begin{proof}
  We show that these relations hold.  First it is enough to consider
  the case of a connected surface $\Sigma$.  By Lemma \ref{L:Rep}, an
  equivalence of connected $\Gr$-decorated surface is uniquely determined by
  its source and target together with the image by the representation
  of any path from a base point of the source to a base point of the
  target.  In all relations (1)--(5), the two sides have the same
  source and target.  Then to show these relations hold, it is enough
  to check that their representations coincide on a well chosen path
  $\gamma_{12}$.  Relation (1) was already established. For relation
  (2), we can choose $\gamma_{12}(t)=(y,t)$ for any $y\in Y_2$ which
  is sent to $\vp(y)^{-1}$.%\FCm{Shouldn't it be $y\in Y_1$ ?}\BPm{no I don't think, else $\gamma_{12}$ would not finish on a base point} 
  For relation (3), we can choose
  $\gamma_{12}(t)=(y,t)$ for any $y\in Y_3$ which is sent to $1$.  For
  relation (4) rewitten $R(Y_1,Y_2,\rho)R(Y_2,Y_1,\rho)=\Id$, we can
  choose $\gamma_{12}(t)=(y,t)$ for any $y\in Y_2$ which is sent to
  $1$.  For relation (5), we choose in $\Sigma$ a path
  $\gamma:y_1\in Y_1\leadsto y_2\in Y_2$; then we set
  $\gamma_{12}(t)=(\gamma(t),t)$ which is homotopic to the
  concatenation of $\gamma_1(t)=(y_1,t/2)$ with a path
  $\gamma_2(t)=(\gamma(t),\frac{1+t}2)$; for both side of relation
  (5), the image of $\gamma_{12}$ is $\rho(\gamma)$.  Hence all the relations hold.

  We now show that this set of relations is complete.  First,
  relations (1) and (3) imply that $J_{y\mapsto1_\Gr}$ and
  $R(Y,Y,\rho)$ are the identity because they are idempotents and isomorphisms.%\FCm{Why an idempotent is necessarily the identity ?}\BPm{because its a groupoid presentation(i.e all morphisms are invertible). I add and isomorphisms, is it ok ?}  
  
  Also
  relation (4) implies that $R(Y_1,Y_2,\rho)$ and $R(Y_2,Y_1,\rho)$
  are inverse isomorphisms when $Y_2\subset Y_1$. Then for general
  $Y_1,Y_2$, we have:
  \begin{align*}
      R(Y_1,Y_2,\rho)&\overset{(5)}=R(Y_1\cup Y_2,Y_2,\rho)R(Y_1,Y_1\cup Y_2,\rho)\\
  &\overset{(4)}=(R(Y_1\cup Y_2,Y_1,\rho)R(Y_2,Y_1\cup Y_2,\rho))^{-1}\overset{(5)}=R(Y_2,Y_1,\rho)^{-1}.
  \end{align*}
  Next we can generalize relation (2) to
  \begin{equation}
    \tag{2'}R(Y_1,Y_2,\rho)\circ J_{\vp_{|Y_1}}=J_{\vp_{|Y_2}}\circ
  R(Y_1,Y_2,\vp^{-1}\cdot\rho),
  \end{equation}
  % (2'):
  % $$
  for any $\vp:Y_1\cup Y_2\to\Gr$.
  Indeed relation (5) allows to decompose $R(Y_1,Y_2,\rho)$ as
  $R(Y_1\cup Y_2,Y_2,\rho)\circ R(Y_1,Y_1\cup Y_2,\rho)$, then
  relation (2) describes how gauge generators commute with these
  restrictions.

  To show that the set of relations is complete, it is enough to show
  that the product of two elements in normal form (i.e. as in
  Proposition \ref{P:JR}) % \FCm{Added this}\BPm{Good. Should have we
    % defined ``normal form'' in Proposition \ref{P:JR}?}
  can be reduced
  to its normal form using these relations. Let
  $\vp=\vp_1\vp_2:Y_2\to\Gr$ where $\vp_1$ has support in
  $Y_1\cap Y_2$ and $\vp_2$ has support in $Y_2\setminus Y_1$; by
  relation (1) and (2'), we have
  $$J_\vp R(Y_2,Y_1,\rho)=J_{\vp_1}R(Y_2,Y_1,\vp_2.\rho),$$
  which is a normal form.  Let $J_i$ be gauge generators and
  $R_i$ be restriction cylinders.  Then by relation (2') a product
  $J_1R_1J_2R_2$ can be reduced to the form
  $J_1J_3R_3R_2$.  Hence it is enough to show that a product of two
  restrictions can be reduced to its normal form since then, we can
  write $R_3R_2=J_4R_4$, and any product
  $J_1J_3J_4R_4$ can be reduced to its normal form using (1).  %\FCm{I don't understand don't we have $J_1J_2=J_{12}$ by    $(1)$ ? Why having $J_4$ here ?}  \BPm{I changed the formulation}
    % \BP{old
    % to remove: Any product $J_1J_3J_4R_4$ can be reduced to its normal
    % form thus it is enough to show that a product of two restrictions
    % can be reduced to its normal form.}

  First if $Y_\cup:=Y_1\cup Y_3\subset Y_2$, then we have
\begin{align*}
  R(Y_2,Y_3,\rho)R(Y_1,Y_2,\rho)
  &\overset{(3)}=R(Y_\cup,Y_3,\rho)R(Y_2,Y_\cup,\rho)R(Y_\cup,Y_2,\rho)R(Y_1,Y_\cup,\rho)\\
  &\overset{(4)}=R(Y_\cup,Y_3,\rho)R(Y_1,Y_\cup,\rho)\overset{(5)}=R(Y_1,Y_3,\rho).\qquad(R6)
       %        We denote this new relation by $(R6)$.\\
\end{align*}
Consider a product $R(Y_2,Y_1,\rho')\circ R(Y_1,Y_2,\rho)$ with $Y_2\subset Y_1$.  Then there exists $\vp:Y_1\to\Gr$ such that $\rho'=\vp\cdot\rho$. Then,
\begin{align*}
  R(Y_2,Y_1,\rho')R(Y_1,Y_2,\rho)
  &=R(Y_2,Y_1,\vp.\rho)R(Y_1,Y_2,\rho)\\
  &\overset{(2')}= J_{\vp_{|Y_2}} R(Y_2,Y_1,\rho)R(Y_1,Y_2,\rho)=J_{\vp_{|Y_2}}.\qquad(R7)
\end{align*}
    Now if $Y_2\subset Y_\cap:=Y_1\cap Y_3$, we have
\begin{align*}
  R(Y_2,Y_3,\rho')R(Y_1,Y_2,\rho)
  &\overset{(3)}= R(Y_\cap,Y_3,\rho')R(Y_2,Y_\cap,\rho')R(Y_\cap,Y_2,\rho)R(Y_1,Y_\cap,\rho)\\
  &\overset{(R7)}=R(Y_\cap,Y_3,\rho')J_\vp R(Y_1,Y_\cap,\rho) \text{ for some $\vp:Y_2\to\Gr$}\\
  &\overset{(2')}=J_\vp R(Y_\cap,Y_3,\rho'')R(Y_1,Y_\cap,\rho);
\end{align*}
since $\rho$ and $\rho''$ coincide on $Y_\cap=Y_1\cap Y_3$, there exists
a unique extension to $Y_\cup=Y_1\cup Y_3$ which restrict to $\rho$ on
$Y_1$ and to $\rho''$ on $Y_3$.  Then
$R(Y_\cap,Y_3,\rho'')R(Y_1,Y_\cap,\rho)=R(Y_\cup,Y_3,\wb\rho)R(Y_\cap,Y_\cup,\wb\rho)R(Y_\cup,Y_\cap,\wb\rho)R(Y_1,Y_\cup,\wb\rho)=R(Y_1,Y_3,\wb\rho)$
%\FCm{Corrected first R here (previous commented)}\BPm{thanks}
%{$R(Y_3,Y_\cap,\rho'')R(Y_1,Y_\cap,\rho)=R(Y_\cup,Y_3,\wb\rho)R(Y_\cap,Y_\cup,\wb\rho)R(Y_\cup,Y_\cap,\wb\rho)R(Y_1,Y_\cup,\wb\rho)=R(Y_1,Y_3,\wb\rho)$}
and thus we deduce an eighth relation $(R8)$:
$R(Y_2,Y_3,\rho')R(Y_1,Y_2,\rho)=J_\vp R(Y_1,Y_3,\wb\rho)$ when
$Y_2\subset Y_1\cap Y_3$.
Finally, for any $Y_1,Y_2,Y_3$, we have
\begin{align*}
  R(Y_2,Y_3,\rho')&R(Y_1,Y_2,\rho)\\
  &\overset{(5)}= R(Y_2\cup Y_3,Y_3,\rho')R(Y_2,Y_2\cup Y_3,\rho')R(Y_1\cup Y_2,Y_2,\rho)R(Y_1,Y_1\cup Y_2,\rho)\\
  &\overset{(R8)}=R(Y_2\cup Y_3,Y_3,\rho')J_\vp R(Y_1\cup Y_2,Y_2\cup Y_3,\wb\rho) R(Y_1,Y_1\cup Y_2,\rho),\\
  &\overset{(2'),(3)}=J_\vp R(Y_2\cup Y_3,Y_3,\wb\rho_{|Y_2\cup Y_3}) R(Y_1,Y_2\cup Y_3,\wb\rho),\\
  &\overset{(3),(R6)}=J_\vp R(Y_1,Y_3,\wb\rho_{|Y_1\cup Y_3}).
\end{align*}
This completes the proof that the system of relations is complete.
\end{proof}
\subsection{\texorpdfstring{The representation of $\wt\cob_e$ in $\Skein$}{The representation of Cob e in $\Skein$}}
%\subsection{Cylinder maps}
We assume $\cat$ is a sharp $\Gr$-graded pivotal category.

\begin{figure}
  \begin{equation*}
  %\label{eq:create-V-skein}
  \epsh{fig33d.pdf}{10ex}
  \putn{33}{60}{${\gamma}$}
  \putw{50}{86}{$\ms{V_1}$}
  \pute{66}{87}{$\ms{V_2}$}
  \putw{96}{87}{$\ms{P}$}
  \putc{81}{61}{${\cdots}$}
  \quad\skeq\quad
  \epsh{fig33e.pdf}{10ex}
  \putn{9}{69}{$\ms{V}$}
  \putnw{48}{90}{$\ms{V_1}$}
  \putne{64}{89}{$\ms{V_2}$}
  \putnw{92}{87}{$\ms{P}$}
  \putw{48}{10}{$\ms{V_1}$}
  \pute{64}{10}{$\ms{V_2}$}
  \putw{92}{9}{$\ms{P}$}
  \putc{71}{50}{${e_{V,[\gamma\cap \Graph]_\cat}}$}
  \putc{78}{85}{${\cdots}$}
  \putc{80}{11}{${\cdots}$}
\end{equation*}
If $V\in\cat$ and $\gamma$ is a path transverse to a skein $\Graph$ which
ends on the left of an edge of $\Graph$ colored by a projectif object $P$, then the above
skein relation, in a tubular neighborhood of $\gamma$, create at
$\gamma(0)$ an edge colored by $V$.  Here $Q=[\gamma\cap \Graph]_\cat$ is a
projective object and $e_{V,Q}$ is a morphism as in Lemma
\ref{L:V-appear}.  $Q$ is not null unless a color of $T$ is null.
  \caption{Locally creating an edge colored by $V$. }\label{fig:V-appear}
\end{figure}

\begin{proposition}\label{P:gauge-map}
  Let $(\Sigma,Y,\rho)$ be a $\Gr$-decorated surface and let $\vp\in\Gr^Y$.
  There exists {\em gauge isomorphisms}:
  $$\vp_*:\Skein(\Sigma,Y,\rho)\to\Skein(\Sigma,Y,\vp.\rho),$$
  which satisfy $\vp_*\circ \vp'_*=(\vp\vp')_*$ for any $\vp,\vp'\in\Gr^Y$.
\end{proposition}
\begin{proof}
  If $p\in Y$ and $g\in\Gr$, let $g^{\delta_p}\in\Gr^Y$ be defined on
  $q\in Y$ by
  \[g^{\delta_p}(q)=\left\{
      \begin{array}{l}
        g\ \text{ if }q=p,\\1_\Gr\text{ else.}
      \end{array}\right.
  \]
  We first define $(g^{\delta_p})_*$ by the following: If $\Graph$ is
  a skein representing an element of
  $\Skein(\Sigma,Y,\rho)$, % let $D$ be a small disk in $\Sigma$ centered on $p$ which does not intersect $T$.
  we say that a path $\alpha$ is a gauge path if $\alpha$ is a path
  from $p$ to a point on an edge $e$ of $\Graph$,
  $\alpha([0,1])\cap \Graph=\alpha(1)$, the intersection is transverse
  and $[\alpha\cap e]_\Gr=g^{-1}$.  First, any non trivial skein is
  skein equivalent to a skein which has a gauge path.  Indeed, any non
  trivial skein is skein equivalent to a skein which
  has % an edge colored by a projective
  % object close to the base point $p$, then we can use Lemma
  % \ref{L:V-appear} to create
  an edge close to $p$ colored by a non-null object of degree $g$
  (see Figure \ref{fig:V-appear}).  %\FCm{If $p$ is encircled by a circle colored by a null object this is not true. But I guess we can prove that in this case the skein is the $0$ skein...}\BPm{you are right, I added non trivial} \FCm{I think we need a lemma stating that if a color of an edge of $\Graph$ is null then the skein is null...} 
  Given a skein $\Graph'$ with a gauge
  path $\alpha$, we obtain a skein $(\Graph')_\alpha$ by modifying $\Graph'$ in
  a tubular neighborhood of $\alpha$ to pass the edge $e$ through the
  base point.
  \[\epsh{fig33g.pdf}{10ex}
    \pute{14}{85}{$\ms{V\in\cat_g}$}
    \putn{41}{49}{$\ms{\alpha}$}
    \puts{96}{39}{$\ms{p}$} \quad\leadsto\quad \epsh{fig33h.pdf}{10ex}
    \putnw{80}{47}{$\ms{p}$} \putne{15}{85}{$\ms{V}$}
\]

  We claim that the class of $(\Graph')_\alpha$ is in
  $\Skein(\Sigma,Y,g^{\delta_p}.\rho)$ %\BPm{to justify?}\FCm{I would not justify this: it seems clear to me}\BPm{ok} 
  and that it is
  independent of $\alpha$.  Indeed, if $\alpha$ and $\alpha'$ are two
  gauge paths that only intersect at the base point $p$, we have a
  sequence of skein equivalences between $(\Graph')_\alpha$ and
  $(\Graph')_{\alpha'}$ given by:%\BP{maybe reverse arrows in the picture}
  % \[\epsh{fig1a}{10ex}\to\epsh{fig1b}{10ex}\skeq\epsh{fig1c}{10ex}\skeq
  %   \epsh{fig1d}{10ex}\skeq\epsh{fig1e}{10ex}\]
  \[\epsh{fig33i.pdf}{10ex}
    \putn{29}{49}{$\ms{\alpha}$}
    \putn{72}{55}{$\ms{\alpha'}$}
    \pute{10}{86}{$\ms{V}$}
    \putw{91}{12}{$\ms{V'}$}
    \puts{53}{40}{$\ms{p}$}
    \quad\leadsto\quad\]
  \begin{equation}
    \label{eq:fus-trans}
  \epsh{fig33j.pdf}{10ex}
    \putw{46}{43}{$\ms{p}$}
    \pute{9}{85}{$\ms{V}$}
    \putw{91}{14}{$\ms{V'}$}
    \ \skeq\
    \epsh{fig33k.pdf}{10ex}
    \putn{18}{83}{$\ms{V}$}
    \putn{15}{24}{$\ms{V}$}
    \putw{88}{8}{$\ms{V'}$}
    \pute{84}{49}{$\ms{V\otimes{V'}^*}$}
    \putw{87}{93}{$\ms{V'}$}
    \puts{47}{38}{$\ms{p}$}
    \putc{79}{71}{$\ms{\Id}$}
    \putc{79}{28}{$\ms{\Id}$}
    \qquad\ \skeq\ \qquad
    \epsh{fig33l.pdf}{10ex}
    \putw{3}{89}{$\ms{V}$}
    \putw{2}{6}{$\ms{V}$}
    \putw{16}{48}{$\ms{V\otimes{V'}^*}$}
    \puts{89}{78}{$\ms{V'}$}
    \putn{86}{29}{$\ms{V'}$}
    \puts{52}{52}{$\ms{p}$}
    \putc{19}{72}{$\ms{\Id}$}
    \putc{19}{29}{$\ms{\Id}$}
    \ \skeq\
  \epsh{fig33m.pdf}{10ex}
  \pute{10}{85}{$\ms{V}$}
  \pute{55}{57}{$\ms{p}$} \putw{91}{14}{$\ms{V'}$}
  \end{equation}
where the middle equivalence $\skeq$ is a transparency relation.  Now
we suppose by induction that $(T')_\alpha=(T')_{\alpha'}$ whenever
$\alpha\cap \alpha'$ has less than $k$ transverse intersections
(including $\{p\}$) and claim that if $\alpha\cap \alpha'$ has $k$
intersections then $(T')_\alpha=(T')_{\alpha'}$. Indeed let $\alpha''$
be the path from $p$ to $\alpha'\cap T$ which is equal to $\alpha$ on
the interval from $p$ to the ``last point'' $q$ of
$\alpha\cap \alpha'$ i.e. the point such that the open subarc of
$\alpha'$ connecting $q$ and $\alpha'\cap T$ does not intersect
$\alpha$.  Then up to slightly perturbing this latter arc, both
$\alpha'\cap \alpha''$ and $\alpha\cap \alpha''$ are composed by less
than $k$ points and by induction we have
$(T')_\alpha=(T')_{\alpha''}=(T')_{\alpha'}$.%\FCm{added this}\BPm{ok}

Now if $\sum_ia_i\Graph_i=0\in\Skein(\Sigma,Y,\rho)$ is a skein
relation in a box $B$ which does not intersect $\alpha$, we can use
the same modifications outside $B$ and use the same gauge path to get
a skein relation
$\sum_ia_i(\Graph'_i)_\alpha=0\in\Skein(\Sigma,Y,g^{\delta_p}.\rho)$.
If it intersects $\alpha$ then we can change $\alpha$ to another path
not intersecting $B$, with the same conclusion.  %\FCm{Added this: should we detail why this is true ?(I would rather not...}\BPm{I agree with no more detail}
Finally, two skeins related by a transparency relation at a base point $q$ have their
image related by a transparency relation. This is clear when $p\neq q$ and in the case $p=q$, we have for $U$ in $\cat_1$ and $V,V'\in\cat_g$:
\[
  \begin{array}{ccc}
    \epsh{fig33v.pdf}{10ex}
    \putw{0}{88}{$\ms{V}$}
    \putw{43}{88}{$\ms{U}$}
    \putw{90}{13}{$\ms{V'}$}
    \puts{70}{60}{$\ms{\overset{\mathsf x}{p}}$}
    &\tto{\text{transparency}}
    & \epsh{fig33v.pdf}{10ex}
      \putw{0}{88}{$\ms{V}$}
      \putw{43}{88}{$\ms{U}$}
      \putw{90}{13}{$\ms{V'}$}
      \puts{30}{60}{$\ms{\overset{\mathsf x}{p}}$}
    \\\\g^\delta_*\downarrow&&\downarrow g^\delta_*\\\\
    \epsh{fig33v.pdf}{10ex}
    \putw{0}{88}{$\ms{V}$}
    \putw{43}{88}{$\ms{U}$}
    \putw{90}{13}{$\ms{V'}$}
    \puts{105}{60}{$\ms{\overset{\mathsf x}{p}}$}
    &\tto{\text{transparency}}
    & \epsh{fig33v.pdf}{10ex}
      \putw{0}{88}{$\ms{V}$}
      \putw{43}{88}{$\ms{U}$}
      \putw{90}{13}{$\ms{V'}$}
      \puts{-5}{60}{$\ms{\overset{\mathsf x}{p}}$}
  \end{array}
\]
where the last transparency relation involve a fusion of the strands
similar to Equality \eqref{eq:fus-trans}. % the two involved base point are distincts
% and for the same (***)
%\FCm{Why these three asterisks? To me this does  not need justification}\BPm{I am not sure this is true if the  transparency relation is at the same vertex $p$: we also need to merge/unmerge the two $g$-colored edges... I added some detail}  
  Hence
$\Graph\mapsto (\Graph')_\alpha$ defines a map $(g^{\delta_p})_*$
which clearly has an inverse $((g^{-1})^{\delta_p})_*$.  Since the
construction is local, it is clear that if $p\neq q$ are distinct base
points, $(g^{\delta_p})_*$ and $(h^{\delta_q})_*$ commute for any
$g,h\in\Gr$.  This gives a natural definition of $\vp_*$ by
$\prod_{p\in Y}\vp(p)^{\delta_p}$. Also
$(g^{\delta_p})_*\circ(h^{\delta_p})_*=((gh)^{\delta_p})_*$ since
passing an edge colored by $V\in\cat_h$, then an edge colored by
$U\in\cat_g$ through the base point $p$ is skein equivalent to passing
an edge colored by $U\otimes V\in\cat_{gh}$.  This implies the last
statement of the definition.
\end{proof}

\begin{proposition}\label{P:restr-iso}
  Let $(\Sigma,Y,\rho)$ and $(\Sigma,Y',\rho')$ be $\Gr$-decorated surface
  where ${Y\subset Y'}$ and $\rho\in\Rep_\Gr(\Sigma,Y)$ is the
  restriction of
  $\rho'\in\Rep_\Gr(\Sigma,Y')$. % Let $\emptyset\neq Y\subset Y'$ be finite subset of $\Sigma$ and
  % $\rho\in\Rep_\Gr(\Sigma,Y)$ be the restriction of
  % $\rho'\in\Rep_\Gr(\Sigma,Y')$.
  Then the inclusion
  $\Sigma\setminus Y'\subset \Sigma\setminus Y$ induces an isomorphism
  $$\Skein(\Sigma,Y',\rho')\to\Skein(\Sigma,Y,\rho).$$
\end{proposition}
\begin{proof}
  First the map
  $\Skein(\Sigma\setminus Y',\rho')\to\Skein(\Sigma\setminus Y,\rho)$
  composed with the projection
  $\Skein(\Sigma\setminus Y,\rho)\to\Skein(\Sigma,Y,\rho)$ (recall
  that the latter module is obtained by quotienting the former by
  transparency relations) factors through $\Skein(\Sigma,Y',\rho)$
  since the additional transparency relations are sent either to
  isotopies of the skein or to transparency relations in
  $\Skein(\Sigma,Y,\rho)$.  Hence, we have a map
  $f:\Skein(\Sigma,Y',\rho')\to\Skein(\Sigma,Y,\rho).$ It is enough to
  show that this map is bijective when $Y'=Y\cup \bs{p_1}$ has a
  unique additional base point.  Let $p_0\in Y$ be in the same
  connected component of $\Sigma$ as $p_1$, let's fix a path $\gamma$
  from $p_0$ to $p_1$ embedded in $\Sigma$ and let
  $g=\rho'(\gamma)\in\Gr$. Remark that $\rho'$ is determined by $\rho$
  and $g$.  To show that $f$ is bijective, we construct an inverse
  $f'$ of $f$ as follows.

  Let $\Graph$ be a skein representing an element of
  $\Skein(\Sigma,Y,\rho)$.  If the point $p_1$ belongs to an edge of
  the skein $\Graph$, we slightly move $\Graph$ near $p_1$ to get a
  skein $\Graph'\in\Skein(\Sigma\setminus Y')$ (there are two
  possibilities up to isotopy in $\Sigma\setminus Y'$).  Let
  $h=[\gamma\cap \Graph']_\Gr$, then $\Graph'$ represents an element of
  $\Skein(\Sigma\setminus Y',((h^{-1}g)^{\delta_{p_1}}).\rho')$.  Then
  we define the image by $f'$ of $\Graph$ by
  $f'(\Graph)=((g^{-1}h)^{\delta_{p_1}})_*(\Graph')$. Remark that when
  $p_1\in \Graph$, the two possibilities for $\Graph'$ are related by
  a gauge map and the two associated values for $f'(\Graph)$ coincide.

  This implies that $f'(\Graph)$ is a well defined element of
  $\Skein(\Sigma\setminus Y',\rho')$.  Also this implies that if
  $\Graph'$ and $\Graph''$ are two skeins in $\Skein(\Sigma,Y,\rho)$
  related by a small isotopy where an edge pass through $p_1$ then
  $f'(\Graph')\skeq f'(\Graph'')$.  Hence,
  $f'(\Graph)\in\Skein(\Sigma,Y',\rho')$ only depends of the isotopy
  class of $\Graph$ in $\Sigma\setminus Y$.

  It is also clear that if $\Graph'$ and $\Graph''$ are two skeins in
  $\Skein(\Sigma,Y,\rho)$ related by a transparency relation, then so
  are $f'(\Graph')$ and $f'(\Graph'')$.

  Finally let us show that the image by $f'$ of a skein relation in
  $\Skein(\Sigma\setminus Y,\rho)$ is zero in
  $\Skein(\Sigma,Y',\rho')$.  This is clear when the box of the skein
  relation does not contain $p_1$ since the construction of the image
  by $f$' can be done outside the box.  We now consider a skein
  relation $\sum_ia_i\Graph_i$ where the box contains the point $p_1$.
  Consider an isotopy in $\Sigma\setminus Y$ moving the box away from
  $p_1$.  This isotopy transform the skeins $\Graph_i$ to skeins
  $\Graph'_i$ with $f'(\Graph_i)\skeq f'(\Graph'_i)$.  Then
  $\sum_ia_if'(\Graph_i)\skeq\sum_ia_if'(\Graph_i')\skeq0\in\Skein(\Sigma,Y',\rho')$
  where this last equivalence follows from the first case since
  $\sum_ia_i\Graph'_i$ is a skein relation away from $p_1$.

  In conclusion we have defined a map
  $f':\Skein(\Sigma,Y,\rho)\to\Skein(\Sigma,Y',\rho')$ which is an
  inverse of $f$, thus $f$ is bijective.
\end{proof}
\begin{definition}\label{D:restr-map}
  Let $R_{\Sigma,Y_1,Y_2,\rho}$ be a restriction cylinder from
  $(\Sigma,Y_1,\rho_1)$ to $(\Sigma,Y_2,\rho_2)$.  Then we define the
  {\em restriction map}
  $$r(\Sigma,Y_1,Y_2,\rho)=f_2\circ
  f_1^{-1}:\Skein(\Sigma,Y_1,\rho_1)\to\Skein(\Sigma,Y_2,\rho_2),$$
  where
  $f_i:\Skein(\Sigma,Y_1\cup Y_2,\rho)\to\Skein(\Sigma,Y_i,\rho_i)$
  is the isomorphism of Proposition \ref{P:restr-iso}. 
\end{definition}
% \BP{Add a notion of generalized restriction-extension isomorphism
%   $\Skein(\Sigma,Y,\rho)\overset\sim\to\Skein(\Sigma,Y',\rho')$ where
%   $(\Sigma,Y,\rho)$ and $(\Sigma,Y',\rho')$ are both restriction of
%   some $(\Sigma,Y'',\rho'')$.  This exists if
%   $(\Sigma,Y\cap Y',\rho_|)=(\Sigma,Y\cap Y',\rho'_|)$ is a $\Gr$-decorated
%   surface this is independant of $(\Sigma,Y'',\rho'')$.}
% \begin{proposition}
%   Let $(\Sigma,Y,\rho)$, $(\Sigma,Y',\rho')$ be $\Gr$-decorated surfaces.
%   Suppose that the restrictions of $\rho$ and $\rho'$ coincide in
%   $\Rep_\Gr(\Sigma,Y\cap Y')$, then there exists $(\Sigma,Y'',\rho'')$
%   with $Y\cup Y'\subset Y''$, such that $\rho$ and $\rho'$ are both
%   restriction of $\rho''$.  Futhermore the induced isomorphism
%   $\Skein(\Sigma,Y,\rho)\overset\sim\to\Skein(\Sigma,Y',\rho')$ is
%   independant of $(\Sigma,Y'',\rho'')$.
% \end{proposition}
% \begin{proof}
%   **********************
% \end{proof}

\begin{theorem}\label{P:SkeinRepCobe}
%Fix a surface $\Sigma$ and a class $\rho\in \Hom(\pi_1(\Sigma),\Gr)/\Gr$. 
  There is an unique
  functor
  $$\Skein:\wt\cob_e \to \operatorname{Vect}$$
  %\BPm{I changed the formulation where image/preimage were ambigous}
  which sends %associates
  a $\Gr$-decorated surface $\wt\Sigma$ (see Subsection \ref{def:Gdecorated}) to
  $\Skein(\wt\Sigma)$, a gauge cylinder $C_{\Sigma,Y,\rho}^\vp$ to $\vp_*$
  and a restriction cylinder $R_{\Sigma,Y_1,Y_2,\rho}$ to the
  restriction map $r(\Sigma,Y_1,Y_2,\rho)$. 
%  \\\BP{I would remove this:\\  In particular, for a fixed surface $\Sigma$ and class $\rho\in \Hom(\pi_1(\Sigma),\Gr)/\Gr$, the vector spaces $\Skein(\wt\Sigma)$ are all canonically isomorphic.} \FCm{added this : do you agree?Also, promoted it to a theorem}\BPm{No I don't think.  Maybe we should stress the opposite: there is no canonical iso}
\end{theorem}
\begin{proof}
  We need to show that the gauge map and restriction map satisfy the
  analogous relations of Lemma \ref{L:rel_eq}.  Relation (1) was
  stated in Proposition \ref{P:gauge-map}. Relation (3) just mean that
  base points can be erased in any order without changing the result.
  Relations (4) and (5) follow directly from Definition
  \ref{D:restr-map}. To prove relation (2) it is enough to consider
  gauge functions $\vp=g^{\delta_y}$ where $y\in Y_1$.  There are two cases to distinguish.

  If $y\notin Y_2$ then the gauge modification at $y$ is followed by
  erasing the base point $y$.  But passing an edge of the skein
  through $y$, then erasing $y$, gives a skein isotopic to the one
  obtained by just erasing $y$.  Hence the relation holds.

  If $y\in Y_2$, then the point $y$ is not erased by the restriction
  maps so passing an edge of the skein through $y$ commutes with the
  suppression of other base points and the relation also holds.
%  \BP{I would remove:\\ The last statement comes from the fact that each equivalence has a canonical expression as a composition of a gauge cylinder and a restriction morphism (see Proposition \ref{P:JR}) which in turn yield explicit isomorphisms on the spaces $\Skein(\widetilde{\Sigma})$ (see Proposition \ref{P:gauge-map} and Definition \ref{D:restr-map}).}\BPm{but the equivalence is not unique because you need to choose a representation in the cylinder}
\end{proof}
\section{$\Gr$-Chromatic categories}
\subsection{Graded chromatic maps and categories}

Let $\cat$ be a sharp $\Gr$-graded pivotal $\kk$-category.   Recall the definition of a generator given in Subsection \ref{SS:KCategories}.  
\begin{definition}
%N: 
%  Let $\cat$ be a $\kk$-linear category.  We say that $W$ is a
%  projective generator of $\cat$ if for any projective module
%  $P\in\cat$, there exists a finite family of morphisms
%  $P\tto {\alpha_i}W$ and $W\tto {\beta_i}P$ such that
%  $\sum_i\beta_i\alpha_i=\Id_P$.
  For $g\in \Gr$, a \emph{projective generator in degree $g$} is a
  projective object which is a generator of the subcategory of
  projective objects of $\cat_g$.%\BPm{added}
\end{definition}
\begin{lemma} 
 %N: removed:  Let $\cat$ be a $\Gr$-graded category.  
 Suppose $W_g$ and $W_{g'}$ are
  projective generators in degrees $g$ and $g'$, respectively.  Then $W_g^*$
  and $W_g\otimes W_{g'}$ are projective generators in degrees $g^{-1}$
  and $gg'$, respectively.
\end{lemma}
\begin{proof}
%N:  If $P^*\in \cat_{g^{-1}}$ then $P\in \cat_g$ and there exist  $f_i:P\to W_g$ and $g_i:W_g\to P$ such that $\Id_P=\sum_i g_i\circ f_i$.  Then $\Id_{P^*}=\sum_{i} f_i^*\circ g_i^*$ implying $W_g^*$ is a  projective generator of degree $g^{-1}$. 
  If $P\in \cat_{g^{-1}}$ is projective then $P^*\in \cat_g$ is projective and there exist  
  $$f_i:P^*\to W_g \text{ and } g_i:W_g\to P^*$$ such that $\Id_{P^*}=\sum_i g_i\circ f_i$.  Then $\Id_{P^*}^*=\sum_{i} f_i^*\circ g_i^*$ implying $\Id_P=\sum_{i} \phi^{-1}\circ f_i^*\circ g_i^* \circ\phi $ where $\phi:P\to P^{**}$ is the pivotal structure.  Thus, $W_g^*$ is a  projective generator of degree $g^{-1}$. 

 %N:   If $Q\in \cat_{gg'}$ then $W_{g^{-1}}\otimes Q\in \cat_{g'}$ so there exist $f_i:W_{g^{-1}}\otimes Q\to W_{g'}, g_i:W_{g'}\to W_{g^{-1}}\otimes Q$ such that $\Id_{W_{g^{-1}}\otimes Q}=\sum_i g_i\circ f_i$. 

  If $Q\in \cat_{gg'}$ is projective then
  $W_{g}^*\otimes Q\in \cat_{g'}$  is projective so there exist
  $f_i:W_{g}^*\otimes Q\to W_{g'}$ and $ g_i:W_{g'}\to W_{g}^*\otimes
  Q$ such that $\Id_{W_{g}^*\otimes Q}=\sum_i g_i\circ f_i$. 
   Then let
  $$f''_i=(\Id_{W_{g}}\otimes f_i)\circ (\lcoev_{W_g}\otimes \Id_Q)
  {\rm\ and\ } g''_i= (\rev_{W_g}\otimes \Id_Q)\circ(\Id_{W_g}\otimes
  e_{W_g,Q})\circ (\Id_{W_g}\otimes g_i)$$ where 
  %N: $e'_{W_g,Q}$ 
  $e_{W_g,Q}$
  is as in
  Lemma \ref{L:V-appear}. Then
%N  $$g''_i\circ f''_i=(\rev_{W_g}\otimes \Id_Q)\circ(\Id_{W_g}\otimes  e_{W_g,Q})\circ (\lcoev_{W_g}\otimes  \Id_Q)=\ptr^L_{W_g}(e'_{W_g,Q})=\Id_Q.$$
$$\sum_i g''_i\circ f''_i=(\rev_{W_g}\otimes \Id_Q)\circ(\Id_{W_g}\otimes  e_{W_g,Q})\circ (\lcoev_{W_g}\otimes  \Id_Q)=\ptr^L_{W_g^*}(e_{W_g,Q})=\Id_Q.$$
  Thus, $W_g\otimes W_{g'}$ is a projective generator in degree $gg'$.  
\end{proof}

\begin{definition}
%\NGm{In the book we say "A category C is called spherical if it is a unimodular sharp K-additive pivotal category such that there exists a non-degenerate (two-sided) m- trace on Proj. " What we have here is different.  No unimodular, no non-deg ...}
 %N: A \emph{spherical category} is a sharp pivotal $\kk$-category such that there exist a non-degenerate  m-trace on $\Proj$.   
  A \emph{$\Gr$-finite spherical category} is a $\Gr$-graded spherical category
  $\cat=\bigoplus_{g\in\Gr}\cat_g$ where for
  each $g\in\Gr$, $\cat_g$ has a  non-null projective generator. % a family
  % $\bs{\P_g}_{g_in\Gr}$ where for each $g\in\Gr$, $\P_g$ is a
  % projective generator of $\cat_g$.
  A \emph{finite spherical category} is a $\Gr$-finite spherical category where the grading group is trivial. 
%  \BPm{could we remove this definition and define a finite/$\Gr$-finite monoidal $\kk$-category in section 1.2/1.3 ? I added spherical in 1.4.} 
\end{definition}
%Here the word finite is used to because if the category is abelian and locally finite then it is finite if and only if the category has finitely many simple objects.  
Let $\cat$ be a $\Gr$-finite spherical category. Note, that $\cat_1$
is a finite spherical category.  Let $\mt$ be a non-zero trace on
$\Proj$.  Then $\mt$ is \emph{non-degenerate} (see Definition
\ref{D:spherical}).  % which means that for any $P\in\Proj$, the pairing
% $\Hom_{\cat}(\unit,P)\otimes_\kk\Hom_{\cat}(P,\unit)\to\kk$ defined by
% $x\otimes y\mapsto \mt(x\circ y)$ is non-degenerate.
Given % such a non-degenerate m-trace $\mt$ on
% $\Proj$ and
an object $P\in \Proj$, set
\begin{equation}\label{E:copairing}
\Omega_p=\sum_i x^i\otimes_\kk x_i\in \Hom_{\cat}(P,\unit)\otimes_\kk
\Hom_{\cat}(\unit,P) \text{ and } \Lambda_P=\sum_{i} x_i\circ x^i\in \Hom_\cat(P,P)
\end{equation}
where $\{x^i\}_i$ and $\{x_i\}_i$ are basis of $ \Hom_{\cat}(P,\unit)$ and $ \Hom_{\cat}(\unit,P)\}$, respectively, which are dual with respect to the m-trace $\mt$, \ie $\mt_P(x_i \circ x^j) = \delta_{i,j}$. 
%\NGm{I rewrote this.  I think there was a mistake it said "$\Omega_p=\sum_i x_i\otimes_\kk x^i\in \Hom_{\cat}(P,\unit)\otimes_\kk \Hom_{\cat}(\unit,P)$" in the last version.  I changed the $x_i$ to $x^i$.  Bertrand, if ok please remove this comment.}  
 It is a
straightforward exercise in linear algebra to verify that $\Omega_P$
and $\Lambda_P$ do not depend on the choice of the basis.

%\NGm{in the below lemma I changed $\ideal $ to $\Proj$ and $X, Y$ to $P, Q, V$. If ok please remove this comment.}
%\begin{lemma} \label{P:Omega-nat}
%Let $X,Y\in \ideal$ and $Z \in \cat$, and let $f: X\to Y$ be a morphism in $\cat$. 
%  \begin{enumerate}
%%\labela
%  \item \emph{Duality}: If $\Omega_X=\sum_ix^i\otimes x_i$, then
%    $$\Omega_{X^*}=\sum_i  (x_i)^*  \otimes
%    (x^i)^*\in\Hom_\cat(X^*,\unit)\otimes_{\kk} \Hom_\cat(\unit,X^*).$$
%  \item \emph{Naturality}: If $\Omega_X=\sum_ix^i\otimes x_i$ and    $\Omega_Y=  \sum_j y^j\otimes y_j $,  then
%    $$\sum_ix^i\otimes (f\circ x_i)= \sum_j(y^j\circ f)\otimes
%    y_j  \in\Hom_\cat(X,\unit)\otimes_{\kk}\Hom_\cat(\unit,Y).$$
%  \item \emph{Rotation}: If $\Omega_{X\otimes Z}=\sum_iz^i\otimes z_i$
%    then $\Omega_{Z\otimes X}=\sum_i\wt z^i\otimes \wt z_i$ where
%    $$ \wt z^i=\rev_Z(\Id_Z\otimes
%    z^i\otimes\Id_{Z^*})(\Id_{Z \otimes X}\otimes\lcoev_Z) \quad \text{and} \quad \wt z_i=(\Id_{Z \otimes X}\otimes\rev_Z)(\Id_Z \otimes
%    z_i\otimes\Id_{Z^*})\lcoev_Z.$$
%  \end{enumerate}
%\end{lemma}
%\begin{proof}
%  The duality and rotation properties follow  from the fact that  we apply
%  transformations  sending  dual basis to dual basis. The naturality
%  can be checked by applying
%  $\mt_X(x_k\circ\_)\otimes\mt_Y(\_\circ y^\ell)$ to both sides:  it
% reduces then  to the cyclic property
%  $\mt_Y(f\circ x_k\circ y^\ell)=\mt_X(x_k\circ y^\ell\circ f)$ of the m-trace $\mt$.
%\end{proof}
\begin{lemma} \label{P:Omega-nat}
Let $P,Q\in \Proj$ and $V \in \cat$, and let $f: P\to Q$ be a morphism in $\cat$. 
  \begin{enumerate}
%\labela
  \item \emph{Duality}: If $\Omega_P=\sum_ix^i\otimes x_i$, then
    $$\Omega_{P^*}=\sum_i  (x_i)^*  \otimes
    (x^i)^*\in\Hom_\cat(P^*,\unit)\otimes_{\kk} \Hom_\cat(\unit,P^*).$$
  \item \emph{Naturality}: If $\Omega_P=\sum_ix^i\otimes x_i$ and    $\Omega_Q=  \sum_j y^j\otimes y_j $,  then
    $$\sum_ix^i\otimes (f\circ x_i)= \sum_j(y^j\circ f)\otimes
    y_j  \in\Hom_\cat(P,\unit)\otimes_{\kk}\Hom_\cat(\unit,Q).$$
  \item \emph{Rotation}: If $\Omega_{P\otimes V}=\sum_iz^i\otimes z_i$
    then $\Omega_{V\otimes P}=\sum_i\wt z^i\otimes \wt z_i$ 
      \end{enumerate}
where    $$ \wt z^i=\rev_V(\Id_V\otimes
    z^i\otimes\Id_{V^*})(\Id_{V \otimes P}\otimes\lcoev_V) $$
    and  $$
    \wt z_i=(\Id_{V \otimes P}\otimes\rev_V)(\Id_V \otimes
    z_i\otimes\Id_{V^*})\lcoev_V.$$

\end{lemma}
\begin{proof}
  The duality and rotation properties follow  from the fact that  we apply
  transformations  sending  dual basis to dual basis. The naturality
  can be checked by applying
  $\mt_P(x_k\circ\_)\otimes\mt_Q(\_\circ y^\ell)$ to both sides:  it
 reduces then  to the cyclic property
  $\mt_Q(f\circ x_k\circ y^\ell)=\mt_P(x_k\circ y^\ell\circ f)$ of the m-trace $\mt$.
\end{proof}

Recall a chromatic map in $\cat_1$ for a projective generator $W$ is a
map $\chr_W\in\End_\cat(W\otimes W)$ such that the following equality
holds for all $V\in \cat$:

  \begin{equation}
    \label{eq:chrP1}
    \epsh{fig34a.pdf}{10ex}
\putc{24}{52}{$\ms{\Lambda_{V\otimes W^*}}$}
\putc{81}{50}{${\chr_W}$}
\pute{93}{15}{$\ms{W}$}
\pute{93}{81}{$\ms{W}$}
\putw{7}{81}{$\ms{V}$}
\putw{7}{14}{$\ms{V}$}
\puts{53}{8}{$\ms{W}$}
\putn{56}{93}{$\ms{W}$}
    =\epsh{fig34c}{10ex} \putw{0}{90}{$\ms{V}$}\pute{100}{90}{$\ms{W}$}\;
    \quad\text{ that is }\quad
    \sum_i
    \epsh{fig34b.pdf}{10ex}
\putc{15}{70}{$\ms{x_i}$}
\putc{15}{29}{$\ms{x^i}$}
\putc{80}{50}{${\chr_W}$}
\pute{93}{82}{$\ms{W}$}
\pute{93}{15}{$\ms{W}$}
\putw{6}{93}{$\ms{V}$}
\putw{6}{6}{$\ms{V}$}
\puts{42}{2}{$\ms{W}$}
\putn{44}{98}{$\ms{W}$}
    =\epsh{fig34c}{10ex}
    \putw{0}{90}{$\ms{V}$}\pute{100}{90}{$\ms{W}$} \;\;\ .
  \end{equation}

\begin{definition}\label{D:chr}
  A \emph{degree $g$ chromatic map based on $P \in \Proj$} for a
  projective generator $W_g$ of $\cat_g$ is a morphism
  $\chr_P\in\End_\cat(W_g\otimes P)$ satisfying
  % \begin{equation}\label{E:ChrMapDef}
  %   \epsh{tfig19}{10ex}\put(-17,0){$\chr$}
  %   \put(-69,0){\ms{\Lambda_{\otimes G^*}}}=\epsh{tfig20}{10ex} \;.
  % \end{equation}
  % More generally, for all $V\in\cat$, we
  % have
  \begin{equation}
    \label{eq:chrP}
    \epsh{fig34a.pdf}{10ex}
\putc{24}{52}{$\ms{\Lambda_{V\otimes W_g^*}}$}
\putc{81}{50}{${\chr_P}$}
\pute{93}{15}{$\ms{P}$}
\pute{93}{81}{$\ms{P}$}
\putw{7}{81}{$\ms{V}$}
\putw{7}{14}{$\ms{V}$}
\puts{53}{8}{$\ms{W_g}$}
\putn{56}{93}{$\ms{W_g}$}
    =\epsh{fig34c}{10ex} \putw{0}{90}{$\ms{V}$}\pute{100}{90}{$\ms{P}$}
    \quad\text{ that is }\quad
    \sum_i
    \epsh{fig34b.pdf}{10ex}
\putc{15}{70}{$\ms{x_i}$}
\putc{15}{29}{$\ms{x^i}$}
\putc{80}{50}{${\chr_P}$}
\pute{93}{82}{$\ms{P}$}
\pute{93}{15}{$\ms{P}$}
\putw{6}{93}{$\ms{V}$}
\putw{6}{6}{$\ms{V}$}
\puts{42}{2}{$\ms{W_g}$}
\putn{44}{98}{$\ms{W_g}$}
    =\epsh{fig34c}{10ex}
    \putw{0}{90}{$\ms{V}$}\pute{100}{90}{$\ms{P}$}
  \end{equation}
for any $V\in\cat_g$, where $\{x_i\}_i$ and $\{x^i\}_i$ are any dual bases with respect to the m-trace.
\end{definition}
\begin{lemma}\label{L:chr-based}
  If $\cat$ has a degree $g\in\Gr$ chromatic map $\chr_P$ based on
  $P \in \Proj$ for a projective generator $W_g$, then for any
  $Q\in\Proj$ and any projective generator $W'_g$ of $\cat_g$, there
  exists a chromatic map $\chr_Q$ based on $Q$ for the
  projective generator $W'_g$.
\end{lemma}
\begin{proof}
  Let $W_g\tto {\alpha_i}W'_g$ and $W'_g\tto {\beta_i}W_g$ be a family
  of morphisms such that $\sum_i\beta_i\alpha_i=\Id_{W_g}$.  Then
  $$\chr_Q=\sum_i(\alpha_i\otimes\rev_P\otimes\Id_Q)\circ(\chr_{P}\otimes e_{P,Q})
  \circ(\beta_i\otimes\lcoev_P\otimes\Id_Q)$$
  is the desired chromatic map, where $e_{P,Q}$ is as in Lemma \ref{L:V-appear}
\end{proof}

\begin{definition}\label{def:newdef}
  A \emph{$\Gr$-chromatic category} is a $\Gr$-finite spherical category
  equipped with a non-degenerate modified trace $\mt$ on $\Proj$ for which there exists a chromatic map in each degree
  $g\in\Gr$.
\end{definition}

The proof of the following is straightforward from the definitions of
$\chr_{W_h\otimes P}$ and of $e'_{W_h,P}$:
\begin{lemma}\label{L:chr-degree}
  Let $\cat$ be a $\Gr$-finite spherical category with projective generators
  $W_g$ and $W_{h}$ in degrees $g$ and $h$, respectively.  If $\chr_{W_h\otimes P}$
  is a chromatic map for $W_g$ based on ${W_h\otimes P}$, then
  $\chr_P=(\Id_{W_g}\otimes e'_{W_h,P})\circ\chr_{W_h\otimes P}$ is a
  chromatic map for $W_g\otimes W_h$ based on $P$, where $e'_{W_h,P}$
  is as in Lemma \ref{L:V-appear}. %\BHm{I think the content of this Lemma should be the proof of the next corollary.}
\end{lemma}
% \begin{corollary}
%   If $\cat$ is a $\Gr$-spherical category with modified trace $\mt$, then there exists a
%   chromatic map in one degree $g\in\Gr$ if and only if there exists a
%   chromatic map in each degree.
% \end{corollary}
\begin{proposition}\label{chrom-cat1}
  If $\cat$ is a $\Gr$-finite spherical category with modified trace
  $\mt$, then $\cat$ is a $\Gr$-graded chromatic category if and only
  if $\cat_1$ is a chromatic category.
\end{proposition}
%\FCm{I would add a proof of the corollary}\NGm{I added it, if ok please remove these comments.}
\begin{proof}
  By definition, if $\cat$ is a $\Gr$-graded chromatic category then
  $\cat_1$ is a chromatic category.  Conversely, assume $\cat_1$ is a
  chromatic category.  Let $P\in \Proj$ and let $W_g$ be a projective
  generator of degree $g\in \Gr$.  Then $W_g\otimes W_g^*$ is a
  projective generator of degree $1$ so by Lemma \ref{L:chr-based} there exists a chromatic map
  for $W_g\otimes W_g^*$ based on $W_g\otimes P$.  Then Lemma \ref{L:chr-degree} 
  implies that there exists a chromatic map based on $P$ of degree $g$
  for the generator $W_g\otimes W_g^*\otimes W_g$.  
\end{proof}
\begin{remark}
  Let $\cat$ be $\Gr$-finite spherical category where $\cat_1$ is a finite
  tensor category in the sense of \cite{EGNO15}. Then from \cite{CGPV23}
  % \cite{CGPV23a,CGPV23b}
  it has a chromatic map.  Then the corollary implies that $\cat$ is a
  $\Gr$-graded chromatic category.
\end{remark}
  % more generally
% \begin{proposition}
%   If $\cat$ is a $\Gr$-spherical category with modified trace $\mt$,
%   then $\cat$ is a $\Gr$-graded chromatic category if and only if
%   there exists a chromatic map in some degree $g\in\Gr$.
% \end{proposition}
A different flavour of ``$\Gr$-finite spherical'' based on generically
semi-simple categories was used in \cite{GPT09,GP13} to construct
renormalized Turaev-Viro invariants:
\begin{definition}\label{def:olddef}
  A \emph{relative $\Gr$-finite spherical category} $\cat$, relative to $\XX$ is a
  sharp $\Gr$-graded category which is finitely semi-simple in each
  degree outside a small subset $\XX\subset\Gr$:
  \begin{enumerate}
  \item $\XX\subset\Gr$, $\XX^{-1}=\XX$;
  \item for any $g_1,\ldots g_n\in\Gr$, $\Gr\not\subset\bigcup_ig_i\XX$;
  \item for any $g\in\Gr\setminus\XX$, the category $\cat_g$ is
    finitely semi-simple.  Simple objects of $\cat_g$ with $g\notin\XX$
    are called generic simple.
  \item There exists a non zero modified trace on $\Proj$.
  \item There exists a $\kk$-valued function $\mb$ on generic simple
    objects, which is constant on isomorphism classes and satisfies
    for any $\bs{g_1,g_2,g_1g_2}\subset\Gr\setminus\XX$ and any simple
    object $V\in\cat_{g_1g_2}$,
    $$\mb(V^*)=\mb(V)=\sum_{V_1\in S_{g_1},V_2\in S_{g_2}}\mb(V_1)\mb(V_2)
    \dim_\kk(\Hom(V,V_1\otimes V_2)),$$ where $S_g$ is a set of
    representative of isomorphic classes of simple objects of
    $\cat_g$.
  \end{enumerate}
\end{definition}
For example, the category of weight modules of an unrestricted quantum
group at odd root of unity is a relative $\Gr$-finite spherical category (see Subsection \ref{SS:exemple}).
% (where $\Gr$ is a Lie group which is Poisson-Lie dual to the Lie group
% associated to the quantum group), see \cite{GP13}.
\begin{proposition}\label{rel-spherical-is-chromatic}
  % A relative $\Gr$-spherical category (generically semi-simple) is a
  % $\Gr$-spherical category (chromatic maps).
  A relative $\Gr$-finite spherical category in the sense of Definition \ref{def:olddef} is a
  $\Gr$-finite spherical category in the sense of Definition \ref{def:newdef}.
\end{proposition}
\begin{proof}
  % In generic degree $g$, $W=\bigoplus_iS_i$ is a projective generator
  % (where the $S_i$ are representing the isomorphic classes of simple
  % object of $\cat_g$) and $\chr_P=\oplus_i\qd(S_i)\Id_{S_i}\otimes P$ is a chromatic
  % map for $W$ based on $P$.
  Observe that for generic $g\in \Gr\setminus\XX$, $\cat_g$ is
  semisimple so that $\Proj\cap \cat_g=\cat_g$. So, in generic degree
  $g$, $W=\bigoplus_iS_i$ is a projective generator (where the $S_i$
  are representing the isomorphism classes of simple object of
  $\cat_g$).  For $\qd(S_i)=\mt_{S_i}(\Id_{S_i})$, it is easy to check
  directly that $\chr_P=\oplus_i\qd(S_i)\Id_{S_i}\otimes \Id_P$ is a
  chromatic map for $W$ based on $P$.
\end{proof}

\subsection{Example of $\Gr$-chromatic categories}\label{SS:exemple}
\newcommand{\rk}{n}
\newcommand{\e}{\operatorname{e}}
\newcommand{\roots}{\Delta}
\newcommand{\La}{L_W}
\newcommand{\h}{\mathfrak{h}}
\newcommand{\Uqg}{U_\xi\g}
\newcommand{\ro}{\ell}
Let $\ell$ be an odd integer such that $\ell\geq 2$ (and $\ell\notin 3\Z$ if $\g=G_2$).  Let $\g$ be a simple finite-dimensional complex Lie algebra of rank
$\rk$ and  dimension $2N+\rk$ with the following:
\begin{enumerate}
\item  a Cartan subalgebra $\h$,
\item a root
system consisting in simple roots
$\{\alpha_1,\ldots,\alpha_\rk\}\subset\h^*$,
\item  a Cartan
matrix $A=(a_{ij})_{1\leq i,j\leq \rk}$,
\item a set $\roots^{+}$ of $N$
positive roots,
\item a root lattice $L_R=\bigoplus_i\Z\alpha_i\subset\h^*$,
\item a scalar product $\left\langle{\cdot,\cdot}\right\rangle$ on the real span of $L_R$ given by its matrix $DA=(\ba{\alpha_i,\alpha_j})_{ij}$ where $D=\operatorname{diag}(d_1,\ldots,d_\rk)$ and the minimum of all the $d_i$ is 1.
\end{enumerate}
The Cartan subalgebra has a basis $\{H_i\}_{i=1\cdots\rk}$ determined
by $\alpha_i(H_j)=a_{ji}$ and its dual basis of $\h^*$ is the
fundamental weights basis of $\h^*$ which generates the lattice of
weights $\La$.  Let
$\rho=\frac12\sum_{\alpha\in\roots^{+}}\alpha\in \La$.
 
% For $i=1,\ldots,\rk$, let
% $\xi_i=\xi^{d_i}$.
  For $x\in \C$ and $k,l\in \N$ we use the notation:
$$
\xi^x=\e^{\frac{2i\pi x}\ell},\quad \qn x_\xi=\xi^x-\xi^{-x},\quad % \qN x_\xi=\frac{\qn
  % x_\xi}{\qn 1_\xi},\quad
\qn k_\xi!=\qn1_\xi\qn2_\xi\cdots\qn k_\xi,\quad{k\brack
  l}_\xi=\frac{\qn{k}_\xi!}{\qn{l}_\xi!\qn{k-l}_\xi!}.
$$    

For each lattice $L$ with $L_R\subset L\subset \La$, there is an
associated \emph{unrestricted quantum group}  which contains the group ring of $L$: 
define $U_\xi^L\g$ as the $\C$-algebra with generators
$K_\beta, \, X_i,\, X_{-i}$ for $\beta\in L, \,i=1,\ldots,\rk$ and
relations
\begin{eqnarray}\label{eq:rel1}
 & K_0=1, \quad K_\beta K_\gamma=K_{\beta+\gamma},\, \quad K_\beta X_{\sigma
    i}K_{-\beta}=\xi^{\sigma \ba{\beta,\alpha_i}}X_{\sigma i}, & \\
 \label{eq:rel2} &[X_i,X_{-j}]=
 \delta_{ij}\frac{K_{\alpha_i}-K_{\alpha_i}^{-1}}{\xi^{d_i}-\xi^{-d_i}}, &  \\
 \label{eq:rel3} & \sum_{k=0}^{1-a_{ij}}(-1)^{k}{{1-a_{ij}} \brack
   k}_{\xi^{d_i}} X_{\sigma i}^{k} X_{\sigma j} X_{\sigma i}^{1-a_{ij}-k} =0,
 \text{ if }i\neq j &
\end{eqnarray}
where $\sigma=\pm 1$.   
Drinfeld and Jimbo consider the quantum group corresponding to
$L=L_R$.  
The algebra $U_\xi^L\g$ is a Hopf algebra with coproduct
$\Delta$, counit $\epsilon$ and antipode $S$ defined by
\begin{align*}
  \Delta(X_i)&= 1\otimes X_i + X_i\otimes K_{\alpha_i}, &
  \Delta(X_{-i})&=K_{\alpha_i}^{-1} \otimes
  X_{-i} + X_{-i}\otimes 1,\\
  \Delta(K_\beta)&=K_\beta\otimes K_\beta, & \epsilon(X_i)&=
  \epsilon(X_{-i})=0, &
  \epsilon(K_{\alpha_i})&=1,\\
  S(X_i)&=-X_iK_{\alpha_i}^{-1}, & S(X_{-i})&=-K_{\alpha_i}X_{-i}, &
  S(K_\beta)&=K_{-\beta}.
\end{align*}
For a fixed choice of a convex order
$\beta_*=(\beta_1,\ldots,\beta_N)$ of $\roots^{+}$ define recursively
a convex set of root vectors $(X_{\pm\beta})_{\beta\in\beta_*}$ in
$U_\xi^L\g$ (see Section 8.1 and 9.1 of \cite{CP95}).  The algebra
$Z_0$ is the subalgebra generated by the elements
$$X_{\pm\beta}^\ro\text{ and } K_\gamma^{\ro}\text{ for all }\beta
\in\Delta^+,\gamma\in L.$$ It is a central sub-Hopf algebra of
$U_\xi^L\g$ and $U_\xi^L\g$ is free $Z_0$-module of rank
${\ro^{2N+n}[L\!:\! L_R]}$.  A $U_\xi^L\g$-module is a weight module
if it is finite dimensional semi-simple as a $Z_0$-module. Then let
$\cat$ be the category of weight modules.

The element $\phi=K_{2\rho}^{1-\ell}$ in $U_\xi^L\g$ satisfies
$S^2(x)=\phi x\phi^{-1}$.  It follows that $U_\xi^L\g$ is a pivotal
Hopf algebra and the category $\cat$ of weight modules is a pivotal
$\C$-category, which is spherical, for details see \cite{GP13}.  Here,
the duality morphisms $\lev$ and $\lcoev$ are the standard evaluation
and coevaluation whereas $\rev\,=\lev\circ\tau\circ(\phi\otimes1)$ and
$\rcoev=\tau\circ(\phi^{-1}\otimes1)\circ\lcoev$ where $\tau$ is the
flip.
 
Let $\Gr=\Hom_{\operatorname{algebra}}(Z_0,\C)$, then $\cat$ is a
sharp $\Gr$-graded category where $\cat_g$ is the full subcategory of
modules $V$ such that any $z\in Z_0$ acts on $V$ by the scalar $g(z)$.
The coproduct of $Z_0$ induces a group structure on $\Gr$ which is
known to be isomorphic to the Poisson-Lie dual of a connected Lie
group with Lie algebra $\g$; Let $\Gr'$ be a connected complex Lie group
with Lie algebra $\g$ and $T$ be the image of $\h$ by the exponential
map.  Suppose $L$ is the lattice dual to the kernel of the exponential
map $h\mapsto\e^{2i\pi h}$, and let $B^+$, $B^-$ be the positive and
negative Borel of $\Gr'$ and $\pi:B^+\times B^-\to T$ be the product of
the two projections $B^\pm\to T$.  Then $\Gr$ is isomorphic to the
kernel of $\pi$.
\begin{proposition}
  The category $\cat$ of weight $U_\xi^L\g$-modules is a
  $\Gr$-chromatic category.
\end{proposition}
\begin{proof}
  By \cite[Theorem 5]{GP18}, the category is $\Gr$-spherical, then by
  \cite{CGPV23}, $\cat_1$ is a chromatic tensor category so by
  Proposition \ref{chrom-cat1}, it is a $\Gr$-chromatic category.
\end{proof}
Remark that the unrestricted quantum group $U_\xi^L\g$ has more
structure: the quantum coadjoint action (see \cite{DKP92}) and a
holonomic braiding (see \cite{KR05,KR05a,BGPR20,McP20,McP21,MR25})
that we do not exploit in our construction.

\subsection{Admissible skein module of $\Gr$-chromatic categories}%a $\Gr$-decorated surface

%Let $\Sigma$ be a closed oriented surface with a finite set of basepoints $Y\subseteq \Sigma$, at least one per connected component, and a representation $\rho\in \Rep_\Gr(\Sigma,Y)$ of its fundamental groupoid. To simplify notations, we will call the triple $\wt\Sigma=(\Sigma, Y,\rho)$ a \emph{$\Gr$-decorated surface}.\BPm{I think Ben you already defined this in Definition \ref{D:sk-dec-surface}...}
 Let $\cat$ be a $\Gr$-chromatic category and $\ideal = \Proj$.%\BHm{I think in the section above you switched to $\ideal = \Proj$ right?}

% \BP{Introduction of $\Gr$-colored oriented red edges which satisfy the graded
%   slidding relation and the inversion.}

% \BP{**********  I think François, it would be easier to say the red circles are just a notation for a real blue skein. In any case the relations (3red), (4) and (5) below are just consequences of the red-to-blue relation (we say a ``chromatic move'' in the book).}

We will now introduce a notation, the ``$\Gr$-colored red circles'' which will later represent surgeries being operated on surfaces. 
As we will see immediately, they are actually place holders: each graph containing some $\Gr$-colored red circle is represented by a well defined skein equivalence class of admissible skeins in the previous sense. So a priori red circles could be avoided in our account but they are very handy for proofs. 
\begin{definition}
  A bichrome graph in a connected surface $\Sigma$ is the disjoint
  union of an admissible $\cat$-colored graph $\Graph\subset\Sigma$
  and a finite set of disjoint oriented circles (called ``red
  circles'') each decorated by some element $g\in \Gr$ and disjoint from
  $\Graph$ and from the basepoints of $\Sigma$. A bichrome graph in a
  disjoint union of surfaces $\sqcup \Sigma_i$ is the datum of a
  bichrome graph in each $\Sigma_i$. We will call the union of the red
  circles the ``red part'' of $\Graph$ and its complement in $\Graph$
  the ``blue part'' and we will say that $\Graph$ is monochrome if the
  red part is empty.
\end{definition}
For each connected component of $\Sigma$, by definition at least one red circle in a bichrome graph $\Graph$ is adjacent to at least one region of the component containing some strand of $\Graph$ colored by a projective object. By recurrence then, one can apply the following:
\begin{definition}[Admissible skein class of bichrome graphs]\label{def:bichrome}
Let $\Graph\subset \Sigma$ be a bichrome graph. Its skein equivalence class, denoted $[\Graph]\in \Skein(\Sigma)$ is the admissible skein equivalence class of any monochrome graph obtained from $\Graph$ by applying any finite sequence of the following modifications:
\begin{itemize}
\item one of the two ``red-to-blue'' modifications depicted in Figure
  \ref{fig:chrom-mod} applied to a red circle
\item an admissible skein relation on the blue subgraph of $\Graph$
  operated in a box not intersecting the red part.
\end{itemize}
\end{definition}
The previous definition is independent on all the choices made by the following:
\begin{proposition}\label{prop:bichrome_welldef}
The skein equivalence class $[\Graph]\in \Skein(\Sigma)$ of the bichrome graph $\Graph$ is well defined. %In particular it does not depend on : the choice of which strands are used to perform the ``red-to-blue'' modifications, what isotopies are used to bring these strands near the red circles, what chromatic maps are used and on what sides of the red circles are the ``red-to blue modification'' applied and what list of admissible skein relations is applied to the blue part of $T$. 
\end{proposition}
\begin{proof}
It is sufficient to prove the statement for connected $\Sigma$. 
For our proof we will argue by induction on the number $n$ of red circles. 
For $n=0$ there is nothing to prove. Suppose the statement true for bichrome graphs containing up to $n$ red circles and let $\Graph$ be one containing $n+1$ red-circles. 

To reduce the number of red-circles one needs to apply a red-to-blue modification to $\Graph$; so one needs to connect an edge of $\Graph$ colored by a projective color, call it $e$, to a red-circle, call it $c$, via a path $\alpha$ embedded in $\Sigma\setminus R$ (where $R$ is the red part) having one endpoint on $e$ and the other one on $c$. In general, though, $\alpha$ is transverse to $\Graph$ and intersects in its interior finitely many of its edges besides $c$ and $e$. The \emph{generalised red-to-blue} move along $\alpha$ is the composition of the admissible skein relation consisting in fusing all the blue strands intersecting $\alpha$ (including $e$, so that the resulting color is projective because $\Proj$ is an ideal) and then applying the red-to-blue modification to the so-obtained graph using $\alpha$ to approach the new projective color to $c$. We will denote $\Graph_\alpha$ the resulting bichrome graph (the slightly abusive notation suppresses the choice of the generator and the chromatic morphism used in the modification). 

It is clear that $\Graph_\alpha$ depends only on $\alpha$ up to
isotopies which do not change its intersections with $\Graph$.
Therefore, if $\alpha$ and $\alpha'$ are two arcs as above connecting
edges $e$ (resp. $e'$) to red circles $c$ (resp. $c'$) we can suppose,
up to perturbing them via a small isotopy, that they intersect
transversally and in a finite number, call it $m$, of points.  We
claim that the admissible skein equivalence class of $\Graph_\alpha$
and of $\Graph_\alpha'$ are equal (these classes are well defined by
inductive hypothesis because both $\Graph_\alpha$ and
$\Graph_{\alpha'}$ have $m$ red circles).

To prove our claim we argue by induction on $m$.  If $m=0$ $\alpha$
and $\alpha'$ are disjoint and we distinguish two cases. If $c\neq c'$
then $\Graph_\alpha$ is skein equivalent to
$(\Graph_\alpha)_{\alpha'}$ simply by applying the generalised
red-to-blue modification along $\alpha'$ to $\Graph_\alpha$. But
$(\Graph_{\alpha})_{\alpha'}=(\Graph_{\alpha'})_{\alpha}$ (because
$\alpha$ and $\alpha'$ are disjoint) and the latter is skein
equivalent to % equal to
$(\Graph_{\alpha'})$ by application of the
generalised red-to-blue modification along $\alpha$.

If $c=c'$, let $P$ (resp. $Q$) be the color obtained after operating the fusion of all the blue edges intersecting $\alpha$ (resp. $\alpha'$),
let $W_g,W'_g$ be projective generators of $\cat_g$ and pick a chromatic map $\chr_P$ based on $P$ at $W_g$ and a chromatic map  $\chr_Q$ based on $Q$ at $W'_g$. There are two subcases to consider. First, if the two
  modifications are made on the same side of the red curve, then
$$
  \epsh{fig4a_2}{24ex}\put(-19,-32){\ms{{{\chr}}_P}}\put(-21,45){\ms{Q}}
    =\sum_i
    \epsh{fig4b_2}{24ex}\put(-19,-32){\ms{{{\chr}}_P}}\put(-19,13){\ms{{{\chr}}_Q}}
    \put(-38,36){\ms{x_i}}\put(-38,-11){\ms{x^i}} =\sum_i
    \epsh{fig4c}{24ex}\put(-19,-9){\ms{{{\chr}}_P}}\put(-19,36){\ms{{{\chr}}_Q}}
    \put(-40,12){\ms{{x^*}_i}}\put(-40,-34){\ms{{x^*}^i}} =
    \epsh{fig4d}{24ex}\put(-19,36){\ms{{{\chr}}_Q}}\put(0,-44){\ms{P}}
$$
where ${x^*}_i$ and ${x^*}^i$ are the dual basis obtained by   ${x^*}^i=(x_i)^*\circ(\phi_{W'_g}\otimes\Id_{W_g^*})$ and ${x^*}_i=(\phi_{W'_g}^{-1}\otimes\Id_{W_g^*})\circ(x^i)^*$ (recall that $\phi_X:X\to X^{**}$ is the pivotal structure).   Here the first and third equalities follow from  \eqref{eq:chrP}  and the second equality from isotopying  the coupon and applying duality of Lemma~\ref{P:Omega-nat}.
Second, if the modifications are made on opposite sides of the red curve, then (with implicit summation):
$$
    \epsh{fig5a}{24ex}\put(-12,0){\ms{{{\chr}}_P}}\put(-38,-2){\ms{Q}}
    \quad=\quad
    \epsh{fig5b}{24ex}\put(-12,-4){\ms{{{\chr}}_P}}\put(-48,18){\ms{{{\chr}}_Q}}
    \put(-65,-5){\ms{x^i}}\put(-65,38){\ms{x_i}}
    \quad=
    \epsh{fig5c}{24ex}\put(-57,14){\ms{{{\chr}}_Q}}\put(-14,22){\ms{{{\chr}}_P}}
    \put(-75,34){\ms{x_i}}\put(-19,-11){\ms{x^i}}
    \quad=
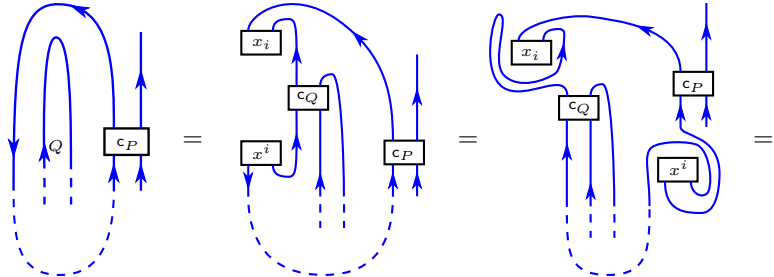
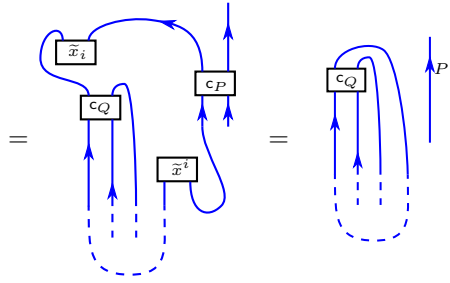
$$
    $$=
    \epsh{fig5d}{24ex}\put(-55,14){\ms{{{\chr}}_Q}}\put(-11,22){\ms{{{\chr}}_P}}
    \put(-63,34){\ms{\wt x_i}}\put(-24,-11){\ms{\wt x^i}}
    \quad=\quad
    \epsh{fig5e}{18ex}\put(-37,24){\ms{{{\chr}}_Q}}\put(0,27){\ms{P}}
$$
where $\wt x_i$ and ${\wt x}^i$ are the dual basis obtained from   $x_i$ and $x^i$ by the rotation property of Lemma~\ref{P:Omega-nat}.
Then the admissible skein equivalence classes of the bichrome graphs $\Graph_\alpha$ and $\Graph_{\alpha'}$ are equal in this case. Notice that the two red-to blue modifications above are corresponding to the two different versions of Figure \ref{fig:chrom-mod} so that the blue strand corresponding to the red circle is oriented in one direction with color $W_g$ for one move, and in the opposite direction with color $W_{g}^*=W_{g^{-1}}$ in the other case.

Arguing by induction on $m$, suppose now that we proved that $[\Graph_\alpha]=[\Graph_{\alpha'}]$ for each $\alpha,\alpha'$ which intersect transversally at most $m$ times, and let $\alpha,\alpha'$ be two arcs as above intersecting transversally exactly $m+1$ times. 

We claim that there exist an other arc, call it $\beta$ embedded in
$\Sigma\setminus R$, connecting a blue edge $e''$ (possibly coinciding
with $e$ or $e'$) to a red circle $c''$ (possibly coinciding with $c$
or $c'$), transverse to $\Graph$, $\alpha$ and $\alpha'$ and such that
$\card(\alpha\cap \beta)\leq m$ and $\card(\alpha'\cap \beta)\leq
m$. This will conclude the proof as, by inductive hypothesis we will
have $[\Graph_{\alpha}]=[\Graph_\beta]=[\Graph_{\alpha'}]$. To exhibit
$\beta$, let $\alpha_0$ be the subarc of $\alpha$ formed by the
connected component of $\alpha\setminus \alpha'$ containing
$q=\alpha\cap c$ and let $p\in \alpha\cap \alpha'$ be its other
endpoint. The arc $\beta$ is then the concatenation of $\alpha_0$ and
of the subarc of $\alpha'$ connecting $p$ to $e'$. It is
straightforward to see that $\beta$ can be isotoped so that $\beta$ is
disjoint from $\alpha'$ and intersects $\alpha$ at most $m$ times
(corresponding to $(\alpha\cap \alpha')\setminus \{p\}$).
\end{proof}
By Proposition \ref{prop:bichrome_welldef} we can now draw skein diagrams including red circles (oriented and $\Gr$-colored) and these represent well defined classes in $\Skein(\Sigma)$. They satisfy a list of relations summarised here:
\begin{proposition}\label{P:slide} $[\Graph]=[\Graph']$ if $\Graph'$ is obtained from $\Graph$ by one of the following modifications:
\begin{enumerate}
\item Inversion relations: consisting in switching the orientation of
  a $g$-colored red circle and simultaneously changing its color to
  $g^{-1}$, see Figure \ref{fig:change-orient}.
\item Transparency relations (including also the case of a $1$-colored
  red edge).
\item Slide relations: a blue edge of $\Graph$ can be slid over a
  $g$-colored red circle as in Figure \ref{fig:slide}.
\end{enumerate}
\end{proposition}
\begin{proof}
  The invariance under inversion comes directly from the definition of
  the two red-to-blue moves in Figure \ref{fig:chrom-mod}. Indeed by
  the definition given in the figure, a red-to-blue on the left side
  of a $g$-colored red circle is exactly the same as a red-to-blue
  move applied to the right side of a $g^{-1}$-colored red circle
  oriented in the opposite way.

  Concerning transparency relation, if a red circle is colored by
  $1\in \Gr$, then, after the red-to blue modification, it is represented
  by an edge colored by $W_1$ which has degree $1$, and, this edge, by
  the transparency relation of blue edges, can slide over any base
  point. Undoing then the red-to-blue modification proves the relation
  for the red circles.

Finally the slide relation is proved as follows : 
We first consider the case where we want to slide a strand colored by $P\in\Proj$ of degree $h\in \Gr$ over a red curve.  Then we have the following skein relations:
% \FC{The strands in the figures need to be oriented: I could not find the original svg files from the paper with alexis. Do you guys have it ?}
% $$
% \epsh{fig8a}{16ex}\put(-9,26){\ms{P}}
%   =\epsh{fig8b}{18ex}\put(-20,-22){\ms{{{\chr}}_P}}\put(-30,-12){\ms{W_g}}\put(-10,-12){\ms{P}}
%   =\epsh{fig8c}{24ex}\put(-33,-37){\ms{{{\chr}}_P}}\put(-45,-27){\ms{W_g}}\put(-25,-27){\ms{P}}\put(-20,0){\ms{W_{gh}}}\put(-11,27){\ms{P}}
%   \put(-21,18){\ms{{{\chr}}_{P^*}}}
%   \put(-43,37){\ms{x_i}}\put(-34,-12.5){\ms{x^i}}
%   =\epsh{fig8d}{24ex}\put(-33,-37){\ms{{{\chr}}_P}}
%   \put(-21,19){\ms{{{\chr}}_{P^*}}}
%   \put(-46,38){\ms{x^{*i}}}\put(-36,-11){\ms{{{x^*}_i}}}\put(-40,-22){\ms{W_g}}\put(-20,-22){\ms{P}}\put(-30,5){\ms{W_{gh}}}
%   =\epsh{fig8e}{24ex}\put(-31,-15){\ms{{{\chr}}_P}}
%   \put(-20,37){\ms{{{\chr}}_{P^*}}}
%   \put(-40,-39.5){\ms{{x^{*i}}}}\put(-34,8.5){\ms{{x^*}_i}}\put(-40,-28){\ms{W_g}}\put(-20,-28){\ms{P}}\put(-40,-05){\ms{W_g}}\put(-20,-05){\ms{P}}\put(-36,27){\ms{W_{gh}}}\put(-10,20){\ms{P}}
%   =\epsh{fig8f}{18ex}\put(-21.5,26){\ms{{{\chr}}_{P^*}}}\put(-26,-7){\ms{W_{gh}}}
%   =\epsh{fig8g}{16ex}\put(-9,-26){\ms{P}}\put(-15,2){\ms{W_{gh}}}
%   $$
%   \BP{new pictures:}% \chr_{P^{\!*}}
  \[\epsh{fig8ao.pdf}{20ex}
\puts{33}{95}{$\ms{g}$}
\pute{101}{86}{$\ms{P}$}
\ =\ 
\epsh{fig8bo.pdf}{22ex}
\putc{58}{19}{$\ms{\chr_P}$}
\putw{89}{89}{$\ms{P}$}
\putw{45}{43}{$\ms{W_g}$}
\pute{100}{7}{$\ms{P}$}
\ =\ 
\epsh{fig8co.pdf}{24ex}
\putc{25}{85}{$\ms{x_i}$}
\putc{72}{66}{$\ms{\chr_{P^{*}}}$}
\putc{43}{38}{$\ms{x^i}$}
\putc{47}{14}{$\ms{\chr_P}$}
\putsw{59}{80}{$\ms{W_{gh}}$}
\pute{77}{85}{$\ms{P}$}
\pute{100}{79}{$\ms{P}$}
\pute{73}{46}{$\ms{W_{gh}}$}
\pute{57}{23}{$\ms{P}$}
\putw{33}{26}{$\ms{W_g}$}
\putn{19}{7}{$\ms{W_g}$}
\pute{98}{7}{$\ms{P}$}
\ =\ 
\epsh{fig8do.pdf}{24ex}
\putne{66}{93}{$\ms{P}$}
\putw{64}{78}{$\ms{W_{gh}}$}
\pute{100}{81}{$\ms{P}$}
\putnw{52}{56}{$\ms{W_{gh}}$}
\putw{40}{24}{$\ms{W_g}$}
\pute{67}{27}{$\ms{P}$}
\putn{21}{7}{$\ms{W_g}$}
\putne{96}{8}{$\ms{P}$}
\putc{29}{87}{$\ms{x^{\!*i}}$}
\putc{73}{68}{$\ms{\chr_{P^{*}}}$}
\putc{45}{39}{$\ms{{x^{\!*}}_{\!i}}$}
\putc{49}{14}{$\ms{\chr_P}$}
\]\[\ =\ 
\epsh{fig8eo.pdf}{24ex}
\putc{73}{84}{$\ms{\chr_{P^{*}}}$}
\putc{45}{58}{$\ms{{x^{\!*}}_{\!i}}$}
\putc{49}{34}{$\ms{\chr_P}$}
\putc{36}{12}{$\ms{x^{\!*i}}$}
\putn{37}{101}{$\ms{P}$}
\putse{29}{89}{$\ms{W_{gh}}$}
\pute{101}{88}{$\ms{P}$}
\putnw{52}{73}{$\ms{W_{gh}}$}
\pute{69}{48}{$\ms{P}$}
\putw{42}{45}{$\ms{W_g}$}
\putne{95}{22}{$\ms{P}$}
\pute{60}{14}{$\ms{W_g}$}
\ =\ 
\epsh{fig8fo.pdf}{24ex}
\putc{62}{81}{$\ms{\chr_{P^{*}}}$}
\pute{100}{85}{$\ms{P}$}
\pute{59}{44}{$\ms{W_{gh}}$}
\putn{79}{17}{$\ms{P}$}
\ =\ 
\epsh{fig8go.pdf}{20ex}
\puts{34}{90}{$\ms{gh}$}
\pute{100}{12}{$\ms{P}$}
\]
where ${x^*}_i$ and ${x^*}^i$ are the dual basis defined by
${x^*}^i=(x_i)^*\circ(\phi_{W_{gh}}\otimes\Id_{P^*\otimes {W_g}^*})$
and
${x^*}_i=(\phi_{W_{gh}}^{-1}\otimes\Id_{P^*\otimes
  {W_g}^*})\circ(x^i)^*$ (where we recall that $\phi_X:X\to X^{**}$ is
the pivotal structure).  Next, consider the general case where we want
to slide a strand colored by $V \in \cat$ over a red curve. Applying a
skein relation as in Figure \ref{fig:V-appear},%\BPm{modified} 
we can assume there is
% the procedure explained in the proof of Proposition~\ref{prop:ddd}, we
% can push
a strand colored by $P \in \Proj_\cat$ next to the
$V$-colored strand. Inserting coupons colored by identities, we
replace the $V$-colored arc we want to slide by an arc colored by
$V\otimes P\in\Proj_\cat$ which we then slide over the red curve. By
removing then the inserted coupons and sliding back the $P$-colored
strand we obtain the desired result.
\end{proof}

\begin{figure}
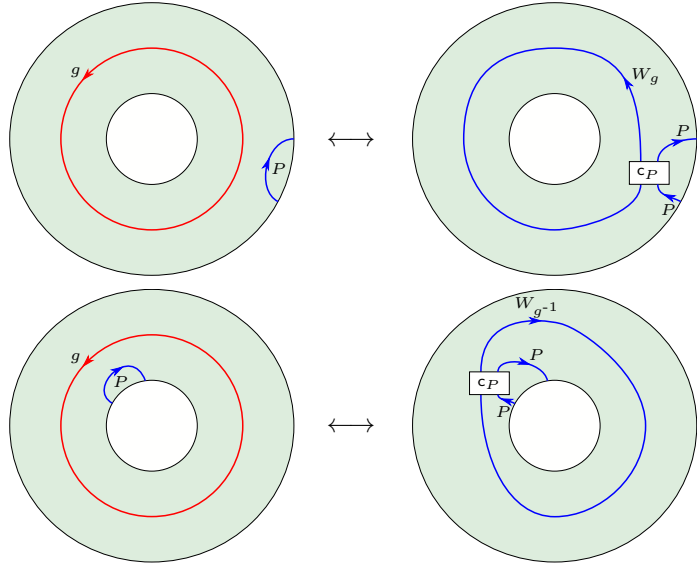

  \centering
\[
  \epsh{fig33n.pdf}{24ex}
  \putnw{25}{74}{$\ms{g}$}
  \putse{92}{41}{$\ms{P}$}
  \quad\longleftrightarrow\quad
  \epsh{fig33q.pdf}{24ex}
  \putne{77}{72}{$\ms{W_g}$}
  \putn{95}{51}{$\ms{P}$}
  \puts{90}{26}{$\ms{P}$}
  \putc{83}{38}{$\ms{\chr_P}$}
\]
\[
  \epsh{fig33p.pdf}{24ex}
  \puts{39}{68}{$\ms{P}$}
  \putnw{25}{74}{$\ms{g}$}
  \quad\longleftrightarrow\quad
  \epsh{fig33r.pdf}{24ex}
  \putn{44}{93}{$\ms{W_{g^{\text-1}}}$}
  \putne{41}{74}{$\ms{P}$}
  \puts{32}{57}{$\ms{P}$}
  \putc{27}{66}{$\ms{\chr_P}$}
\]
  \caption{Chromatic modification}
  \label{fig:chrom-mod}
\end{figure}

\begin{figure}
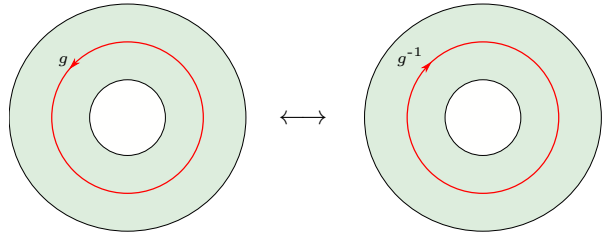

  \centering
\[
  \epsh{fig33s.pdf}{20ex}
  \putnw{25}{74}{$\ms{g}$}
  \quad\longleftrightarrow\quad
  \epsh{fig33t.pdf}{20ex}
  \putnw{25}{74}{$\ms{g^{\text-1}}$}
\]
  \caption{Reversing the orientation of a red circle}
  \label{fig:change-orient}
\end{figure}
\begin{figure}
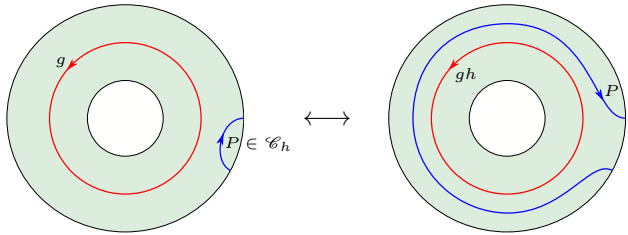

  \centering
  \[
  \epsh{fig33n.pdf}{20ex}
  \putnw{25}{74}{$\ms{g}$}
  \putse{92}{42}{$\ms{P\in\cat_h}$}
  \qquad\longleftrightarrow\quad
  \epsh{fig33o.pdf}{20ex}
  \putse{28}{72}{$\ms{gh}$}
  \pute{91}{62}{$\ms{P}$}
\]
  \caption{Slidding on a red circle}
  \label{fig:slide}
\end{figure}

\section{The 3-dimensional $\Gr$-HQFT} \label{S:GHQFT}
%\BP{No general setting, just write some properties.}

\subsection{Juh\'asz's presentation of the cobordism category.}\label{SS:Juhasz}
% The lift of a generator is the product of gauge/restriction cylinders with a
% standard lift (\ie one with base points $\bs{Y\times0,Y\times 1}$
% where $Y$ is away from the framed sphere and $\rho(*\times[0,1])=1$
% for any $*\in Y$).
% \subsection{Generators of the (decorated) cobordism category}
In this section, following Juh\'asz \cite{Juh18}, we present
elementary cobordisms that generate the cobordism category $\cob$.
%  and its decorated version.
Let $n\in\N^*$ (we will in particular focus on the
$n=2$ case but we will prove Proposition \ref{P:diffeo-are-rel} for
any $n$).  Recall that the objects of $\cob$ are closed oriented
$n$-manifolds and the morphisms are equivalence classes of
$(n+1)$-dimensional cobordisms.  We write $C\simeq C'$ when two
cobordisms are equivalent, meaning that they are diffeomorphic via a
diffeomorphism which is the identity on the boundary.

{\bf Handle cobordisms $\handle(\Sp):\Sigma\to \Sigma(\Sp)$.}\\
We recall the cobordisms obtained by attaching an
$(n+1)$-dimensional index-$k$ handle for
$k\in\bs{0,1,\ldots,n+1}$. Let $\Sigma$ be a
closed oriented $n$-manifold. A \emph{framed $(k-1)$-sphere} in
$\Sigma$ is an orientation-reversing embedding
$$\Sp=\Sp^{k-1}: S^{k-1} \times D^{n-k+1} \hookrightarrow \Sigma\ .$$
Then we can
perform surgery on $\Sigma$ along $\Sp$ by removing the interior of
the image of $\Sp$ and gluing in $D^{k} \times S^{n-k}$, obtaining a
well-defined topological manifold $\Sigma(\Sp)$ which, using the
framing of the sphere, can be endowed with a canonical smooth
structure.

Given a
framed $(k-1)$-sphere $\Sp$ we denote by $\handle(\Sp)$ the associated smooth
cobordism from $\Sigma$ to $\Sigma(\Sp)$ (unique up to diffeomorphism
relative to the boundary; see for instance \cite[Theorem 3.13]{Mil65b})
which is topologically the space
\begin{equation}
  \label{eq:handle-cob}
  \handle(\Sp):= (\Sigma\times[-1,1])\cup_{\Sp\times\bs1} (D^{k} \times D^{n-k+1})\ :\Sigma \to \Sigma(\Sp).
\end{equation}
Note that in the case $k=0$, a $(-1)$-framed sphere $\Sp$ in $\Sigma$
is the unique map $\emptyset\to \Sigma$.  In this case
$\Sigma(\Sp)=\Sigma\sqcup S^n$ and
$\handle(\Sp)=\Sigma\times[-1,1]\sqcup D^{n+1}$.
\\

{\bf Mapping cylinder $\cyl(\psi):\Sigma\to \Sigma'$.}\\
For any orientation-preserving diffeomorphism $\psi:\Sigma\to\Sigma'$
between two closed oriented $n$-manifolds, the mapping cylinder of $\psi$
is the cobordism
\begin{equation}
  \label{eq:mapping-cylinder}
  \cyl(\psi):= \Sigma\times[0,1] :\Sigma \to \Sigma',
\end{equation}
where $\Sigma\times\bs{0}$ is identified with $\Sigma$ and
$\Sigma\times\bs{1}$ is identified with $\Sigma'$ by $\psi$.

Juh\'asz shows that the cobordisms $\bs{\handle(\Sp)}_\Sp$ indexed by the
framed spheres in closed oriented $n$-manifolds together with the
mapping cylinders $\bs{\cyl(\psi)}_\psi$ indexed by the orientation-preserving
diffeomorphisms generate the category $\cob$, and he gives a complete
set of relations:
\begin{enumerate}
\item[(R1)] For composable diffeomorphisms $\psi$ and $\psi'$ between closed
  oriented $n$-mani\-folds, we have the relation
  $\cyl({\psi \circ \psi'}) \simeq \cyl(\psi)\circ \cyl({\psi'})$.  We also have the
  relations % $e_{\Sigma,\emptyset}\simeq \cyl({\Id_\Sigma}$ and
  $\cyl(\psi)\simeq \Id_\Sigma\in \End_{\cob}(\Sigma)$ if $\psi\co \Sigma \to \Sigma$ is a
  diffeomorphism isotopic to the identity.
\item[(R2)] Let $\psi \co \Sigma \to \Sigma'$ be an orientation
  preserving diffeomorphism between closed oriented $n$-manifolds and $\Sp$
  be a framed sphere in $\Sigma$.  Consider the framed sphere
  $\Sp'=\psi\circ \Sp$ in $\Sigma'$ and denote by
  $\psi^\Sp\co \Sigma(\Sp)\to \Sigma'(\Sp')$ the induced diffeomorphism.
  Then the commutativity of the following diagram defines a relation:
$$\xymatrix{
\Sigma \ar[d]_{\cyl({\psi})} \ar[r]^{\handle({\Sp})} & \Sigma(\Sp) \ar[d]^{\cyl({\psi^\Sp})} \\
\Sigma' \ar[r]^{\handle({\Sp'})} & \Sigma'(\Sp')}$$
\item[(R3)] Let $\Sp, \Sp'$ be disjoint framed sphere in an oriented
  $n$-manifold $\Sigma$.  Then $\Sp$ (resp. $\Sp'$) also defines a framed
  sphere $\Sp_{\Sp'}$ in $\Sigma(\Sp')$ (resp. $\Sp'_{\Sp}$ in
  $\Sigma(\Sp)$). Notice that
  $\Sigma(\Sp)(\Sp'_{\Sp})= \Sigma(\Sp')(\Sp_{\Sp'})$ and denote this
  $n$-manifold by $\Sigma(\Sp,\Sp')$.  The commutativity of the following
  diagram defines a relation:
$$\xymatrix{
\Sigma \ar[d]_{\handle(\Sp')} \ar[r]^{\handle({\Sp})} & \Sigma(\Sp) \ar[d]^{\handle(\Sp'_\Sp)} \\
 \Sigma(\Sp')\ar[r]^{\handle(\Sp_{\Sp'})} & \Sigma(\Sp,\Sp')}$$

\item[(R4)] Let $\Sp$ be a framed $k$-sphere in an oriented $n$-manifold
  $\Sigma$ and $\Sp'$ a framed $k'$-sphere in $\Sigma(\Sp)$.  If the
  attaching sphere $\Sp'(S^{k'}\times \{0\})\subset \Sigma(\Sp)$
  intersects the belt sphere
  $\{0\}\times S^{-k+1} \subset \Sigma(\Sp)$ once transversely, then
  there is a diffeomorphism (well defined up to isotopy)
  $\phi \co \Sigma \to \Sigma(\Sp)(\Sp')$ (see \cite[Definition
  2.17]{Juh18}) and the following is a relation:
 $$
 \handle(\Sp') \circ  \handle(\Sp) \simeq \cyl({\phi}).
 $$
\item[(R5)] For each framed $k$-sphere $\Sp$ in an oriented
  $n$-manifold $\Sigma$, there is a relation
  $\handle(\Sp) \sim \handle(\bar\Sp) $, where the framed
  $k$-sphere $\bar \Sp \co S^k \times D^{2-k} \hookrightarrow \Sigma $
  is defined by $\bar\Sp(x,y)= \Sp(r_{k+1}(x),r_{2-k}(y))$ for any
  $x\in S^k \subset \R^{k+1}$ and $y\in D^{2-k}\subset \R^{2-k}$, with
  $r_m(x_1,x_2, \dots, x_m)=(-x_1,x_2, \dots, x_m)$.
\end{enumerate}

One can modify this presentation to get a presentation of the
subcategory $\cob'$ of cobordisms such that each component of every
cobordism has a non-empty source and non-empty target (see
\cite{Juh18}): For this, one only has to remove the generators
$\handle(\Sp)$ where $\Sp$ is a framed attaching sphere of index $0$ or
$n+1$ and ignore the relations that involve them.

Similarly, a presention of the category the ``non-compact'' cobordism
category $\cobnc$ such that each component of every cobordism has a
non-empty source is obtained by removing the generators $\handle(\Sp)$ where
$\Sp$ is a framed attaching sphere of index $0$ and ignore the
relations that involve them (see \cite{CGPV23}).

\subsection{$\Gr$-Decorated handles.}
In this subsection, we assume $n=2$ and we describe generators of the
category $\wt\cob$ of $\Gr$-decorated $(2+1)$-cobordisms.

{\bf Mapping cylinder $\wt\cyl(\psi):\wt\Sigma\to \wt\Sigma'$.}\\
% If $\psi:\Sigma\to \Sigma'$ is a diffeomorphism the mapping cylinder of $\psi$ is the
% cobordism $\cyl(\psi)=\Sigma\times[0,1]$ where $\Sigma\times\bs0\subset\partial \cyl(\psi)$
% is identified to $-\Sigma$ (which is $\Sigma$ with the opposite orientation) and
% $\Sigma\times\bs1\subset\partial \cyl(\psi)$ is identified to $\Sigma'$ using $\psi$.
If $\wt\Sigma$ and $\wt\Sigma'$ are $\Gr$-decorated surfaces, and $\psi$ is an isomorphism of $\Gr$-decorated surfaces, so $\psi(Y)=Y'$ and $\psi_*\rho = \rho'$, then it induces a $\Gr$-decorated cobordism $\wt \cyl(\psi)$ given by $\cyl(\psi)$ equipped with the unique $\Gr$-representation $\rho_\cyl$ compatible with the boundary and satisfying $\rho_\cyl(\{y\}\times [0,1]) = 1$ for every $y\in Y$.

{\bf Standard $\Gr$-decorated handle cobordisms.}\\
We now turn to define $\Gr$-decorated handle attachments.  A $\Gr$-decorated
handle cobordism $\wt C$ is a $\Gr$-decorated cobordism whose
underlying cobordism is a handle cobordism in $\cob$.  We distinguish
among them a subclass of standard $\Gr$-decorated handle cobordisms
with some compatiblity conditions for the base points in their source
and target.  Let us denote by $1\in S^{k-1}\subset D^{k}$ the first
unit vector of the canonical bases of $\R^{k}$ and $-1$ its antipodal
element, for $k\geq 1$.

% \BP{maybe insist that a decorated handle attachments is just a handle attachment with any decoration and only restriction on the set of base points}
\begin{definition} \label{D:StdDecHandles}
  Assume $n=2$ and let $\wt \Sigma=(\Sigma,Y,\rho)$ be a $\Gr$-decorated
  surface.  Let $\Sp$ be a framed $(k-1)$-sphere in $\Sigma$ and for $\varepsilon,\eta \in \{\pm\}$ let
  $y_{\varepsilon\eta}=\Sp(\varepsilon\times \eta)\in\Sigma$.  We say that $\Sp$ is a \emph{standard framed
    $(k-1)$-sphere} in $\wt \Sigma$ if the conditions below are
  met. In this case, we define the \emph{surgered $\Gr$-decorated surface}
  $\wt\Sigma(\Sp)=(\Sigma(\Sp), Y(\Sp), \rho(\Sp))$. For $k=1$, these
  require some extra data $g\in \Gr$.  % \BHm{removed: "if no base point
    % of $Y$ is in the interior of the framed sphere $\Sp$" This cannot
    % be ensured for attaching a 3-handle. Indeed, we're attaching on a
    % 2-sphere, and there must be at least one marked point in this
    % connected component 2-sphere. I replace with the conditions
    % below.}
\begin{description}
\item[If $k=0$] The attaching sphere is empty. We ask no condition. We set 
$$\wt\Sigma(\Sp) := (\Sigma\sqcup S^2, Y\sqcup \{1\}, \rho)$$ where we have identified $\Pi_1(S^2,1) \simeq \ast$ as the trivial groupoid.

\item[If $k=1$] The attaching sphere is two disks. We ask that $Y \cap \Sp = \{y_{++},y_{-+}\}$ is one point on each disk. For an arbitrary element $g\in\Gr$, we set % \BHm{What do you think of this notation?} \FCm{OK for me: thanks !}
$$\wt\Sigma(\Sp,g) := (\Sigma(\Sp), Y, \rho(\Sp,g))$$
where $\rho(\Sp,g)$ is the unique $\Gr$-representation which agrees with $\rho$ on the exterior of the two disks $\Pi_1(\Sigma,Y)\simeq \Pi_1(\Sigma\smallsetminus \Sp,Y)\to \Pi_1(\Sigma(\Sp),Y)$ and evaluates to $g$ on the path $\gamma = D^1\times \{1\}\subseteq D^1\times S^1\subseteq \Sigma(\Sp)$. 

\item[If $k=2$] The attaching sphere is an annulus. We ask that $Y \cap \Sp = \{y_{++},y_{+-}\}$ and that the core $\gamma = \Sp(S^1\times\{1\}) \in \Pi_1(\Sigma,Y)$ evaluates to $\rho(\gamma)=1 \in \Gr$. We set 
$$\wt\Sigma(\Sp) := (\Sigma(\Sp), Y, \overline\rho)$$
where $\overline\rho$ is the unique $\Gr$-representation on $\Sigma(\Sp)$ which agrees with $\rho$ on $\Sigma\smallsetminus \Sp$. It is well-defined because $\rho(\gamma)=1$. 

\item[If $k=3$] The attaching sphere is a 2-sphere, so $\Sigma = \Sigma'\sqcup S^2$. We ask that $Y \cap \Sp = \{y_+\}$. We set 
$$\wt\Sigma(\Sp) := (\Sigma', Y\cap \Sigma', \rho\vert_{\Sigma'})\ . $$
\end{description}
In all of these cases, we define the \emph{$\Gr$-decorated $k$-handle attachment 3-cobordism} $$\wt \handle(\Sp) := (\handle(\Sp), \rho_\handle) : \wt\Sigma\to \wt\Sigma(\Sp)$$ where $\rho_\handle$ is the unique $\Gr$-representation compatible with the boundary and satisfying $\rho_\handle(\{y\}\times [0,1]) = 1$ for every $y\in Y\cap (\Sigma\smallsetminus \text{interior}(\Sp))$. For $k=1$, we denote it $\wt \handle(\Sp,g)$ when we need to emphasize the dependence on the choice of $g\in\Gr$.
%\BP{Not finished**************
%2) let $y_\pm=\Sp(1\times(\pm1))\in\Sigma$, then $y_+,y_-\in Y$ (this condition is void if $k=-1$ and $\Sp=0$) 3) If $k=1$, $\rho(\Sp(S^1\times\bs1))=1_\Gr$ where $\Sp(S^1\times\bs1)$ is seen as a loop $y_+\leadsto y_+$.
%.  Let $\rho\in\Rep_\Gr(\Sigma,Y)$, then we say that $\Sp$ is a standard framed sphere for the decorated surface $(\Sigma,Y,\rho)$}
\end{definition}

We will see that mapping cylinders, standard $\Gr$-decorated handle
cobordisms and equivalences of $\Gr$-decorated surfaces generate $\wt\cob$.

\subsection{Handle maps and handle cancellation}\label{S:handlemaps}
%\BP{We only consider standard decorated handle cobordism.}
% \begin{itemize}
% \item 3-handle $\to \F'$
% \item 2-handle along a curve with trivial holonomy $\to$ cutting map
% \item degree $g$ 1-handle $\to$ assume there are no base points in the
%   attaching framed sphere and two base points $*_\pm$ at the end of
%   the path going throught the handle
%   $\gamma=(1,0)\times[0,1]\subset S^1\times[0,1]$.  Then the
%   representation $\rho$ of the decorated cobordism is determined by
%   its restriction on the source surface and by $\rho(\gamma)=g$.
%   Then we add a red meridian colored with $g$ at the belt circle.
% \end{itemize}
\begin{definition}
    Let $\wt\Sigma$ be a $\Gr$-decorated surface, $\Sp:S^2\hookrightarrow\Sigma$ a standard attaching 2-sphere and $\wt \handle(\Sp): \wt\Sigma\to \wt\Sigma(\Sp)$ the standard $\Gr$-decorated 3-handle attachment of Definition \ref{D:StdDecHandles}. We write $\wt\Sigma = \wt\Sigma'\sqcup S^2$. 
    
    The \emph{3-handle map} is the linear map:
    \begin{equation*}
        \Skein(\wt \handle(\Sp)) : \begin{array}{cll}
         \Skein(\wt\Sigma'\sqcup S^2) &\to & \Skein(\wt\Sigma')\\
         \Graph\sqcup \Graph' &\mapsto & \F'(\Graph')\cdot  \Graph
        \end{array}
    \end{equation*}
    given by applying the invariant of admissible skeins in $S^2$ from \eqref{E:DefF'}.
\end{definition}
\begin{definition}
    Let $\wt\Sigma$ be a $\Gr$-decorated surface, $\Sp:S^1\times D^1\hookrightarrow\Sigma$ a standard framed 1-sphere and $\wt \handle(\Sp): \wt\Sigma\to \wt\Sigma(\Sp)$ the standard $\Gr$-decorated 2-handle attachment of Definition \ref{D:StdDecHandles}. We write $\wt\Sigma = \wt\Sigma'\underset{S^1\times S^0}{\cup} S^1\times D^1$ and $\wt\Sigma(\Sp) = \wt\Sigma'\underset{S^1\times S^0}{\cup} D^2\times S^0$. 
    
    The \emph{2-handle map} is the linear map:
    \begin{equation*}
        \Skein(\wt \handle(\Sp)) : \begin{array}{cll}
         \Skein(\wt\Sigma) &\to & \Skein(\wt\Sigma(\Sp))\\
         \Graph &\mapsto & \Graph'
        \end{array}
    \end{equation*}
where $\Graph$ is assumed to intersect the core of $\Sp$ transversely and have standard form below, and $\Graph'$ is obtained from $\Graph$ by removing $\Graph \cap S^1\times D^1$ and replacing it with the skein in $D^2\times S^0$ below.
\[
\Graph=\epsh{fig6c}{10ex}\quad \mapsto\quad
\Graph'=\sum_i\epsh{fig6d}{10ex}\putc{40}{65}{${x_i}$}
\putc{80}{37}{${x^i}$}\;\;.
\]
where $\sum_i x_i\otimes x^i$ is the copairing \eqref{E:copairing}.
\end{definition}
A priori, this map depends on how we isotoped $\Graph$ to be transverse to the core of $\Sp$, and chose an ordering on the intersection points. However, by Lemma \ref{P:Omega-nat} (2) and (3), one can slide any coupon through the cutting morphism and change the linear ordering inducing the same cyclic ordering. So this map is well-defined. Moreover, by Lemma \ref{P:Omega-nat}(1), it is invariant under reversing the attaching sphere.

\begin{definition}
  Let $\wt\Sigma$ be a $\Gr$-decorated surface,
  $\Sp:S^0\times D^2\hookrightarrow\Sigma$ a standard framed 0-sphere,
  $g \in \Gr$ any element and
  $\wt \handle(\Sp,g): \wt\Sigma\to \wt\Sigma(\Sp,g)$ the standard
  $\Gr$-decorated 1-handle attachment of Definition
  \ref{D:StdDecHandles}. We write
  $\wt\Sigma = \wt\Sigma'\underset{S^1\times S^0}{\cup} D^2\times S^0$
  and
  $\wt\Sigma(\Sp,g) = \wt\Sigma'\underset{S^1\times S^0}{\cup}
  S^1\times D^1$. We denote
  $\gamma = S^1\times\{0\} \subseteq \wt\Sigma(\Sp,g)$ the cocore and
  by $\gamma^g$ the curve $\gamma$ in red colored by $g$.
    
    The \emph{1-handle map} is the linear map:
    \begin{equation*}
        \Skein(\wt \handle(\Sp)) : \begin{array}{cll}
         \Skein(\wt\Sigma) &\to & \Skein(\wt\Sigma(\Sp))\\
         \Graph &\mapsto & \Graph \cup \textcolor{red!80!black}{\gamma^g}
        \end{array}
    \end{equation*}
where $\Graph$ is assumed to be disjoint from $\Sp$ hence give a skein in $\wt\Sigma' \subseteq \wt\Sigma(\Sp,g)$, and $\gamma$ is a $g$-colored red circle which is interpreted as a skein using Figure \eqref{fig:chrom-mod}. Graphically, we get:
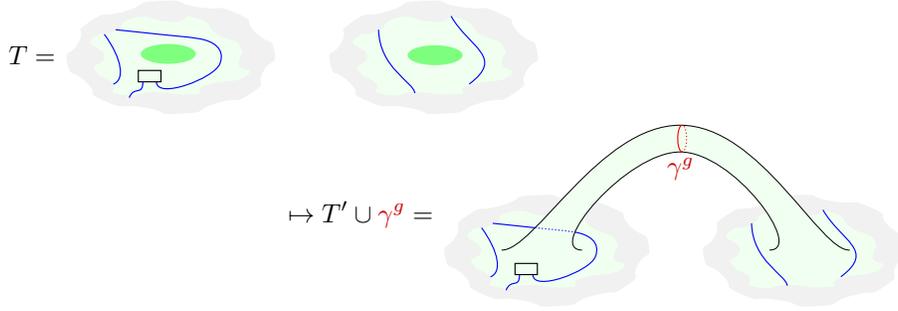
\begin{multline*}
        \Graph=\epsh{fig9e}{10ex}
    \\
    % \quad \rightsquigarrow\quad \Graph'= \epsh{fig9b}{10ex}\\
\mapsto \Graph'\cup \textcolor{red!80!black}{\gamma^g} = \epsh{fig9c}{16ex}\puts{51}{82}{$\textcolor{red!80!black}{\gamma^g}$}
\end{multline*}
\end{definition}
Again, this map a priori depends on how we turned $\gamma$ blue and how to isotope $\Graph$ to be disjoint from $\Sp$. Two such isotopies are related by a sliding relation as in Figure \ref{fig:change-orient} which give equivalent results by Proposition \ref{P:slide}(3). Two ways of turning $\gamma$ blue give equivalent results by Proposition \ref{prop:bichrome_welldef}. Hence this map is well-defined. Moreover, it is invariant under reversing the attaching sphere and trading $g$ to $g^{-1}$ as in Figure \ref{fig:slide} by Proposition \ref{P:slide}(1).

% \paragraph
{\textbf{Cancellation of 1 and 2-handles}}
Let $\wt\Sigma$ be a $\Gr$-decorated surface, $\Sp^0$ a standard framed 0-sphere in $\wt\Sigma$ and $\Sp^1$ a standard framed 1-sphere in $\wt\Sigma(\Sp^0)$ which cancels with $\Sp^0$. 

The undecorated handle attachments cancel in the sense that they compose to a mapping cylinder for some diffeomorphism $\psi : \Sigma \to \Sigma(\Sp^0)(\Sp^1)$ which is the identity outside a ball in $\Sigma$.
By our requirements on standard handle attachments for $\Gr$-decorated surfaces, this also holds at the level of $\Gr$-decorated cobordisms: 
$$\wt \handle(\Sp^1)\circ \wt \handle(\Sp^0,g) \simeq \wt M_\psi\ .$$

\begin{proposition}\label{P:hm-12-cancel}
  For a cancelling pair of 1 and 2-handles as above, the skein handle maps satisfy the handle cancellation:
  $$\Skein(\wt \handle(\Sp^1))\circ\Skein(\wt \handle(\Sp^{0},g))=\Skein(\wt M_\psi)$$
  agree as linear map $\Skein(\wt\Sigma)\to \Skein(\wt\Sigma(\Sp^0)(\Sp^1))$.
\end{proposition}
\begin{proof}
  Given a skein in $\Skein(\Sigma\setminus (Y\cup\Sp^{0}))$, we can
  assume it intersects the circle of $\Sp^1$ at a unique edge colored
  by some $V\in\cat_g$. Note that the degree of $V$ is imposed by the condition tat $\Sp^1$ is a standard $\Gr$-decorated handle attachment, and attaches on a circle which evaluates to 1 under $\rho$. Then the claim follows by the defining relation of chromatic maps, i.e:
    \begin{align*}
      \epsh{fig29d.pdf}{5ex}\putw{42}{53}{$\ms{V}$}\putw{29}{56}{$\ms{y_{\!-\!+}}\!{\cdot }$}
\putc{65}{55}{${\cdot }\ms{y_{\!+\!+}}$}
      \putsw{76}{23}{$\ms{P}$}\tto{\Skein(\wt M(\Sp^0))}% \Skein_\cat(e_{\Sigma, \Sp^0})((\Sigma,\GammaT))=
      &
        \epsh{fig29a.pdf}{10ex} \putse{42}{30}{$\ms{V}$}
        \putse{78}{22}{$\ms{P}$} \pute{48}{86}{$\ms{g}$}
= \\
      \epsh{fig29b.pdf}{12ex} \putw{40}{22}{$\ms{V}$}
\putw{35}{53}{$\ms{\Sp^1}$}
%\putc{48}{83}{$\ms{C}$}
\putne{47}{97}{$\ms{W_g}$}
      %\putne{76}{25}{$\ms{E}$}
      \putne{76}{22}{$\ms{P}$}
      \putc{49}{84}{$\ms{{\chr}_P}$} %\puts{45}{27}{$\ms{\Sp^1}$}
      \!\!\tto{\Skein(\wt \handle(\Sp^1))}\!\!%{\Skein_\cat(e_{\Sigma, \Sp^1})}
      &\epsh{fig29p.pdf}{12ex}
        \putne{76}{22}{$\ms{P}$} \putc{47}{65}{$\ms{\chr_P}$}
        \putc{41}{84}{$\ms{x_i}$}
        \putc{40}{43.5}{$\ms{x^i}$} \putw{40}{22}{$\ms{V}$}
        % \epsh{fig29c.pdf}{10ex} \putlc{76}{22}{$\ms{P}$}
      % \putc{43}{85}{$\ms{x_i}$}
      % \putc{47}{65}{$\ms{{\chr}_P}$} \putc{43}{44}{$\ms{x^i}$}
      \\[2ex]
      \overset{\eqref{eq:chrP}}= \epsh{fig29d.pdf}{5ex}
      \putsw{76}{23}{$\ms{P}$}\putw{42}{53}{$\ms{V}$}
\ .\quad% =&(\Sigma,\GammaT)
    \end{align*}
    where we have left the diffeomorphism $\psi$ implicit in the last equality.
\end{proof}
% \epsh{fig29a.pdf}{10ex}
% \putse{42}{30}{$\ms{A}$}
% \pute{48}{86}{$\ms{B}$}
% \pute{78}{30}{$\ms{C}$}
% \epsh{fig29b.pdf}{10ex}
% \putw{40}{22}{$\ms{A}$}
% \putw{35}{53}{$\ms{B}$}
% \putc{48}{83}{$\ms{C}$}
% \putne{47}{97}{$\ms{D}$}
% \putne{76}{25}{$\ms{E}$}
% \epsh{fig29d.pdf}{10ex}
% \putw{42}{53}{$\ms{A}$}
% \pute{78}{49}{$\ms{B}$}
% \putc{29}{56}{$\ms{C}$}
% \putc{65}{55}{$\ms{D}$}
% \epsh{fig29p.pdf}{10ex}
% \putc{41}{83}{$\ms{A}$}
% \putc{46}{65}{$\ms{B}$}
% \putc{40}{44}{$\ms{C}$}
% \putw{40}{22}{$\ms{D}$}
% \pute{76}{25}{$\ms{E}$}
% \paragraph
{\textbf{Cancellation of 2 and 3-handles}}
Let $\wt\Sigma$ be a $\Gr$-decorated surface, $\Sp^1$ a standard framed 1-sphere in $\wt\Sigma$ and $\Sp^2$ a standard framed 2-sphere in $\wt\Sigma(\Sp^1)$ which cancels with $\Sp^1$. 

The undecorated handle attachments cancel in the sense that they compose to a mapping cylinder for some diffeomorphism $\psi : \Sigma \to \Sigma(\Sp^1)(\Sp^2)$ which is the identity outside a ball in $\Sigma$.
By our requirements on standard handle attachments for $\Gr$-decorated surfaces, the $\Gr$-decorated cobordisms compose to this mapping cylinder composed with a restriction cylinder $e$ forgetting the basepoint in the attaching 2-sphere:
$$\wt \handle(\Sp^2)\circ \wt \handle(\Sp^1) \simeq \wt M_\psi \circ e\ .$$

\begin{proposition}\label{P:hm-23-cancel}
  For a cancelling pair of standard 2 and 3-handle cobordisms as above,
  the skein handle maps satisfy the handle cancellation:
  $$\Skein(\wt \handle(\Sp^2))\circ\Skein(\wt \handle(\Sp^{1}))=\Skein(\wt M_\psi)\circ \Skein(e)$$
  as linear map $\Skein(\wt\Sigma)\to \Skein(\wt\Sigma(\Sp^1)(\Sp^2))$.
\end{proposition} 
\begin{proof}
  This follows by defining relation of
  $\Omega_P = \sum_ix^i\otimes x_i$ as a copairing for the modified
  trace $\mt_P$, i.e. a skein which in the disk where the handle
  cancellation happens is a single strand with a unique coupon colored
  by $f \in \Hom_\cat(\unit,P)$ will map to the same skein with coupon
  colored by $\sum_i \mt_P(f\circ x^i) x_i$ which is $f$ as
  $\sum_ix^i\otimes x_i$ is a copairing.
\end{proof}

\subsection{Main result}
Let $\wt\cobnc$ be the subcategory of $\wt\cob$ formed by
$\Gr$-decorated cobordisms whose underlying cobordism is in $\cobnc$.
\begin{theorem}\label{T:main}
  Let $\cat$ be a $\Gr$-chromatic category. Then there is a unique
  non-compact $\Gr$-HQFT
  $$\Skein:\wt{\cobnc} \to \operatorname{Vect}$$
  which sends % associate
  a $\Gr$-decorated surface $\wt\Sigma$ to $\Skein(\wt\Sigma)$, a
  gauge cylinder $C_{\Sigma,Y,\rho}^\vp$ to $\vp_*$, a restriction
  cylinder $R_{\Sigma,Y_1,Y_2,\rho}$ to the restriction map
  $r(\Sigma,Y_1,Y_2,\rho)$, a mapping cylinder $C_\psi$ to the map
  $\psi_*$, and each standard handle attachment to the handle maps of
  Section \ref{S:handlemaps}.
\end{theorem}
\begin{corollary}\label{cor:invariant}
  Assume that the modified trace $\mt$ on $\cat$ is unique up to a
  scalar. Let $M$ be a closed $3$-manifold and $\rho:\pi_1(M)\to \Gr$
  be a representation up to conjugation.  Define
  $\hat{M}=M\setminus B^3$, fix $*\in \partial \hat{M}$ and extend
  arbitrarily $\rho$ to $\hat{\rho}:\pi_1(\hat{M},*)\to \Gr$. Then
  $\Skein(\hat{M},\hat{\rho},*)=\Skein({M},{\rho} ) \F'$ for some
  scalar $\Skein({M},{\rho})\in \kk$ independent on the choices of
  $\hat{\rho}$ and $*$.
\end{corollary}
\begin{proof}
Since $\mt$ is unique by Theorem \ref{T:DiskRmt} we have $\dim_\kk(\Skein(S^2,\hat{\rho},*))=1$. Furthermore if $e: (S^2,\hat{\rho},*)\to (S^2,\hat{\rho}',*)$ is an equivalence then $\F'\circ \Skein(e)=\F'$ because by the standard properties of the modified traces, if $\Graph$ and $\Graph'\subset S^2$ differ by an isotopy of an edge passing over $*$ then  $\F'(\Graph)=\F'(\Graph')$. 
\end{proof}
We prove Theorem \ref{T:main} in the rest of this subsection.
%We prove this result in the next subsection.
%To prove this result, we will need a presentation of the decorated cobordism category. 

%\subsection{Existence of the HQFT}
The following statement essentially follows from \cite{Juh18}, though
it is not stated as such. We recall the proof for the reader's
convenience, see also \cite{Syt25} for a detailed proof of a stronger
statement.

Recall the presentation of the category $\cob$ (resp. $\cob'$,
resp. $\cobnc$) of $(n+1)$-cobordims in Subsection \ref{SS:Juhasz}. If
$w$ is a word in the generators, we denote by $|w|$ the corresponding
$n+1$-cobordism.  Then each of the relation (R1)---(R5) of the form
$w\simeq w'$ defines a diffeomorphism which is an equivalence of
cobordisms $|w|\overset\sim\to |w'|$.  We call such a diffeomorphism a
\emph{relation diffeomorhism} of $\cob$ (resp. $\cob'$,
resp. $\cobnc$).

% \BP{Ben I realize we would like the next proposition for $\cobnc$.
%   Does it make a difference in the proof?}
\begin{proposition}\label{P:diffeo-are-rel}
  Let $n\in\N^*$ and consider the presentation of the category $\cob$
  (resp. $\cob'$, resp. $\cobnc$) of $(n+1)$-cobordims in Subsection
  \ref{SS:Juhasz} (we stress that in this proposition the cobordism are not $\Gr$-decorated).  Let $w_1, w_2$ be words of generators and
  $\abs{w_1},\abs{w_2}$ the corresponding $n+1$-cobordisms. Any
  diffeomorphism $D:\abs{w_1}\to\abs{w_2}$ can be realized, up to an
  isotopy which is the identity on the boundary, as a sequence of
  relation diffeomorphisms of $\cob$ (resp. $\cob'$, resp. $\cobnc$).
\end{proposition}
\begin{proof}
  The main idea used in the presentation of the cobordism category is
  that writing a cobordism $W$ as a composition of handle attachments
  is similar to equipping this cobordism with a Morse function
  $f$. The handle attachments correspond to the critical
  points. Therefore, one can understand changing the handle
  decomposition as changing the Morse function, which itself is
  well-understood via Cerf theory.

  Actually, a Morse function is not really enough data, and in
  \cite{Juh18} the author talks of Morse datum, which is a Morse
  function $f$ together with a pseudo gradient vector field $v$ and a
  set of cutting values $\underline b$. To fully have an equivalence
  between Morse data and handle decomposition, one must actually add a
  last piece of data: a chart $\phi$ near each critical point where
  $f$ is in standard form
  $f = \sum_{i=1}^k -x_i^2+\sum_{i=k+1}^n x_i^2$, and
  $v = \operatorname{grad } f$.

  Being obtained as a composition of generators, the cobordisms
  $W_i:=\abs{w_i}$ come equipped with such Morse data, i.e. a Morse
  function $f_i$ together with a pseudo gradient vector field $v_i$, a
  set of cutting values $\underline b_i$ and standard chart near
  critical points $\underline \phi_i$. Let us denote these Morse data
  $\FF_1$ and $\FF_2$ respectively.

%We will consider three cobordisms equipped with Morse data:
%$$ (W_1, f_1, v_1,\underline b_1)\quad , \quad (W_2, D^{-1}\circ f_1, D_*v_1,\underline b_1) \quad \text{and} \quad (W_2, f_2, v_2,\underline b_2)$$
%To lighten notations, let us denote these Morse data $\FF_1, D_*\FF_1$ and $\FF_2$ respectively.
%We can now think of the diffeomorphism $D:\abs{w_1}\to\abs{w_2}$ as a two step process. First we have a Morse-datum preserving diffeomorphism $D:(W_1,\FF_1)\to (W_2,D_*\FF_1)$. Then, we have a path of Morse data $(W_2,D_*\FF_1) \to (W_2,\FF_2)$. 
 
  Now, consider the diffeomorphism $D:\abs{w_1}\to\abs{w_2}$. We can
  think of $D$ as a two-step process. First we have a Morse-datum
  preserving diffeomorphism $D:(W_1,\FF_1)\to (W_2,D_*\FF_1)$, where
  $D_*\FF_1 = (f_1\circ D^{-1}, D_*v_1,\underline b_1, D\circ
  \underline{\phi}_1)$ is the Morse datum $\FF_1$ transported under
  $D$. Then, we have a path of Morse data
  $(W_2,D_*\FF_1) \to (W_2,\FF_2)$.  The relations (R1)--(R5) of
  \cite{Juh18} (see Subsection \ref{SS:Juhasz}) are actually of two
  types: those induced by a Morse-datum preserving diffeomorphism, and
  those induced by a path of Morse data.

  Following the proof of \cite[Thm. 1.7]{Juh18}, we see that the
  Morse-datum preserving diffeomorphism
  $D:(W_1,\FF_1)\to (W_2,D_*\FF_1)$ is induced by repeatedly applying
  relation (R2) and the first case of relation (R1). On the other
  hand, by construction the relations coming from the path of Morse
  data induce a diffeomorphism of $W_2$ which is isotopic to the
  identity.

  As argued in \cite{Juh18}, a path between Mose data which do not
  have $0$-handles can avoid $0$-handles throughout, and the
  diffeomorphism is induced by non-compact relations diffeomorphisms.
  Similarly, avoiding $0$ and $n+1$-handles we get the result for
  $\cob'$.
%In order to understand the relation between any two Morse data, we proceed step by step by modifying only one piece of data.
%
%If two Morse data $\FF$ and $\FF'$ on the same manifold $W$ have the same $f,v$ and $\underline b$ but different standard charts near critical points $\underline \phi$, then by \cite[Rem. 2.13]{Juh18}, the induced handle decompositions are related by reversing the attaching sphere, relation (5), and isotopies, second case of relation (1) and relations (1) and (2) for diffeomorphism isotopic to the identity. So all of these relations indeed induce self-diffeomorphisms of $W$ which are isotopic to the identity. 
%
%If $\FF$ and $\FF'$ differ only by changing $\underline b$, 
%
%Indeed, following the proof of \cite[Thm. 2.24]{Juh18}, we see that the path of Morse data is induced by a sequence of critical point creations and cancelations (4), critical value crossings (3), isotopies of the gradient (1), adding or removing regular values (1), and left-right equivalences (2). All these relations induce a diffeomorphism of $W_2$ that is isotopic to the identity. The only non-trivial one to show is the left-right equivalences
\end{proof}
\begin{lemma}\label{L:[eq,gen]}
  % \BP{equivalence maps commute with standard handle maps and mapping
  %   cylinder}
  If $\wt C_1$, $\wt C_2$ are standard $\Gr$-decorated handle attachment of
  index $k\in\bs{1,2,3}$ with $\abs{\wt C_1}=\abs{\wt C_2}$ or if they are
  mapping cylinders associated to the same diffeomorphisms, and if
  $e,e'$ are equivalences of $\Gr$-decorated surfaces such that
  $\wt C_2=e'\circ\wt C_1\circ e$, then
  $\Skein(\wt C_2)=\Skein(e')\circ\Skein(\wt C_1)\circ \Skein(e)$.
\end{lemma}
\begin{proof}
  It is enough to consider the case where $\abs{\wt C_1}$ is
  connected.  First one can see $\wt C_2$ is uniquely determined by
  $\wt C_1$ and $e$, then by Lemma \ref{L:unique-eq}, $e'$ is uniquely
  determined by $e$.  In the case
  $\wt C_1:(\Sigma,Y,\rho)\to (\Sigma',\psi(Y),\rho\circ \psi^{-1})$ is the
  mapping cylinder of $\psi$ and $e=J_\vp\circ R(\Sigma,Z,Y,\wb\rho)$, then
  $\wt C_2$ has to be the mapping cylinder of $\psi$ from
  $(\Sigma,Z,\vp^{-1}.\wb\rho_{|Z})$ and $e'$ is determined by
  the relation $\wt C_1\circ J_\vp \circ R(\Sigma,Z,Y,\wb\rho)=J_{\vp\circ
    \psi^{-1}}\circ R(\Sigma',\psi(Z),\psi(Y),\wb\rho\circ \psi^{-1})\circ \wt C_2$ and the
  result follows because
  $\bp{\vp\circ \psi^{-1}}_*\circ r(\Sigma',\psi(Z),\psi(Y),\wb\rho\circ
  \psi^{-1})=\Skein(\psi)\circ{\vp}_*\circ
  r(\Sigma,Z,Y,\wb\rho)\circ\Skein(\psi^{-1})$.

  If $\wt C_1$ is an index $3$ handle, $e'$ is trivial and this
  follows since the map $\Skein(\wt C_1)=\F':\Skein(S^2,Y,\rho)\to\kk$
  factors through $\F':\Skein(S^2)\to\kk$ and is not sensible to the
  maps associated to change of the decoration.

  When $\wt C_1$ is an index $1$ or $2$ handle, we can identify the
  set of base points $Y$ in its source and target the result follows
  since $\Skein(\wt C_1)$ commute with the supression of a base point
  outside the attaching sphere and commute with $(g^{\delta_y})_*$ for
  any $(y,g)\in Y\times\Gr$.  This last point is clear if the point is
  outside the attaching sphere or if $k=2$.  For $k=1$ and $y=y_\pm$
  this is a consequence of the slidding of a degree $g$ colored strand
  of a red circle (see Figure \ref{fig:slide}).
\end{proof}
\begin{proof}[Proof of Theorem \ref{T:main}]
Let $\wt C$ be a $\Gr$-decorated cobordism with $\abs{\wt C}=C$.  Choose a
decomposition of $C$ as a product of generators of $\cob$. Insert identity
cylinders between generators and put some base points on the slicing
surfaces such that each generator is standard. Choose a lift of the
representation of $\wt C$ to this new set of base points.  This gives a
decomposition of $\wt C$ into $\Gr$-decorated generators.
$$\wt C=e_n\circ \wt C_{n}\circ e_{n-1}\circ \cdots\circ e_{1}\circ \wt C_1\circ e_0$$
where the $e_i$ are equivalences of $\Gr$-decorated surfaces and the
$\wt C_i$ are standard $\Gr$-decorated generators.

We first claim that
$\Skein(e_n)\circ\Skein(\wt C_n)\cdots \Skein(\wt C_1)\circ
\Skein(e_0)$ is independant of the choice of the base points and of
the lift of the representation of $\wt C$.  Indeed, suppose
$\wt C=e'_n\circ \wt C'_{n}\circ e'_{n-1}\circ \cdots\circ e'_{1}\circ
\wt C'_1\circ e'_0$ is an other decomposition. % Let
% $\wt C_{\le k}=e_k\circ \wt C_{k}\circ e_{k-1}\circ \cdots\circ
% e_{1}\circ \wt C_1\circ e_0$.  Let $e''_0=e_0\circ (e'_0)^{-1}$.
By Lemma \ref{L:unique-eq}, and by induction on $k=1,\ldots,n$ there
exists a unique equivalence of $\Gr$-decorated surfaces $e''_k$ such that
$e_k\circ\wt C_{k}=e''_k\circ(e'_k\circ \wt C'_{k}\circ
(e''_{k-1})^{-1})$.  In particular, the unicity of Lemma
\ref{L:unique-eq} applied to $\wt C$ implies that $e''_n=\Id$.  Then by
Lemma \ref{L:[eq,gen]} we have
$\Skein(\wt C_{k})=\Skein(e_k^{-1}\circ e''_k\circ e'_k)\circ
\Skein(\wt C'_k)\circ \Skein((e''_{k-1})^{-1})$.  Since $\Skein$ is
functorial on $\wt\cob_e$, this implies that
$$\Skein(e_k)\circ\Skein(\wt C_{k})=\Skein(e''_k)
\circ(\Skein(e'_k)\circ \Skein(\wt C'_k))\circ
\Skein(e''_{k-1})^{-1},$$ and by composition,
$\bp{\prod_k\Skein(e_k)\circ\Skein(\wt
C_{k})}\circ\Skein(e_0)=\bp{\prod_k\Skein(e'_k)\circ\Skein(\wt
C'_{k})}\circ\Skein(e'_0)$.

If $f:\wt{C_1}\tto\sim \wt{C_2}$ with $h_i:\abs{w_i}\to C_i$ be the
decomposition of $C_i$ as the realization of a word $w_i$ in the
generators of $\cobnc$, then
$h_2^{-1}\circ f\circ h_1:\abs{w_1} \tto\sim \abs{w_2}$ can be
realized as a sequence of relation in $\cobnc$.  Suppose that the
relation is a handle cancellation between $C_i$ and $C_{i+1}$.  Using
mapping cylinder of isotopies, and Juhász relation (1) and (2), we can
assumme that the relative positions of the framed attaching spheres of
handle generators allow to choose base points and representations such
that $e_i$ is the identity and $\wt C_{i+1},\wt C_i$ are standard
$\Gr$-decorated handle attachments.  Then by Propositions
\ref{P:hm-12-cancel},\ref{P:hm-23-cancel},
$\Skein(\wt C_{i+1})\circ\Skein(\wt C_{i})$ is the map associated to
the equivalence of $\Gr$-decorated surfaces
$\wt C_{i+1}\circ\wt C_{i}$.

Hence $\Skein(\wt C)$ only depends on the equivalence class of
the $\Gr$-decorated cobordism $\wt C$.
\end{proof}
\section{Relation with the modified Turaev-Viro invariant}
Let $\cat$ be a $\Gr$-finite spherical category relative to $\XX$ as in Definition \ref{def:olddef}.
\subsection{An invariant of TV-manifolds from the skein $\Gr$-HQFT}
Let $\Gamma$ be a finite graph, with a non empty set of vertices
$\Gamma_0$, and which is embedded in a closed connected 3-manifold
$M$.  Let $\rho\in\Rep_\Gr(M,\Gamma_0)$.  We call the triple
$(M,\Gamma,\rho)$ a TV-manifold.  In \cite{GPT09,GP13}, a
diffeomorphism invariant of TV-manifold is defined from $\cat$.

Otherwise, using the graded skein $\Gr$-HQFT, we define an invariant of
TV-manifold as follows:

First, if $g\in\Gr\setminus\XX$, define the formal linear combination of object
$$\mb_{g}=\sum_{V_i}\mb(V_i)V_i,$$
where the sum is over a set of representatives of isomorphism classes
of simple objects of $\cat_{g}$.  Then, whenever
$g_1,g_2,g_1g_2\not\in\XX$, we have the following $\Proj$-skein
relations in an annulus:
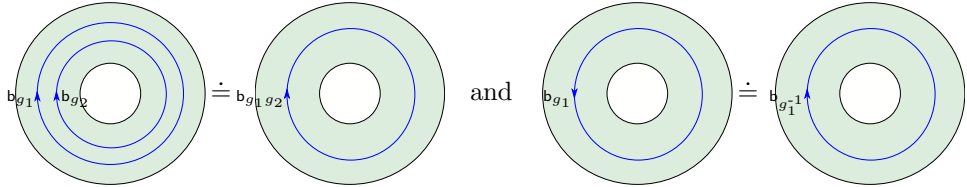
\begin{equation}
  \label{eq:gr-b}
  \epsh{fig33a.pdf}{16ex}
  \putw{11}{49}{$\ms{\mb_{g_1}}$}
  \pute{24}{49}{$\ms{\mb_{g_2}}$}
  \ \skeq\ \ \ 
  \epsh{fig33b.pdf}{16ex}
  \putw{16}{49}{$\ms{\mb_{g_1g_2}}$}
  \et
  \epsh{fig33c.pdf}{16ex}
  \putw{16}{49}{$\ms{\mb_{g_1}}$}
  \ \skeq\ \ 
  \epsh{fig33b.pdf}{16ex}
  \putw{16}{49}{$\ms{\mb_{g_1^{\text-1}}}$}
\end{equation}
% \[\epsh{fig33a.pdf}{22ex}
%   \putw{11}{49}{$\ms{\mb_{g_1}}$}
%   \pute{24}{49}{$\ms{\mb_{g_2}}$}
%   \ \skeq\ \ 
%   \epsh{fig33b.pdf}{22ex}
%   \putw{15}{49}{$\ms{\mb_{g_1g_2}}$}
% \]

% in the Grothendieck ring of $\cat$, whenever
% $g_1,g_2,g_1g_2\not\in\XX$,
%$$\mb_{g_1}\otimes\mb_{g_2}=\mb_{g_1g_2}.$$
This gives a way to extend the definition of $\mb_g$ for $g\in\XX$ by
$\mb_{h}\otimes\mb_{h^{-1}g}$ which is independant of
$h\in\Gr\setminus\bp{\XX\cup g\XX}$.

Next, let $(M,\Gamma,\rho)$ be a TV-manifold and $\Sigma$ be the
boundary of a closed tubular neighborhood
$N(\Gamma)\simeq \Gamma\times D^2$ of $\Gamma$ in $M$.  We fix the set
of base points $Y=\Gamma_0\times\bs1$ on $\Sigma$.  Then, by
restricting the representation $\rho$, we get a $\Gr$-decorated
surface with a cobordism to the empty set:
$$(M\setminus N(\Gamma),\rho):(\Sigma,Y,\rho)\to\wt\emptyset.$$
For each edge $e$ of $\Gamma$ equipped with an arbitrary orientation,
we get an element of $\Gr$ defined by $g_e=\rho(e)$ and we get a
simple curve $m_e\subset\Sigma$ which is an oriented meridian of $e$
that we color with $\mb_{g_e}$. This gives a skein
$\Gamma_\mb\in\Skein(\Sigma,Y,\rho)$.  Remark that by Equation
\eqref{eq:gr-b}, changing the orientation of the skein does not affect
the class of $\Gamma_\mb$ in $\Skein(\Sigma,Y,\rho)$.
\begin{theorem}\label{T:TV=Skein}
  Let $(M,\Gamma,\rho)$ be a TV-manifold, and
  $\Gamma_\mb\in\Skein(\Sigma,Y,\rho)$ be as above. Then we have
  $$\TV(M,\Gamma,\rho)=\Skein(M\setminus N(\Gamma))(\Gamma_\mb).$$
\end{theorem}
The proof of this theorem is given in the next subsection.
\subsection{The modified Turaev-Viro invariant}
In \cite{GPT11,GP13}, an invariant of TV-manifolds is defined from
$\cat$, using an extra choice of a basic data: A basic data for $\cat$
consists in
\begin{enumerate}
\item a family $\bs{I_g}_{g\in\Gr}$ of finite sets and for each
  $g\in\Gr\setminus\XX$, $\bs{V_i}_{i\in I_g}$ is a family of
  representatives of isomorphic classes of simple objects of $\cat_g$ ;
\item an involution $i\mapsto i^\star$ on
  $I=\bigcup_{\Gr\setminus\XX}I_g$ and a family of isomorphisms in
  $\cat$ $\bs{w_i:V_i\to (V_{i^\star})^*}_{i\in I}$ which satisfies:
  for any $i\in I$, $w_{i^\star}= (w_i)^*\circ\phi_{V_{i^\star}}$,
  where $\phi_V:V\to V^{**}$ is the pivotal structure.
\end{enumerate}
Let us denote by
$H_{i_1,i_2,\ldots,i_n}^{j_1,j_2,\ldots,j_n}=\Hom_\cat(V_{i_1}\otimes
V_{i_2}\otimes \cdots\otimes V_{i_n},V_{j_1}\otimes
V_{j_2}\otimes\ldots\otimes V_{j_n})$.  Then using the pivotal
structure of $\cat$ and the isomorphisms $\bs{w_i}_{i\in I}$, there
are canonical isomorphisms compatible with the pairings given by the
modified trace:
\begin{align*}
  H_{i,j,k}\simeq H_{k,i,j}\simeq H_{j,k,i}&\simeq H^{k^\star,j^\star,i^\star}\simeq H^{j^\star,i^\star,k^\star}\simeq H^{i^\star,k^\star,j^\star}
  \\
  \simeq H_{i,j}^{k^\star}\simeq H_{k,i}^{j^\star}\simeq H_{j,k}^{i^\star}&\simeq H_{i}^{k^\star,j^\star}\simeq H_{k}^{j^\star,i^\star}\simeq H_{j}^{i^\star,k^\star}
\end{align*}
Identifying all these spaces through this inductive system of
isomorphisms, allow to define the multiplicity space $H(i,j,k)$ which
only depends on $(i,j,k)$ up to cyclic permutation.  The pairing given by the
modified trace induces a non degenerate pairing $H(i,j,k)\otimes H(k^\star,j^\star,i^\star)\to\kk$

\newcommand{\T}{{\mathcal T}}
\newcommand{\St}{{\operatorname{St}}}
\newcommand{\cntr}{{\operatorname{cntr}}}
\newcommand{\YY}{{\T^1_\Gamma}}

The invariant $\TV(M,\Gamma,\rho)$ is computed using an
$H$-triangulation of the TV-manifold $(M,\Gamma,\rho)$: By a
triangulation $\T$ of $M$, we mean a smooth $\Delta$-complex structure
on $M$ (as in \cite{Hat02}).  Loosely speaking, a $\Delta$-complex
structure is a quotient space of a collection of disjoint simplices
obtained by identifying certain of their faces.  In particular, the
interior of the simplices of $\T$ are embedded in $M$ but their faces
are not necessarily distincts and two different simplices might meet
on several faces.  We say that $\T$ is \emph{quasi-regular} if any
simplex of $\T$ is embedded in $M$.  This is equivalent to requiring
that the two endpoints of any edge of $\T$ are distinct vertices of
$\T$.  Let $\YY$ be a set of unoriented edges of $\T$.  Let
$\Phi:\vec\T^1\to\Gr$ be a map on oriented edges. %  satisfying
% $\Phi(-e)=\Phi(e)^{-1}$. 
We say that $\Phi$ is admissible if % it takes
% values in
$\Phi(\vec\T^1)\subset\Gr\setminus\XX$.  Then we say that $(\T,\YY,\Phi)$ is an
$H$-triangulation of $(M,\Gamma,\rho)$ if
\begin{enumerate}
\item $\T$ is a quasi-regular triangulation of $M$,
\item the union of the edges of $\YY$ is $\Gamma$,
\item the set $\Gamma_0$ of vertices of $\Gamma$ is a subset of the
  set $\T^0$ of vertices of $\T$,
\item the graph $\Gamma$ is Hamiltonian, \ie every vertex of $\T^0$ is
  adjacent to an edge of $\YY$,
\item the function $\Phi$ extend to a representation of
  $\Rep_\Gr(M,\T^0)$ whose restriction to $\Rep_\Gr(M,\Gamma^0)$ is
  equal to $\rho$.
\end{enumerate}
In particular, the last point implies that  $\Phi(\wa{AB})\Phi(\wa{BA})=1$ if $\wa{AB}$ and $\wa{BA}$ in $\vec\T^1$ are the same edge with opposite orientation and $\Phi(\wa{AB})\Phi(\wa{BC})\Phi(\wa{CA})=1$ for any triangle $ABC$ in $\T^2$ with sides $\wa{AB},\wa{BC},\wa{CA}$.

If $(\T,\YY,\Phi)$ is an $H$-triangulation of $(M,\Gamma,\rho)$, a
state $\vp$ of $\Phi$ is a map $\vp:\vec\T^1\to I$ from the oriented
edges of $\T$ to $I$ satisfying $\vp(e)\in I_{\Phi(e)}$ and
$\vp(-e)=\vp(e)^{\star}$.  The set $\St(\Phi)$ of states of $\Phi$ is
finite.

If $\vec f=ABC$ is an oriented face of the triangulation, we associate
to it the multiplicity module $H(\vec f,\vp)=H(i,j,k)$ where
$i=\vp(\wa{AB})$, $j=\vp(\wa{BC})$, $k=\vp(\wa{CA})$.  The copairing
of $(f,\vp)$ is the element
$\Omega_{f,\vp}\in H(i,j,k)\otimes H(k^\star,j^\star,i^\star)$ dual
to the pairing from the m-trace.  Remark that
$H(k^\star,j^\star,i^\star)=H(-\vec f,\vp)$.

If $T=ABCD$ is a 3-simplex of $\T$ positively oriented in $M$, its
four faces $\vec f_1=BCD,\vec f_2=ADC,\vec f_3=ABD,\vec f_4=ACB$ inherit an
orientation by the boundary rule.  We associate to $(T,\vp)$ the
element $\abs{T}_\vp$ of
$\bigotimes_{i=1}^4H(\vec f_i,\vp)^*=\bp{\bigotimes_{i=1}^4H(\vec
  f_i,\vp)}^*$ given by the image of the following graph $G$ drawn in the
boundary of $T$ which topologically is an oriented 2-sphere:
$$\abs{T}_\vp:(x_1\otimes x_2\otimes x_3\otimes x_4)\mapsto\F'\bp{
  \epsh{fig21y.pdf}{28ex}\putc{10}{80}{${\circlearrowleft}$}
  \putsw{10}{11}{$\ms{A}$}
  \putn{30}{64}{$\ms{B}$}
  \putse{56}{34}{$\ms{C}$}
  \pute{85}{51}{$\ms{D}$}
  \putc{34}{39}{${x_4}$}
  \putc{52.5}{17}{${x_2}$}
  \putc{38.5}{86}{${x_3}$} \putc{52.5}{61}{${x_1}$}}$$ where the
colors of the edges are given by
$[\vec e\cap G]_\cat=V_{\vp(\vec e)}$.  This scalar is independant of
the choice of the order of the four vertices $A,B,C,D$ of $T$ up to
even permutation.

The modified Turaev-Viro invariant of $(M,\Gamma,\rho)$ is the
following scalar $\TV(M,\Gamma,\rho)=\TV(\T,\YY,\Phi)$:
\begin{equation}
  \label{eq:mTVss}
 % \TV(\T,\YY,\Phi)=
\sum_{\vp\in\St(\Phi)} \bp{\prod_{e\in\YY}\mb(V_{\vp(\vec e)})}\times
\bp{\prod_{e\in\T^1\setminus\YY}\qd(V_{\vp(\vec e)})}\times
\bp{\bigotimes_{T\in\T^3}\abs
  T_\vp}\bp{\bigotimes_{f\in\T^2}\Omega_{\vec f,\vp}}, 
\end{equation}
where $\T^1$ is the set of edges of $\T$, $\T^2$ is its set of
2-simplices and $\T^3$ is its set of 3-simplices; the elements
$\vec e$ and $\vec f$ in the sum represent the edge $e$ and the face
$f$ equipped with any orientation.
The following is \cite[Theorems 22,24,25]{GPT11} and \cite[Theorems 14,15,16]{GP13}
\begin{theorem}
  Any TV-manifold $(M,\Gamma,\rho)$ admits an H-triangulation
  $(\T,\YY,\Phi)$ and the scalar $\TV(\T,\YY,\Phi)$ given by Equation
  \eqref{eq:mTVss} only depends of the diffeomorphism class of
  $(M,\Gamma,\rho)$.
\end{theorem}
% , $\abs T_\vp$ is the 6j-symbol described below which
% is an element of the tensor product of four multiplicity spaces and
% $\cntr$ is the tensor product of the pairings associated to the faces
% of $\T$ colored by $\vp$.

In the following we need to orient coherently the simplices of $\T$.
Then remark that if $T^-=ABCD$ is a 3-simplex of $\T$ negatively
oriented (\ie $ACBD$ is positively oriented in $M$), and $\vp$ is a
state, then $\abs {T^-}_\vp$ is (up to permutation of the factors) an
element of the dual of $\bigotimes_{i=1}^4H(\vec f_i,\vp)$ where
$\vec f_1=BDC,\vec f_2=ACD,\vec f_3=ADB,\vec f_4=ABC$ given by the
following graph $G$:
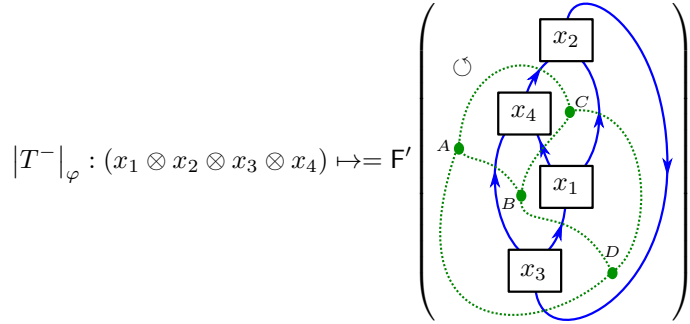
$$\abs {T^-}_\vp:(x_1\otimes x_2\otimes x_3\otimes x_4)\mapsto=\F'\bp{\epsh{fig21z.pdf}{28ex}\putc{10}{80}{${\circlearrowleft}$}
\putw{5}{55}{$\ms{A}$}
\putw{33}{37}{$\ms{B}$}
\pute{58}{69}{$\ms{C}$}
\putn{74}{20}{$\ms{D}$}
\putc{36.5}{66}{${x_4}$}
\putc{54.5}{89}{${x_2}$}
\putc{40.5}{16}{${x_3}$}
\putc{54.5}{43}{${x_1}$}}
$$
where the colors of the edges are given by $[\vec e\cap G]_\cat=V_{\vp(\vec e)}$.

Let us fix a total order on the set $\T^0$ of vertices of $\T$.  This
induces a canonical orientation of any simplex of $\T$, and in
particular of the edges of $\T$.  Let $n_3=\card(\T_3)$
be the number of tetraedra in $\T$.

Let $N(\T_1)\subset M$ be a open tubular neighborhood of the
1-skeleton of $\T$.  Let $\Sigma'$ be its boundary and $Y'$ be a set
of base points, one near each vertex of $\T$.   We can factor $(M\setminus N(\Gamma),\rho):(\Sigma,Y,\rho)\to\emptyset$ as  
$$(\Sigma,Y,\rho)\tto{\wt C_0}(\Sigma,Y',\Phi)\tto{\wt C_1}(\Sigma',Y',\Phi)
\tto{\wt C_2}((S^2)^{\sqcup n_3},Y'')\tto{\wt C_3}\emptyset.$$ In the
above composition, the first $\Gr$-decorated cobordism $\wt C_0$ is a
restriction cylinder obtained by forgetting base points of
$C_{\Sigma,Y',\Phi}$.  The second $\Gr$-decorated cobordism $\wt C_1$ is
obtained by gluing on $\Sigma$ one index 1 handle for each edge
$e\in\T^1\setminus\T^1_\Gamma$.  The index 1 handle is a tubular
neighborhood of the edge $e$, oriented with the canonical orientation
above.  The third $\Gr$-decorated cobordism $\wt C_2$ is obtained by gluing
on $\Sigma'$ an index 2 handle for each face of the triangulation.
Its target is the $n_3$-components surface made of spheres parallel
with the boundaries of the tetraedra of $\T$ with 4 base points near
the vertices of each tetraedron. The last $\Gr$-decorated cobordism
$\wt C_3$ is filling these $n_3$ spheres with index 3 handles (\ie
3-balls).

We now apply the $\Gr$-HQFT $\Skein$ to this composition of $\Gr$-decorated
cobordisms to compute the image of $\Gamma_\mb$.

We first use the skein equivalences of \eqref{eq:gr-b} to modify
$\Gamma_\mb$ to a skein equivalent graph $\Gamma'_\mb$ in
$\Skein(\Sigma,Y,\rho)$ which consist for each edge $e$ of
$\T^1_\Gamma$ to an oriented meridian colored with
$\mb_{\Phi(\vec e)}$ (here we equip $e$ with its canonical orientation $\vec e$
using the ordering of vertices).  The skein $\Gamma'_\mb$ also
represent an element of $\Skein(\Sigma,Y',\Phi)$ which is the image of
$\Gamma_\mb$ by $\Skein(\wt C_0)$.

The map $\Skein(\wt C_1)$ is adding a red circle at the belt sphere of
the index 1 handle associated to $e$, colored with $\Phi(e)$.  Since
$\Phi(e)\notin\XX$ the chromatic move to turn this circle blue just
consists in replacing this red circle with a blue circle colored with
the modified Kirby color $\sum_{i\in I_{\Phi(e)}}\qd(V_i)V_i$.
Expanding multilinearly the formal $\mb$ colors and Kirby colors, we
get that
$$\Skein(\wt
C_1)(\Gamma'_\mb)=\sum_{\vp\in\St(\Phi)}\bp{\prod_{e\in\YY}\mb(V_{\vp(\vec
    e)})}\times \bp{\prod_{e\in\T^1\setminus\YY}\qd(V_{\vp(\vec
    e)})}\times\Gamma_\vp,$$ where $\Gamma_\vp$ is the graph formed by
the oriented meridian of the edges of $\T$ colored with
$V_{\vp(\any)}$ in $\Skein(\Sigma,Y',\Phi)$.

The map $\Skein(\wt C_2)$ is cutting the surface along each face $\vec f_i$, putting in both side the dual bases of $H(\vec f_i,\vp)$.  Finally $\Skein(\wt C_3)$ evaluate these tetraedral graphs with $\F'$.  Hence $\Skein(M\setminus N(\Gamma))(\Gamma_\mb)=\Skein(\wt C_3)\Skein(\wt C_2)\Skein(\wt C_1)\Skein(\wt C_0)(\Gamma_\mb)$ is given by the state sum of Equation \eqref{eq:mTVss}, which proves Theorem \ref{T:TV=Skein}.
%\singlespacing
\footnotesize
\bibliographystyle{halpha-modified}% modified to use the bibkey as bibitem
%\bibliographystyle{abstract}
%\bibliography{References,back-ref,Last-ref}
\bibliography{References-Book}

\end{document}